\documentclass[reqno,11pt]{amsart}
\usepackage[left=1in,right=1in,top=1in,bottom=1in]{geometry}
\usepackage{amsfonts,amssymb,amsmath,graphicx}
\usepackage{multirow,array,booktabs}
\usepackage{accents}

\usepackage{braket}
\usepackage{leftidx}
\usepackage{empheq}
\usepackage{mathtools}  

\usepackage[pdftex,plainpages=false,pdfpagelabels,hypertexnames=false,
       colorlinks=true,pdfstartview=FitV,
       linkcolor=blue,citecolor=red,
       urlcolor=blue]{hyperref}
\usepackage{thumbpdf}

\usepackage{amsaddr}

\usepackage{caption}
\usepackage{eucal}

\newtheorem{thm}{Theorem}[section]

\newtheorem{defn}[thm]{Definition}

\newtheorem{rem}[thm]{Remark}

\numberwithin{thm}{section}
\numberwithin{table}{section}
\numberwithin{equation}{section}
\numberwithin{figure}{section}

\newcommand{\magenta}[1]{{\leavevmode\color{black}#1}}

\newcommand{\INT}{\mathbb{Z}}

\newcommand{\BC}{\ensuremath{{\tt BC}}}
\newcommand{\p}{\ensuremath{{\tt p}}}
\newcommand{\DD}{\ensuremath{{\tt DD}}}
\newcommand{\NN}{\ensuremath{{\tt NN}}}
\newcommand{\DN}{\ensuremath{{\tt DN}}}
\newcommand{\ND}{\ensuremath{{\tt ND}}}
\newcommand{\cpld}{{\tt cpld}}
\newcommand{\LtL}{{\tt LtL}}
\newcommand{\LtN}{{\tt LtN}}
\newcommand{\lleft}{{\tt left}}
\newcommand{\rright}{{\tt right}}
\newcommand{\Trp}{{\tt Trp}}

\renewcommand{\hat}{\widehat}
\newcommand{\Chat}{\hat{C}}
\newcommand{\Ohat}{\hat{\Omega}}
\newcommand{\ahat}{\hat{a}}
\newcommand{\bhat}{\hat{b}}

\newcommand{\M}{\ensuremath{\mathcal{M}}}
\newcommand{\N}{\mathcal{N}}

\newcommand{\K}{\mathcal{K}}
\newcommand{\E}{\mathcal{E}}

\newcommand{\bO}{\ensuremath{\mathcal{O}}}
\renewcommand{\a}{\ensuremath{{\tt a}}}

\newcommand{\Lo}{\accentset{\circ}{L}}
\renewcommand{\L}{\ensuremath{{\tt L}}}
 \newcommand{\NL}{\ensuremath{{\tt NL}}}

\newcommand{\xp}{x^\prime}

\newcommand*\diff{\mathop{}\!\mathrm{d}}

\newcommand{\bff}{\mathbf{f}}
\newcommand{\bu}{\mathbf{u}}
\newcommand{\bv}{\mathbf{v}}
\newcommand{\bn}{\mathbf{n}}

\usepackage{makecell}

\newcommand{\uline}[1]{\multicolumn{#1}{c}{\rule{2cm}{0.4pt}}}

\usepackage{subcaption}
\graphicspath{}

\usepackage{tikz}
\usetikzlibrary{arrows}
\usetikzlibrary{matrix,backgrounds}
\pgfdeclarelayer{myback}
\pgfsetlayers{myback,background,main}
\tikzset{mycolor/.style = {line width=1bp,color=#1}}
\tikzset{myfillcolor/.style = {draw,fill=#1,#1,rounded corners}}

\usepackage{xcolor}
\definecolor{azure}{rgb}{0.0, 0.5, 1.0}
\definecolor{asparagus}{rgb}{0.53, 0.66, 0.42}
\definecolor{cadetgrey}{rgb}{0.57, 0.64, 0.69}
\definecolor{awesome}{rgb}{1.0, 0.13, 0.32}

\NewDocumentCommand{\fhighlight}{O{blue!40} m m}{%
\draw[myfillcolor=#1] (#2.north west)rectangle (#3.south east);
}
\NewDocumentCommand{\fhighlightL}{O{blue!40} m m}{%
\draw[myfillcolor=#1] (#2.south west)rectangle ([xshift=0.5em]#3.south east);
}

\begin{document}

\title[Coupling of Local and Nonlocal Problems]
{Coupling of Local and Nonlocal Problems \\ Using Local Boundary Conditions}
\author [Burak Aksoylu, Fatih Celiker, and Patrick Diehl]{Burak Aksoylu}
\address{Texas A\&M University-San Antonio \\
Department of Computational, Engineering, and Mathematical Sciences \\
San Antonio, TX 78224, USA}
\email{baksoylu@tamusa.edu}
\thanks{Corresponding author: Burak Aksoylu}

\author{Fatih Celiker}
\address{Wayne State University\\ Department of Mathematics\\
Detroit, MI 48202, USA}
\email{celiker@wayne.edu}

\author{Patrick Diehl}
\address{Los Alamos National Laboratory \\ Los Alamos, NM 87545, USA}
\email{diehlpk@lanl.gov}
\date{\today}

\begin{abstract}
We present a novel coupling method for local and nonlocal diffusion
problems in 1D.  Unlike other methods, our coupling method
exclusively uses local boundary conditions.  This is possible because
our nonlocal operators enforce them by construction.  Leveraging this
advantageous property, we construct a seamless coupling that is
remarkably natural.
The utilization of local boundary conditions allows for the transfer
of well-established numerical methods from local problems to nonlocal
ones.  Our local-to-nonlocal coupling method is inspired by the domain
decomposition method, which we would like to transfer to nonlocal
problems.  The main result of our study is the construction of a
local-to-nonlocal coupling method with a quantifiable $\bO(h)$
convergence that holds for an arbitrary solution.
For discretization of the local and nonlocal problems, the finite
element method and the Galerkin projection are employed, respectively.
We verify our convergence rate with extensive numerical experiments.
\end{abstract}

\maketitle

\section{Introduction}

Coupling of local and nonlocal (NL) problems offers a way to combine
the computational efficiency of local models with the capabilities of
NL models, in particular capturing discontinuities.  Unlike existing
methods, we present a coupling method which exclusively uses local
boundary conditions (BCs).  This is possible because our NL operators
enforce local BC. We leverage this advantageous property as the main
feature of our construction to devise a seamless coupling that is
remarkably natural.

We have been advocating that this feature allows for the transfer of
well-established numerical methods developed for local problems to
NL problems.  The main method we would like to transfer to NL problems
is the domain decomposition method (DDM) and our local-to-nonlocal
(LtN) coupling method is inspired by it.  The existing LtN coupling
methods do not exclusively use DDM techniques, probably due to the
lack of ability to satisfy local BC.

The main result of our study is the construction of a LtN coupling method
with a quantifiable $\bO(h)$ convergence that holds for an arbitrary
solution.  The hallmark features (HFs) of our coupling method are the
following:
\begin{enumerate}
\item[{\bf (HF1)}] The coupled equation at the interface is an approximation
  of the equation in the bulk so that the coupling is seamless.  Here
  seamless means that the discretized interface equation becomes a
  discretized bulk equation.
\item[{\bf (HF2)}] For the equilibrium of force, the type of the BC on both sides
  of the interface should be Neumann.  Hence, the operators on either
  side of the interface should produce a derivative operator.
\end{enumerate}
These features are verified by Taylor expansions.  Due to nonlocality
of the operator and the fact that weak formulation is used for
discretization, the resulting stiffness matrix is not sparse. Since
the number of terms appearing is large, we use symbolic computation to
perform the Taylor expansions.  Once these features are in place, we
numerically establish that the rate of convergence of the coupling
method is $1$ with linear finite element discretization.

\subsection{Existing Literature}

Coupling of local and NL problems has received great attention from
the engineering community.  Various approaches to coupling have been
taken.  Similar to our method, one class of methods utilizes an
interface without an overlap region. The interface conditions are
formulated by matching displacements
\cite{bie2018coupling,madenci2018coupling} or stresses
\cite{ongaro_overall_2021,silling2020Couplingstresses,silling15Variablehorizonperidynamicmedium}.

The other class of methods employs an overlap region in which the
local and NL models coexist.  In these approaches, the NL region is
extended by one or more horizon lengths, creating a transition zone
that overlaps with the local domains.  Coupling within the overlap
region is achieved either by matching displacements
\cite{kilic2010coupling,liu2012coupling,ni2019_coupling_with_mirco,sun2019superposition,zaccariottoEtAl2018_coupling_most_cited}
or stresses \cite{NIKPAYAM2019308}. To align with the coexisting local
and NL descriptions, additional constraints are introduced in the
overlap region
\cite{diehl2022_coupling,selesonEtAl2022_coupling_review}.  A common
drawback of this approach is that the NL domain must be artificially
enlarged, which may increase the computational cost and complicate the
physical interpretation of the coupling region.
We do not go into details of these approaches and refer the reader to the comprehensive surveys in \cite{diehl2022_coupling,selesonEtAl2022_coupling_review,madenci_coupling_2014} for further details.

One desirable property in LtN coupling is to pass the patch tests
\cite{diehl2022_coupling,selesonEtAl2022_coupling_review,galvanettoEtAl2016_coupling_most_cited2,jiangGlusa2024,zaccariottoEtAl2018_coupling_most_cited}. These
tests were designed for NL operators with NL BCs.  Since our governing
operators enforce local BC, patch tests are not applicable to our
operators mainly due to the compatibility conditions; see
Sec.~\ref{subsec:compa_cond}.


The rest of the paper is organized as follows.
In Sec.~\ref{sec:comparison}, we first provide a comparison of solving boundary value problems (BVPs) employing integral equations (IEs) versus partial differential equations (PDEs).  Then, we explain how the fact that our operator is Fredholm of the second kind plays a critical role in satisfying BCs. 
In Sec.~\ref{sec:govering_op}, we explain the construction of the governing operator. 
In Sec.~\ref{sec:DD}, we show the details of how to view the domain decomposition of the local problem as a local-to-local (LtL) coupling and how it inspired us to construct the LtN coupling. 
In Sec.~\ref{sec:sparsity}, we present the sparsity structure of the stiffness matrix resulting from the Galerkin projection discretization of the governing operator with linear basis functions.
In Sec.~\ref{sec:BC_weak}, we explain how a Neumann BC is enforced in the weak formulation.  We show that the sum of the corresponding rows collectively produces a Neumann approximation, which is a manifestation of nonlocality in weak form.
In Sec.~\ref{sec:Neu_BC_Taylor}, it is shown that the discretized Neumann operator leads to a derivative condition.

In Sec.~\ref{sec:interface_equ}, we present the equation at the interface resulting from the weak form and its relation to the strong form.  The success of our coupling method hinges on the fact that the discretized coupled operator is an approximation of the local operator at the interface.
The treatment of BCs in IEs is fundamentally different from that in PDEs because of the presence of compatibility conditions between the forcing function and the solution.  In Sec.~\ref{sec:bdry_treatment_compa_cond},  we explain the compatibility conditions in detail.
In Sec.~\ref{sec:rectification}, we introduce the rectification process which is necessary for the solution of a scaled NL problem to satisfy the BC.
Numerical experiments are presented in Sec.~\ref{sec:numerical_experiments}.
We conclude in Sec.~\ref{sec:conclusion}.

\section{Solving Boundary Value Problems with Integral Equations} \label{sec:comparison}

\begin{table}[t]
\centering
\begin{tabular}{
 >{\centering\arraybackslash}p{1.75cm}
 >{\centering\arraybackslash}p{1.5cm}
 >{\centering\arraybackslash}p{3.0cm}
 >{\centering\arraybackslash}p{1.5cm}
 >{\centering\arraybackslash}p{2.75cm}
 >{\centering\arraybackslash}p{2.75cm}
}
\toprule
\textcolor{black}{\textbf{Equation Type}} &
\textcolor{black}{\textbf{BC Type}} &
\textcolor{black}{\textbf{Governing Operator}} &
\textcolor{black}{\textbf{Rhs}} &
\textcolor{black}{\textbf{Boundary Value}} &
\textcolor{black}{\textbf{Compatibility Conditions}} \\
\midrule
IE &
\makecell{user \\ defined} &
\makecell{determined by \\ BC type} &
\makecell{user \\ defined} &
\makecell{determined \\by rhs} &
required \\
\midrule
PDE &
\makecell{user \\ defined} &
\makecell{determined\\independently\\from BC type} &
\makecell{user \\ defined} &
\makecell{user \\ defined} &
not required \\
\bottomrule
\\
\end{tabular}
\caption{Comparison of solving a boundary value problem with an integral equation versus a partial differential equation}
\label{tab:comparison}
\end{table}

The IEs of interest have the ability to accommodate discontinuities in
the solution.  We prefer to use IEs because we are interested in
capturing cracks.  Furthermore, IEs provide the ability to rigorously
prove that BCs hold, thanks to uniform convergence guaranteed by the
Hilbert-Schmidt property.
However, PDEs do not have the ability to treat discontinuities because
they form singularities for the governing operator.  The solutions to
PDEs do not necessarily satisfy the BC rigorously unless a special
arrangement is made.  For instance, a series solution must satisfy the
Weirstrass $M$-test to guarantee uniform convergence see
\cite[Sec.~18.3.2]{greenberg1998_book}. Since this is not always the
case, the series solutions qualify only as formal solutions
\cite[p.~980]{greenberg1998_book}.

Let us describe how one solves a BVP using an IE.  First, the user has
to choose the type of the BC based on which the governing operator is
determined.  Once the right hand side function $f$ is provided, one
can solve the BVP.  However, the boundary value of the solution $u$ is
determined by that of $f$, which gives rise to compatibility
conditions.  We dedicated Sec.~\ref{subsec:compa_cond} to the
explanation of compatibility conditions.
In the local case, on the other hand, the governing operator is
independent from the type of BC.  Since the boundary value of $u$ is
independent from that of $f$, there are no compatibility conditions.
We summarize this comparison in Table~\ref{tab:comparison}.

\subsection{Problem Description}

We utilize the Poisson problem as the main local equation throughout the
paper.  We study the coupling of the Poisson equation with the NL
diffusion equation in 1D.

The domain is chosen as $\Omega = (a,b)$.  We consider three types of
BCs: pure Dirichlet, mixed, and pure Neumann and label them with $\BC
= \DD, \DN, \NN$, respectively.  The three local problems under
consideration stated on a single domain are
\begin{equation} \label{single_domain_prob}
\left\{ 
\begin{aligned} - E \Delta u & = f ~~ \text{in}~\Omega\\ u(a) & = \alpha \\ u(b) & = \beta, \end{aligned} 
\right. \quad
\left\{ 
\begin{aligned} - E \Delta u & = f ~~ \text{in}~\Omega\\ u(a) & = \alpha \\ E u'(b) & = \beta, \end{aligned} 
\right. \quad \text{and} \quad 
\left\{ 
\begin{aligned} - E \Delta u & = f ~~ \text{in}~\Omega\\ E u'(a) & = \alpha \\ E u'(b) & = \beta. \end{aligned} 
\right.
\end{equation}
We denote the local governing operators in \eqref{single_domain_prob}
with $-\Delta_\DD, -\Delta_\DN$, and $-\Delta_\NN$, respectively, when
a label for the BC is needed.  For ease of presentation, the modulus
elasticity is chosen to be $E=1$ so that the usage of $E$ can be
omitted.

\subsection{Boundary Treatment with the Nonlocal Operator} 
The governing operator 
\begin{equation} \label{dom_range}
\M_\BC: L^2(\Omega) \to L^2(\Omega)
\end{equation}
is a densely-defined, self-adjoint, linear (DSL) operator with a
purely discrete spectrum.  Furthermore, $\M_\BC$ is bounded.  By
exploiting its boundedness, we employ holomorphic functional calculus
for bounded operators.  The solutions to problems whose governing
operator is $\M_\BC$, for instance, abstract linear wave equations,
that use can easily be constructed through functional calculus
\cite{aksoyluBeyerCeliker2017_foundation,beyerAksoyluCeliker2016_unbounded}.
The solutions to problems\textemdash for instance, abstract linear
wave equations that use $\M_\BC$ as a governing operator\textemdash
can easily be constructed through functional calculus
\cite{aksoyluBeyerCeliker2017_foundation,beyerAksoyluCeliker2016_unbounded}.

The space $L^2(\Omega)$ has a major weakness: It altogether ignores
values of functions on the boundary $\Omega$.  At first sight, it may
seem odd to work with such a space for a study of BCs.  Later, we will
elaborate on how BCs are treated in our framework.  The choice of
space as $L^2(\Omega)$ comes from the motivation to treat
discontinuities such as cracks.  The peridynamic theory
\cite{silling2000PDfirstPaper} was developed to treat cracks.  The
construction of our operators was inspired by the peridynamic
theory.  Since functions in $L^2(\Omega)$ admit discontinuities, it
has the abilility to lead to a suitable function space.  Further
discussion on the choice of function spaces will be presented in
Remark~\ref{rem:PC}.

The governing operator $\M_\BC$ is defined as
\begin{equation} \label{govern_op}
\M_\BC u(x) = c u(x) - \int_\Omega k_\BC(x,x') u(x') \diff x',
\end{equation}
where 
\begin{equation} \label{convo_op}
\K_\BC u(x) := \int_\Omega k_\BC(x,x') u(x') \diff x',
\end{equation}
is a convolution operator with a square integrable kernel $k_\BC$.  We
arrive at the subtlety that constitutes the cornerstone of our
treatment of boundary values.  Due to square integrability of the
kernel, the operator $\K_\BC$ possesses the Hilbert-Schmidt property.
An operator that possesses the Hilbert-Schmidt property ``feels the
boundary'' of $\Omega$ because $\K_\BC$ has a smoothing
property that guarantees a continuous extension to the boundary \cite[Thm.~6 and 7]{aksoyluBeyerCeliker2017_foundation}: 
For $u \in L^2(\Omega)$,
\begin{equation} \label{cont_ext}
\K_\DD u(x) \in C^0(\overline{\Omega}) \quad \text{and} \quad 
\K_\NN u(x) \in C^1(\overline{\Omega}).
\end{equation}

Let $x_0$ be a boundary point.  Kernel functions are designed in such a way that
\begin{equation} \label{vanishing_kernel}
\lim_{x \to x_0} k_\DD(x,x') = 0 \quad \text{and} \quad
\lim_{x \to x_0} \frac{\partial k_\NN}{\partial x} (x,x') = 0.
\end{equation}
Using \eqref{vanishing_kernel}, one sees that 
\begin{align*}
\lim_{x \to x_0} \K_\DD u(x) & = \lim_{x \to x_0} \int_\Omega k_\DD(x,x') u(x') \diff x' \\
& = \int_\Omega \lim_{x \to x_0} k_\DD(x,x') u(x') \diff x' \\
& = 0 \\
\intertext{and} \\
\lim_{x \to x_0} \frac{\diff}{\diff x} \K_\NN u(x) & = 
\lim_{x \to x_0} \frac{\diff}{\diff x} \int_\Omega k_\NN(x,x') u(x') \diff x' & \\
& =  \int_\Omega \lim_{x \to x_0} \frac{\partial k_\NN}{\partial x}  (x,x') u(x') \diff x' \\
& = 0.
\end{align*}
The interchange of $\lim_{x \to x_0}$ or $\lim_{x \to x_0} \frac{\diff}{\diff x}$ with $\int_\Omega$ is due to the uniform convergence provided by the Hilbert-Schmidt property.  

When one expects to enforce a BC from the utilization of the governing equation
\begin{equation} \label{governing_equ}
\M_\BC u(x) = f(x),
\end{equation}
the standard practice is to start with an $f$ that has a boundary
limit.  Let us consider the case of $\BC=\DD$.  Rewriting
\eqref{governing_equ}, one can determine the boundary limit of $u$:
\begin{equation} \label{bridge_equ}
c \lim_{x \to x_0} u(x) =  \lim_{x \to x_0} f(x) + \lim_{x \to x_0} \K_\DD u(x).
\end{equation}
Since $\lim_{x \to x_0} \K_\DD u(x) = 0$, we arrive at \emph{compatibility
conditions} between $u$ and $f$:
\begin{equation} \label{compa_cond}
c \lim_{x \to x_0} u(x) =  \lim_{x \to x_0} f(x).
\end{equation}
Once the boundary value $\lim_{x \to x_0} f(x)$ is provided, the
solution $u$ is forced to satisfy
$$\lim_{x \to x_0} u(x) =  1/c \lim_{x \to x_0} f(x).$$

We call \eqref{bridge_equ} \emph{the bridge equation}.  We want to
shed light on this critical concept.  The bridge equation does not
magically force both limits to exist for arbitrary $L^2(\Omega)$
functions.  Instead, it provides a connection between the
existence of boundary limits of $f$ and $u$ in the following way:
\begin{itemize}
\item {\bf Admissible Case:} When an $f$ with a boundary limit is
  chosen, $u$ is forced to have a boundary limit, and they must
  satisfy the compatibility conditions \eqref{compa_cond}. \\

\item {\bf Pathological Case:} When an $f$ without a boundary limit is
  chosen, $u$ is forced not to have a boundary limit.
\end{itemize}

Remarkably, the scenario that starts with a chosen $f$ and ends with $u$ can also
be reversed. One can rewrite \eqref{bridge_equ} and obtain the reverse
bridge equation:
\begin{equation*} 
\lim_{x \to x_0} f(x) = c \lim_{x \to x_0} u(x) - \lim_{x \to x_0} \K_\DD u(x).
\end{equation*}
For coupling scenarios, since $f$ is a given, we employ the bridge
equation \eqref{bridge_equ}.  The reversibility in the bridge equation
is a direct consequence of having a governing operator of Fredholm of
the second kind.

We conclude with a summary of the boundary treatment.  As governing
operator, we employ the DSL operator $\M_\BC: L^2(\Omega) \to
L^2(\Omega)$ with a purely discrete spectrum given in
\eqref{govern_op}.  Furthermore, $\M_\BC$ is a Fredholm operator of
the second kind with a square integrable convolution kernel $k_\BC$.
Under these assumptions on $\M_\BC$, we have proved the following
result about the existence of boundary limits:\footnote{To avoid
cluttering in the theorem statement, we did not include the mixed
cases $\BC=\DN$ and $\BC=\ND$, which easily follow.}
\begin{thm}
Consider the problem 
\begin{equation*} 
\M_\BC u(x) = f(x).
\end{equation*}
Let $x_0$ be a boundary point of $\Omega$. Then, 
\begin{eqnarray*}
\BC = \DD &:& \lim_{x \to x_0} f(x) ~\text{exists if and only if} ~\lim_{x \to x_0} u(x) ~\text{exists,} \\
\BC = \NN &:& \lim_{x \to x_0} f'(x) ~\text{exists if and only if} ~\lim_{x \to x_0} u'(x) ~\text{exists.}
\end{eqnarray*}
\end{thm}
\begin{rem} \label{rem:no_free_lunch_thm}
Recall that the space $L^2(\Omega)$ is oblivious to values of functions on the boundary of $\Omega$.  When an operator whose domain is $L^2(\Omega)$ is required to enforce BCs, this becomes possible thanks to the Fredholm of the second kind property of the operator, however,  at the cost of compatibility conditions.  See Sec.~\ref{subsec:compa_cond}.
\end{rem}

\begin{rem} \label{rem:PC}
Since capturing cracks is the most relevant physical application, a
practical function space choice would be the piecewise continuous or
continuously differentiable functions with boundary extensions,
denoted by $PC^0(\overline{\Omega})$ or $PC^1(\overline{\Omega})$,
respectively.  For instance, the space $PC^0(\overline{\Omega})$
contains functions that are continuous everywhere in $\Omega$ except
at a finite number of internal points and possesses one-sided limits
everywhere including the boundary.
\end{rem}

\section{Nonlocal Operators} \label{sec:govering_op}

We studied various aspects of local BCs in NL problems over the years 
~\cite{aksoylu2023_comparison,aksoyluBeyerCeliker2017_implementation,aksoyluBeyerCeliker2017_foundation,aksoyluCelikerGazonas2020_asympComp,aksoyluCelikerKilicer2019_breakthru2D,aksoyluGazonas2020_bvp,aksoyluGazonas2020_dispersion,beyerAksoyluCeliker2016_unbounded}.
We present the main ingredients that are necessary to define the
NL governing operators.  For full detail, we refer to
\cite{aksoyluCelikerDiehl2024_implementation,aksoyluCelikerDiehl2024_construction}.
The midpoint of the domain $\Omega=(a,b)$
$$m:=\frac{a+b}{2}$$ plays a pivotal role in defining the operators.
Even and odd parts of a function will be used in the construction.  The
symmetric partner of $x$ with respect to the midpoint is $2m -x$.  One
defines the self-adjoint orthogonal even and odd projection operators
$P_e$ and $P_o$ with respect to $m$ in the following way:
\begin{defn} 
The even and odd projections $P_e$ and $P_o$ with respect to $m$ 
$$
P_e:L^2(\Omega) \to L^2(\Omega) \quad \text{and} \quad 
P_o:L^2(\Omega) \to L^2(\Omega)
$$
are defined by
\begin{equation} \label{defn:PePo}
\begin{aligned}
P_e u(x) & := && \frac12(u(x) + u(2m-x)), \\
P_o u(x) & := && \frac12(u(x) - u(2m-x)).
\end{aligned}
\end{equation} 
\end{defn}
We extend the concept of even and odd functions to the general domain as follows:
\begin{defn}
A function $u(x)$ is said to be even with respect to $m$ when
\begin{equation*}
u(x) = u(2m -x).
\end{equation*}
It is said to be odd with respect to $m$ when
\begin{equation*}
u(x) = -u(2m-x).
\end{equation*}
\end{defn}
Let $C \in L^2(\Omega)$ be a nonnegative univariate even function with respect to the midpoint $m$. Namely,
\begin{equation} \label{even_C}
C(x) = C(2m-x).
\end{equation}
Define the length of the general domain as 
$$
L := b-a.
$$

The integral based convolution operator $\K_\BC$ defined in
\eqref{convo_op} plays a central role in the construction of the
governing operator $\M_\BC$.  However, it originates from a series
based operator which we refer to as the abstract convolution
\cite{aksoyluBeyerCeliker2017_implementation,aksoyluBeyerCeliker2017_foundation}:
\begin{equation} \label{abstractConvo}
\K_\BC u(x) := \sqrt{L}
\sum_{k \in \INT} \braket{e_k^\BC|C} \braket{e_k^\BC|u} \, e_k^\BC(x), \quad \BC \in \{\a, \p\},
\end{equation}
where $\a$ and $\p$ denote the antiperiodic and periodic BCs, and
$\braket{\cdot|\cdot}$ denotes the $L^2(\Omega)$ inner product.  The
eigenfunctions of the classical operator $-\Delta_\BC$ are denoted by
$e_k^\BC, ~\BC \in \{\a, \p\}$.

We incorporate local BC into the NL operator through eigenfunctions
because they satisfy the BC by definition.  All governing operators
$\M_\BC$ are constructed through antiperiodic and periodic BC, and
their mixed combinations.  Hence, the initial effort is put to
construct the governing operator with the periodic BC.  For
implementation purposes, one has to find an integral representation of
\eqref{abstractConvo} which was one of the main themes in
\cite{aksoyluBeyerCeliker2017_implementation,aksoyluCelikerDiehl2024_implementation,aksoyluCelikerDiehl2024_construction}.
The integral representation of the abstract convolution with periodic
BC is given as follows:
\begin{thm}
Let $C \in L^2(\Omega)$ be an even function with respect to the midpoint $m$. Namely,
\begin{equation*}
C(x) = C(2m-x).
\end{equation*}
Let $\K_\p$ be the abstract convolution with
periodic BC defined in \eqref{abstractConvo}.  
Then, the integral representation of $\K_\p$ is 
\begin{equation} \label{integRep}
\K_\p u(x) = \int_a^b \Chat_{\p}(\xp - x + m) u(\xp) \diff \xp.
\end{equation} 
\end{thm}
\begin{proof} See \cite[Sec.~3.1]{aksoyluBeyerCeliker2017_implementation} for the 
domain $\Omega=(-1,1)$ and 
\cite[Thm.~6.2]{aksoyluCelikerDiehl2024_construction} for the general domain 
$\Omega=(a,b)$.
\end{proof}

The function $\Chat_p$ in \eqref{integRep} is the $L$-periodic extension of the kernel function $C$, which will be defined next.  Note that the argument of the kernel
function $\Chat_p$ in \eqref{integRep} is $\xp-x+m$.  Since $x, \xp
\in (a,b)$, $\xp-x \in (a-b,b-a)$.  Hence,
$$\xp - x + m \in (a-b+\frac{a+b}{2},b-a+\frac{a+b}{2}) = (\frac{3a-b}{2},\frac{3b-a}{2}).$$
Consequently, the kernel function sweeps $\Ohat:=(\ahat,\bhat)$
where 
$$
\ahat = \frac{3a-b}{2} \quad\text{and}\quad \bhat = \frac{3b-a}{2}.
$$
Since $\Chat_\a$ and $\Chat_\p$ are the $L$-antiperiodic and $L$-periodic extensions of $C$, respectively, $\Chat_\a$ and $\Chat_\p$ are the same as $C$ on $\Omega$.  For $x \in \Ohat \setminus \Omega$,  
$\Chat_\a$ and $\Chat_\p$ are obtained by appropriate shifts of length $L$.  More
precisely, the extensions are expressed explicitly as
\begin{equation*}
\begin{aligned}
\widehat{C}_\a(x) & :=
\left\{ \begin{aligned}
- & C(x+L), && x \in (\ahat,a), \\
& C(x), && x \in (a,b), \\ 
- & C(x-L), && x \in (b,\bhat),
\end{aligned} \right.
\quad &
\widehat{C}_\p(x) & :=
\left\{ \begin{aligned}
& C(x+L), && x \in (\ahat,a), \\ 
& C(x), && x \in (a,b), \\ 
& C(x-L), && x \in (b,\bhat),
\end{aligned} \right. \\
& & \\
\widehat{C}_{\a\p}(x) & :=
\left\{ \begin{aligned}
- & C(x+L), && x \in (\ahat,a), 
\\ & C(x), && x \in (a,b), 
\\ & C(x-L), && x \in (b,\bhat),
\end{aligned} \right.
\quad & 
\widehat{C}_{\p\a}(x) & :=
\left\{ \begin{aligned}
& C(x+L), && x \in (\ahat,a), 
\\ & C(x), && x \in (a,b), 
\\ - & C(x-L), && x \in (b,\bhat).
\end{aligned} \right.
\end{aligned}
\end{equation*}
Using the projections $P_e$ and $P_o$ given in \eqref{defn:PePo}, the governing
operators are defined as follows:
\begin{eqnarray*}
(\M_{\DD} - c) u(x) & = & - \int_\Omega \big(\hat{C}_\a(\xp-x+m)P_e + \hat{C}_\p(\xp-x+m)P_o \big) u(\xp) \diff \xp \\
(\M_{\NN} - c) u(x) & = & - \int_\Omega\big(\hat{C}_\p(\xp-x+m)P_e + \hat{C}_\a(\xp-x+m)P_o \big) u(\xp) \diff \xp \\
(\M_{\DN} - c) u(x) & = & - \int_\Omega\big(\hat{C}_{\a\p}(\xp-x+m)P_e + \hat{C}_{\p\a}(\xp-x+m)P_o \big) u(\xp) \diff \xp \\
(\M_{\ND} - c) u(x) & = & - \int_\Omega \big(\hat{C}_{\p\a}(\xp-x+m)P_e + \hat{C}_{\a\p}(\xp-x+m)P_o \big) u(\xp) \diff \xp \\
(\M_{\a} - c) u(x) & = & - \int_\Omega \big(\hat{C}_{\a}(\xp-x+m)P_e + \hat{C}_{\a}(\xp-x+m)P_o \big) u(\xp) \diff \xp \\
(\M_{\p} - c) u(x) & = & - \int_\Omega \big(\hat{C}_{\p}(\xp-x+m)P_e + \hat{C}_{\p}(\xp-x+m)P_o \big) u(\xp) \diff \xp,
\end{eqnarray*}
where 
\begin{equation*} 
c =  \int_\Omega C(x) \diff x.
\end{equation*} 

Using the definition of $P_e$ and $P_o$, one can give the explicit expression
of kernel functions in the governing operator \eqref{govern_op} as follows:
\begin{equation*}
\begin{aligned}
k_{\DD}(x,\xp) & = && \frac{1}{2} \big\{
\big[ \hat{C}_\a(\xp-x+m) + \hat{C}_\a(\xp+x-m) \big] && + &&
\big[ \hat{C}_\p(\xp-x+m) - \hat{C}_\p(\xp+x-m) \big] 
\big\} \\
k_{\NN}(x,\xp) & = && \frac{1}{2} \big\{
\big[ \hat{C}_\p(\xp-x+m) + \hat{C}_\p(\xp+x-m) \big] && + &&
\big[ \hat{C}_\a(\xp-x+m) - \hat{C}_\a(\xp+x-m) \big] 
\big\} \\
k_{\DN}(x,\xp) & = && \frac{1}{2} \big\{
\big[ \hat{C}_{\a\p}(\xp-x+m) + \hat{C}_{\a\p}(\xp+x-m) \big] && + && 
\big[ \hat{C}_{\p\a}(\xp-x+m) - \hat{C}_{\p\a}(\xp+x-m) \big] 
\big\} \\
k_{\ND}(x,\xp) & = && \frac{1}{2} \big\{
\big[ \hat{C}_{\p\a}(\xp-x+m) + \hat{C}_{\p\a}(\xp+x-m) \big] && + &&
\big[ \hat{C}_{\a\p}(\xp-x+m) - \hat{C}_{\a\p}(\xp+x-m) \big] 
\big\}\\
k_\a(x,\xp) & = && \hat{C}_\a(\xp-x+m) && &&  \\
k_\p(x,\xp) & = && \hat{C}_\p(\xp-x+m). && &&
\end{aligned}
\end{equation*}

We already mentioned that the operator $\M_\BC$ was inspired by the
original (linearized) peridynamic governing operator
$\M_{\text{orig}}$ defined as
$$\M_{\text{orig}} u(x) = \Big( \int_\Omega k(x,\xp) \diff \xp \Big) \, u(x) 
- \int_\Omega k(x,\xp) u(\xp) \diff \xp,
$$
where $k$ is a kernel that has no reference to a BC.  Note that
$\M_{\text{orig}}$ is merely a formal operator due to the lack of
reference to a rigorous BC \cite[p.~201]{silling2000PDfirstPaper}.
Instead of $k$, when the above kernel functions $k_\BC$ are used and
$\M_{\text{orig}}$ is modified slightly, we arrive at our governing
operator that was already given in \eqref{govern_op}:
\begin{equation*}
\M_\BC u(x)= \Big( \int_\Omega k_\BC(0,\xp) \diff \xp \Big) \, u(x) 
- \int_\Omega k_\BC (x, \xp) u(\xp) \diff \xp,
\end{equation*}
where the constant $c$ in \eqref{govern_op} is related to the bivariate kernel $k_\BC$ and the univariate kernel $C$ in the following way: 
$$c = \int_\Omega k_\BC(0,\xp) \diff \xp = \int_\Omega C(x) \diff x.$$
The slight modification to $\M_{\text{orig}}$ is necessary to obtain
a Fredholm operator of the second kind.  In practice, the univariate
kernel function $C$ in \eqref{even_C} is supported only in a
neighborhood called the \emph{horizon}.  More precisely, define the
indicator function, also known as the flat-top kernel, for $x \in \Omega$:
\begin{equation} \label{flat_top_kernel}
\chi_\delta(x) := \left\{
\begin{array}{ll}
1, & x \in (m-\delta,m+\delta), \\
0, & \textrm{otherwise.}
\end{array}
\right.
\end{equation}
Hence, the size of nonlocality is determined by $\delta$ and the assumption
$\delta < L/2$ is made to confine the computational domain in $\Omega$.  Since the 
horizon is constructed by $\chi_\delta(x)$, a practical kernel function takes the form
\begin{equation*} 
C(x) = \chi_\delta(x) \nu(x),
\end{equation*}
where $\nu(x) \in L^2(\Omega)$ is even.  The notion of horizon
triggers the definition of bulk: 
$$\text{Bulk} :=(a+\delta,b-\delta).$$
Consequently, the two operators $\M_{\text{orig}}$ and $\M_\BC$ agree
in the following way:
\begin{thm} 
When $k(x,\xp) = k_\BC(x,\xp)$,  the following agreement holds:
\begin{equation*}
\M_{\text{orig}} u(x) = \M_\BC u(x) ~\text{when}~
\begin{cases} 
x \in \Omega & \text{if}~ \BC \in \{\NN, \p\}, \\
x \in \text{Bulk} & \text{if}~ \BC \in \{\a, \DD, \DN, \ND \}
\end{cases}
\end{equation*}
and $\M_\BC$ enforces the local $\BC$.
\end{thm}

\begin{proof}
See \cite[Thm.~4.1]{aksoyluCelikerDiehl2024_construction}.
\end{proof}

\section{The Domain Decomposition of the Local Problem and Local-to-Local Coupling} \label{sec:DD}
We mentioned that the DDM is the inspiration for our LtN coupling
method.  We begin by carefully studying the DDM and demonstrate how we
interpret it as a LtL coupling method.  We pay special attention to
the treatment at the interface.  This treatment constitutes the design
philosophy of our LtN coupling method.

The single domain problem \eqref{single_domain_prob} is equivalently
reformulated as a multi-subdomain problem
\cite{natafEtAl2015_book_DD,mathew2008_book_DD,Quarteroni:1999:DDforPDEs,Toselli:2005:DDBook}.
For simplicity, assume a two-subdomain scenario where $\Omega$ is
decomposed into two non-overlapping subdomains as $\Omega = \Omega_1
\cup \Omega_2 \cup \Gamma$ where $\Gamma:= \overline{\Omega}_1 \cap
\overline{\Omega}_2$. In our case $\Omega_1=(a,e)$, $\Omega_2=(e,b)$,
and $\Gamma = e$. The two-subdomain reformulation of
\eqref{single_domain_prob} for the case of $\BC=\DD$ is as follows:
\begin{equation} \label{two_subdomain_prob}
  \begin{aligned}
    \begin{split}
  -\Delta u_1 & = && f  && \text{in}~ \Omega_1\\
  u_1 & = && 0 && \text{on}~ \partial\Omega_1 \setminus \Gamma \\
  u_1 & = && u_2 && \text{on}~ \Gamma \\
  \nabla u_1 \cdot \bn_1 + \nabla u_2 \cdot \bn_2 & =  && 0 && \text{on}~ \Gamma \\
  -\Delta u_2 & = && f  && \text{in}~ \Omega_2\\
  u_2 & = && 0 && \text{on}~ \partial\Omega_2 \setminus \Gamma
\end{split}
    \end{aligned}
\end{equation}
For simplicity, the problem \eqref{two_subdomain_prob} is given for
homogeneous BC.  It is straightforward to generalize the construction
for inhomogeneous BC.

To fully grasp the LtL coupling process, we pose two independent subdomain
problems and want to obtain the single domain problem from the
coupling of these two.  First, setup the subdomain problems by
splitting the equations in \eqref{two_subdomain_prob} according to the
subdomains and introduce a flux variable $g_i$ at the interface:
\begin{equation} \label{subdomain_prob_1}
  \begin{aligned}
    \begin{split}
  -\Delta u_1 & = && f  && \text{in}~ \Omega_1\\
  u_1 & = && 0 && \text{on}~ \partial\Omega_1 \setminus \Gamma \\
  \nabla u_1 \cdot \bn_1 & =  && g_1 && \text{on}~ \Gamma,
    \end{split}
  \end{aligned}
\end{equation}
and
\begin{equation} \label{subdomain_prob_2}
  \begin{aligned}
    \begin{split}
  -\Delta u_2 & = && f  && \text{in}~ \Omega_2\\
  u_2 & = && 0 && \text{on}~ \partial\Omega_2 \setminus \Gamma \\
  \nabla u_2 \cdot \bn_2 & =  && g_2 && \text{on}~ \Gamma.
\end{split}
    \end{aligned}
\end{equation}
The flux variables $g_1$ and $g_2$ are introduced in order to guarantee well-posed subdomain problems.  They will eventually disappear in the formulation due to the balance of flux assumption given in \eqref{two_subdomain_prob}$_4$.  
Note that the BCs for the problems on $\Omega_1$ and $\Omega_2$ are
$\DN$ and $\ND$, respectively.  At the interface $\Gamma$, a Neumann
BC is utilized from both sides, which ensures the equilibrium of force.

The discretization of \eqref{subdomain_prob_1} and
\eqref{subdomain_prob_2} is obtained in the following way:  First, test
the equations with $v$ defined on $\Omega_i$:
\begin{equation*}
\braket{-\Delta u_i|v}_{\Omega_i} = \braket{f|v}_{\Omega_i}.
\end{equation*}  
Apply the divergence theorem:
\begin{equation*}
(-\nabla u_i \cdot \bn_i, v)_{\partial\Omega_i} + \braket{\nabla u_i | \nabla v}_{\Omega_i} = \braket{f|v}_{\Omega_i},
\end{equation*}  
where $(\cdot,\cdot)_{\partial\Omega_i}$ denotes the inner product on
the boundary $\partial\Omega_i$.
Rearrange the equation:
\begin{equation*}
 \braket{\nabla u_i| \nabla v}_{\Omega_i} = \braket{f|v}_{\Omega_i} + (\nabla u_i \cdot \bn_i, v)_{\partial\Omega_i}.
\end{equation*}
Discretize $u_i$ using a uniform
grid with grid size $h$ and the following nodal linear basis
functions:
$$\left\{ \phi_1, \phi_2, \ldots, \phi_e, \ldots, \phi_{N-1}, \phi_N\right\}.$$
Then, the discretization of $u_i$ becomes
\begin{eqnarray*}
  u_1 & = & u_1(a) \phi_1 + u_1(a+h) \phi_2 + \ldots + u_1(e-h) \phi_{e-1} + u_1(e) \phi_e^L \\
  u_2 & = & u_2(e) \phi_e^R + u_2(e+h) \phi_{e+1} + \ldots + u_2(b-h) \phi_{N-1} + u_2(b) \phi_N.
\end{eqnarray*}  
Here, with a slight abuse of notation, we denote the basis function associated to the node at $x=e$ by $\phi_e$, which
can be written by its pieces supported on the intervals $[e-h,e]$ and
$[e,e+h]$.  Denote the left and right pieces by $\phi_e^L$ and
$\phi_e^R$, respectively.  Hence,
\begin{equation} \label{phi_e}
 \phi_e = \phi_e^L + \phi_e^R.
\end{equation}

After applying the BC, the discretization of the local problem on two
subdomains gives the following systems:
$$
A_1 \bu_1 = \bff_1,
$$
which in matrix form becomes
\begin{equation*} 
\frac1h
\begin{bmatrix*}[r]
   1 &  \\
  -1 & 2        & -1        &            &  \\
     & \ddots & \ddots & \ddots &  \\
     &            & -1        &  2        & -1  \\
     &           &             & -1        & 1
\end{bmatrix*}
\begin{bmatrix*}[l]
  u_1(a) \\ u_1(a+h) \\ ~~\vdots \\ u_1(e-h) \\ u_1(e)
\end{bmatrix*}
=
\begin{bmatrix*}[l]
0 \\ \braket{f|\phi_2} \\ ~~\vdots \\ \braket{f|\phi_{e-1}} \\ \braket{f|\phi_e^L} + (\nabla u_1 \cdot \bn_1,\phi_e^L)_{\partial \Omega_1}
\end{bmatrix*}.
\end{equation*}
Similarly the system on $\Omega_2$ is
$$
A_2 \bu_2 = \bff_2,
$$
which takes the following matrix form
\begin{equation*} 
\frac1h
\begin{bmatrix*}[r]
   1 & -1 \\
  -1 & 2         & -1 \\
    & \ddots & \ddots & \ddots &   \\
    &          & -1       &  2         & -1  \\
    &          &         &           & 1
\end{bmatrix*}
\begin{bmatrix*}[l]
  u_2(e) \\ u_2(e+h) \\ ~~\vdots \\ u_2(b-h) \\ u_2(b)
\end{bmatrix*}
=
\begin{bmatrix*}[l]
 \braket{f|\phi_e^R} + (\nabla u_2 \cdot \bn_2,\phi_e^R)_{\partial \Omega_2}\\ \braket{f|\phi_{e+1}} \\ ~~\vdots \\ \braket{f|\phi_{N-1}} \\ 0
\end{bmatrix*}.
\end{equation*}
The boundary terms individually are equal to
\begin{equation} \label{flux_terms}
\begin{aligned}
 (\nabla u_1 \cdot \bn_1,\phi_e^L)_{\partial \Omega_1} && = &&-u_1'(a) \phi_e^L(a) + u_1'(e) \phi_e^L(e) &&  = && u_1'(e) && = &&  (\nabla u_1 \cdot \bn_1) (e) \\
 (\nabla u_2 \cdot \bn_2,\phi_e^R)_{\partial \Omega_2} && = && -u_2'(e) \phi_e^R(e) + u_2'(b) \phi_e^R(b) && =  && -u_2'(e) && = && (\nabla u_2 \cdot \bn_2) (e).
\end{aligned}
\end{equation}

To obtain the coupled system, as an initial step, append the two
systems by using the identification
\begin{equation*}
u_1(e)=u_2(e),
\end{equation*}
which stems from the no jump condition in
\eqref{two_subdomain_prob}$_3$:
\begin{eqnarray}
&& \frac1h \left[
 \begin{array}{rrrrcrrrr}
   1 &           &           &           &        &          &            &             &    \\
  -1 &2         &-1       &           &       &          &              &            &    \\
      &\ddots &\ddots &\ddots &       &          &             &            & \\
      &           &-1        &2         &-1    &          &             &            &  \\
      &           &           &-1        &1+1 &-1       &             &            &   \\
      &           &           &           &-1    &2        &-1          &            &  \\
      &           &           &           &       &\ddots & \ddots & \ddots &   \\
      &           &           &           &       &           & -1        &  2         & -1  \\
      &           &           &           &       &           &            &             & 1   
 \end{array}
 \right]
\begin{bmatrix*}[l]
  u_1(a) \\ u_1(a+h) \\ ~~\vdots \\ u_1(e-h) \\  u_1(e) \\ u_2(e+h)\\ ~~\vdots \\ u_2(b-h) \\ u_2(b)
\end{bmatrix*}
\nonumber \\
& = &
\begin{bmatrix*}[l]
 0 \\ \braket{f|\phi_2} \\ ~~\vdots \\ \braket{f|\phi_{e-1}} \\ \braket{f|\phi_e^L} + (\nabla u_1 \cdot \bn_1,\phi_e^L)_{\partial \Omega_1} + \braket{f|\phi_e^R} + (\nabla u_2 \cdot \bn_2,\phi_e^R)_{\partial \Omega_2} \\ \braket{f|\phi_{e+1}} \\ ~~\vdots \\ \braket{f|\phi_{N-1}} \\ 0
\end{bmatrix*}.  \label{cpldPr_raw}
\end{eqnarray}

Define a new variable $u$ as the solution on the single domain $\Omega$ in the following way:
\begin{equation*}
  u := \left\{
  \begin{array}{ll}
    u_1, & \text{in}~ \Omega_1, \\
    u_2, & \text{in}~ \Omega_2, \\
    u_1 (=u_2), & \text{on}~ \Gamma.
  \end{array}
  \right.
\end{equation*}
The no jump condition in \eqref{two_subdomain_prob}$_3$ guarantees the
continuity of $u$ on $\Omega$.

To get to the final coupled system, perform the addition in the stiffness
matrix.  Let's concentrate on the crowded entry in the load vector:
\begin{alignat*}{2}
 \braket{f|\phi_e^L} & +  \braket{f|\phi_e^R} + (\nabla u_1 \cdot \bn_1,\phi_e^L)_{\partial \Omega_1} + (\nabla u_2 \cdot \bn_2,\phi_e^R)_{\partial \Omega_2} && \\
& = 
 \braket{f|\phi_e} + (\nabla u_1 \cdot \bn_1,\phi_e^L)_{\partial \Omega_1} + (\nabla u_2 \cdot \bn_2,\phi_e^R)_{\partial \Omega_2} \quad && \text{using  \eqref{phi_e}}\\
& = \braket{f|\phi_e} + (\nabla u_1 \cdot \bn_1) (e) + (\nabla u_2 \cdot \bn_2) (e) \quad && \text{summing the terms in \eqref{flux_terms}} \\
& = \braket{f|\phi_e} \quad && \text{invoking the flux balance in \eqref{two_subdomain_prob}$_4$.}
\end{alignat*}
Now, all the pieces of the interface equation fall into place.  One
clearly sees that the local problem enjoys the hallmark feature {\bf
  (HF2)}.  Consequently, the coupled problem \eqref{cpldPr_raw}
seamlessly turns into the single domain problem on $\Omega$:
$$A \bu = \bff,$$
which, in matrix form, is 
\begin{equation*} 
\frac1h \left[
 \begin{array}{rrrrrrrrr}
   1 &           &           &           &        &          &            &             &    \\
  -1 &2         &-1       &           &       &          &              &            &    \\
      &\ddots &\ddots &\ddots &       &          &             &            & \\
      &           &-1        &2         &-1    &          &             &            &  \\
      &           &           &-1        &2 &-1       &             &            &   \\
      &           &           &           &-1    &2        &-1          &            &  \\
      &           &           &           &       &\ddots & \ddots & \ddots &   \\
      &           &           &           &       &           & -1        &  2         & -1  \\
      &           &           &           &       &           &            &             & 1   
 \end{array}
 \right]
\begin{bmatrix*}[l]
  u(a) \\ u(a+h) \\ ~~\vdots \\ u(e-h) \\  u(e) \\ u(e+h)\\ ~~\vdots \\ u(b-h) \\ u(b)
\end{bmatrix*}
=
\begin{bmatrix*}[l]
0 \\ \braket{f|\phi_2} \\ ~~\vdots \\ \braket{f|\phi_{e-1}} \\ \braket{f|\phi_e} \\ \braket{f|\phi_{e+1}} \\ ~~\vdots \\ \braket{f|\phi_{N-1}} \\ 0
\end{bmatrix*}.
\end{equation*}
\begin{rem}
The first crucial step in domain decomposition is the proof of
equivalence of the two-subdomain problem to the single-domain one.  We
basically reproduced this (the harder part of the equivalence, i.e.,
\eqref{two_subdomain_prob} $\Rightarrow$ \eqref{single_domain_prob})
proof by resorting to linear finite element discretization.  This
proof establishes the fact that the LtL coupling of two subdomain
problems seamlessly gives the single-domain problem in the case of linear
finite element discretization.  The proof employing a general
discretization can be obtained by resorting to weak formulation; see
\cite[Lemma~1.2.1]{ Quarteroni:1999:DDforPDEs}.
\end{rem}

\section{Sparsity of the Stiffness Matrix} \label{sec:sparsity}

For the discretization of local and NL, weak formulations are used.
For the local and NL operators, we use the linear finite element
discretization and the Galerkin projection, respectively.  In this
section, we carefully present the sparsity structure the stiffness
matrix.  It is vital to know which entries contribute to the Taylor
expansions in order to accomplish {\bf (HF1)} and {\bf (HF2)}.

Since weak formulation is used for discretization, the stiffness matrix entry is defined as
\begin{eqnarray} 
  A_{ij} & = & \braket{\M_\BC \phi_i| \phi_j } \nonumber \\
  & = & c \braket{\phi_i| \phi_j } - \braket{\K_\BC \phi_i| \phi_j } \nonumber \\
& =: & c M_{ij} - K_{ij}. \label{stiffMtrx_explicit1}
\end{eqnarray}
For linear basis functions, the matrix $M_{ij}$ in
\eqref{stiffMtrx_explicit1} is the tridiagonal mass matrix.  On the
other hand, the convolution term $K_{ij}$ produces more nonzero
entries per row due to the NL support of the kernel function
$k_\BC(x,\xp)$.  More precisely, considering the fixed node $x_i$ in
the bulk, we want to determine the column locations, i.e.,
$j$-indices, of nonzero entries in row $i$.  Since $i$ is fixed for
our consideration, to identify $j$-indices, one needs to move the
action of the operator $\K_\BC$ from $\phi_i$ to $\phi_j$.  This is
achieved using the self-adjointness of the operator $\K_\BC$:
\begin{equation*} 
  K_{ij} = \braket{\K_\BC \phi_i| \phi_j } = \braket{\phi_i| \K_\BC \phi_j }.
\end{equation*}
After rearranging the inner product, one obtains
\begin{equation} \label{stiffMtrx_explicit2}
  K_{ij} = \braket{\K_\BC \phi_j,\phi_i } =
  \int_\Omega \int_\Omega k_\BC(x,\xp) \phi_j(\xp)  \phi_i(x) \diff \xp \diff x. 
\end{equation}
Since the integral in \eqref{stiffMtrx_explicit2} is a double
integral, one needs to monitor of the support of $\phi_i(x)$ and more
importantly, that of $k_\BC(x,\xp) \phi_j(\xp)$ for fixed $x$.

Let $\text{supp}_i$ denote the support of the basis function $\phi_i$,
namely,
$$ \text{supp}_i := \overline{ \{x: \phi_i(x) \neq 0 \}}.$$ For fixed
$x$, the variable $\xp$ sweeps the interval $(x-\delta,x+\delta)$,
which in interval arithmetic is denoted by $(x-\delta,x+\delta) = x +
(-\delta,\delta).$ Since $x \in \text{supp}_i$, in totality $\xp$
sweeps the interval
$$\bigcup_{x \in \text{supp}_i} \big\{x + (-\delta,\delta)\big\} = \text{supp}_i + (-\delta,\delta),$$
where the sum of intervals is defined as $[a,b] + (\alpha,\beta) := (a+\alpha,b+\beta).$

To get a nonzero integral, we are interested in the basis functions
$\phi_j(\xp)$ whose supports have nontrivial intersection with
$\text{supp}_i + (-\delta, \delta)$.  Hence, the index of such basis
functions is denoted by
 \begin{equation} \label{Ii}
\mathcal{J}_{i} := \Big\{ j: \big\{\text{supp}_i + (-\delta,\delta) \big\}\cap  \text{supp}_j \neq \emptyset \Big\}.
\end{equation}
The index set $\mathcal{J}_{i}$ contains the column indices of nonzero
entries present in the $i$th row.

Throughout the paper, we assume that $\delta$ is an integer multiple
of $h$, i.e., $\delta = R h$ with a positive integer $R$.  Then,
\begin{equation*}
\text{supp}_i + (-\delta,\delta) = [x_{i-1},x_{i+1}] + (-Rh,Rh) = (x_{i-R-1}, x_{i+R+1}).
 \end{equation*}
Then, recalling \eqref{Ii}, one arrives at
$$\mathcal{J}_{i} = \{i-R-1, i-R, \ldots, i, i+1, \ldots, i+R, i+R+1 \}.$$
The number of nonzeros in the $i$th row of $K$ is the number of
indices in $\mathcal{J}_{i}$, which is $2(R+1) + 1$.  This number
dictates the number of nonzeros in the $i$th row of $A$ as well.

\section{The Boundary Condition Enforcement in the Weak Formulation} \label{sec:BC_weak}

The flat-top kernel in \eqref{flat_top_kernel} is chosen throughout
the paper, hence, the NL operator is scaled with
\begin{equation} \label{scaling}
scl = \frac{3}{\delta^3},
\end{equation}
so that the eigenvalues of the NL operator converge to those of the local
operator as $\delta \to 0$
\cite[Sec.~4]{aksoyluGazonas2020_dispersion}. 
In the Galerkin projection method, the stiffness matrix $A$ arising
from the discretization of the NL operator is obtained from the inner
product of the weak form in the following way:
\begin{equation*} 
\bv^\top A \bu :=  \braket{scl \, \M_\BC u| v}, 
\end{equation*}
where $\bu$ and $\bv$ are the coordinates of $u$ and $v$,
respectively.  The fact that the stiffness matrix $A$ is a quadratic
form will become instrumental in the ensuing discussion.

\magenta{In the weak form of the local problem, the BC is captured by a single
row of the stiffness matrix; see \eqref{Neumann_BC_local}.  In the NL
formulation, however, the BC is captured by several rows.  First,
recall that, in a weak formulation, the BC is enforced weakly, meaning
that the discretization provides an approximation of the BC.  Hence,
the associated rows collectively produce an approximation of the BC.}
Depending on the size of the horizon, $\delta=Rh$, the BC equation at
$x=a$ on $\Omega_1=(a,e)$ is obtained by summing the first $R+1$ rows of $A$.  \magenta{Since the
stiffness matrix is a quadratic form, we write the boundary equation
by adopting a quadratic form notation.  For instance, for $\delta=h$,
i.e., $R=1$, the Neumann BC equation obtained from \eqref{quadratic_form_right} is}
\begin{equation} \label{Neumann_BC_NL}
  \begin{bmatrix*} 1 \\ 1\end{bmatrix*}^\top
    \frac{1}{8h}
\begin{bmatrix*}[r]
  5 & -4 & -1 & 0  \\
  -4 & 9 & -4 & -1
\end{bmatrix*}
\begin{bmatrix*}[l]
u(a) \\ u(a+h) \\ u(a+2h) \\ u(a+3h)
\end{bmatrix*}
= \braket{f,\psi},
\end{equation}
where $\psi$ is the sum of the basis functions associated with the
first $R+1$ nodes.  Namely,
\begin{equation} \label{defn:psi}
\psi = \phi_a^R + \phi_{a+1}.
\end{equation}
The coordinate vector of this $\psi$ in the basis that lives on $\Omega_1$ is 
$\begin{bmatrix*} 1 \\ 1\end{bmatrix*}$.
The sum in the BC equation is due to the transpose on the coordinate vector
of the basis functions in the definition of $\psi$ in \eqref{defn:psi}.  The expression in
\eqref{Neumann_BC_NL} gives a Neumann condition as shown in
\eqref{taylor_func_R_L}.

Similar to \eqref{Neumann_BC_NL}, to maintain an alignment between the presentation of local and NL problems, we also adopt a quadratic form notation and write the local boundary equation for Neumann BC as
\begin{equation} \label{Neumann_BC_local}
  \begin{bmatrix*} 1 \end{bmatrix*}^\top
    \frac{1}{h}
\begin{bmatrix*}[r]
  1& -1
\end{bmatrix*}
\begin{bmatrix*}[l]
u(a) \\ u(a+h)
\end{bmatrix*}
= g(a),
\end{equation}
where the boundary data $g(a)$ is provided with the Neumann condition 
$g(a)=\nabla u(a) \cdot \bn(a).$
In Sec.~\ref{sec:interface_equ}, we set up the coupled equation at the interface by using the BC equations in \eqref{Neumann_BC_NL} and \eqref{Neumann_BC_local}. 
We will verify the validity of an interface condition by resorting to a Taylor expansion.

\section{The Neumann Boundary Condition and Taylor Expansions} \label{sec:Neu_BC_Taylor}
Consider a dynamic problem, such as the wave equation, governed either
by the local or NL operator with pure Neumann BCs.  In
\cite[Secs.~5~and~6]{aksoylu2023_comparison}, we proved that both
operators guarantee the balance of linear momentum when they are used
with pure Neumann BCs.  In the static problems discussed, this
favorable balance of linear momentum property is interpreted as the
equilibrium of force.  For coupling, we choose the equilibrium of
force as a design principle.  As a result, any problem that does not
touch the boundary is posed as a pure Neumann problem.  We labeled
this property as hallmark feature {\bf (HF2)}.

When a Neumann condition at a point is denoted simply as a derivative
at that point, the direction of the differentiation is lost.  To
clearly indicate the direction, we adopt a two-dimensional normal
derivative notation.  Note that the normal direction at a left and a
right point is $\bn(a) = (-1,0)$ and $\bn(b) = (1,0)$, respectively.
Hence,
\begin{alignat*}{3}
& (\nabla u \cdot \bn)(a) && =\nabla u (a) \cdot (-1,0)  && = -u'(a) \\
& (\nabla u \cdot \bn)(b) && = \nabla u(b) \cdot (1,0)   && = u'(b).
\end{alignat*}
\magenta{Since the Neumann BC plays such a critical role in the setup of coupling,
we show that the discretized NL operator enforces a Neumann condition
at the interface by resorting to a Taylor expansion.}  For accessibility, we present the case of $\delta=h$.  The case of $\delta=2h$ and $\delta=3h$ are presented in Sec.~\ref{sec:Neumann_conds}.

\subsection{The Neumann Condition on the Right and Left Sides}
Consider the $\M_\ND$ operator on $\Omega_2=(e,b)$ for the case of $\delta=h$. 
Let us clearly identify which block of $A$ is involved with the BCs.
Since the Neumann BC is imposed at the left end point of $\Omega_2$, one
needs to take into account the interaction taking place on the right
side of $x=e$.  One defines the function $\N_{\rright, \delta=h}^\NL
u(e)$ associated with the BC from the right side of $x=e$.
In this case, the test functions that contribute to the BC are
$v=\{\phi_e^R,\phi_{e+1}\}$.  On the other hand, the trial functions
that interact with $v = \phi_e^R$ and $v = \phi_{e+1}$ are
$u=\{\phi_e^R, \phi_{e+1}, \phi_{e+2} \}$ and $u=\{\phi_e^R,
\phi_{e+1}, \phi_{e+2}, \phi_{e+3} \}$, respectively.  Hence, the part
of $A$ that enforces the BC is the top left block.  More precisely,
$$ \N_{\rright, \delta=h}^\NL u(e) : = \braket{scl \, \M_\ND u| v},$$
where 
$$u = u(e) \phi_e^R  + u(e+h) \phi_{e+1} + u(e+2h) \phi_{e+2} + u(e+3h) \phi_{e+3}
\quad \text{and} \quad  v = \phi_e^R + \phi_{e+1}.$$  

The Neumann condition on the left side is similar:
Consider the $\M_\DN$ operator on $\Omega_1=(a,e)$ and define 
$$ \N_{\lleft, \delta=h}^\NL u(e) : = \braket{scl \, \M_\DN u| v}$$
with the choice of
$$u = u(e-3h) \phi_{e-3}  + u(e-2h) \phi_{e-2} + u(e-h) \phi_{e-1} + u(e) \phi_{e}^L 
\quad \text{and} \quad v = \phi_{e-1} + \phi_{e}^L.$$  
Invoking the quadratic form notation in \eqref{Neumann_BC_NL}, one arrives at the following expressions:
\begin{eqnarray}
  \N_{\rright, \delta=h}^\NL u(e) & = & 
  \begin{bmatrix*} 1 \\ 1\\ 0 \\\vdots \\ 0\end{bmatrix*}^\top
    \frac{1}{8h}
\left[ \begin{array}{rrrrcc}
  5 & -4 & -1 & 0 & 0 & \cdots \\
  -4 & 9 & -4 & -1 & 0 & \cdots \\
  * & * & * & * & * & * \\
   \vdots \, & \vdots \, & \vdots \,& \vdots \, & \vdots & \vdots \\
   * & * & * & * & * & *
\end{array} \right]
\begin{bmatrix*}[l]
u(e) \\ u(e+h) \\ u(e+2h) \\ u(e+3h) \\ 0 \\ \vdots \\ 0
\end{bmatrix*} \nonumber \\ 
& = & 
  \begin{bmatrix*} 1 \\ 1\end{bmatrix*}^\top
    \frac{1}{8h}
\begin{bmatrix*}[r]
  5 & -4 & -1 & 0  \\
  -4 & 9 & -4 & -1
\end{bmatrix*}
\begin{bmatrix*}[l]
u(e) \\ u(e+h) \\ u(e+2h) \\ u(e+3h)
\end{bmatrix*} \label{quadratic_form_right} \\
  \N_{\lleft, \delta=h}^\NL u(e) & = & 
  \begin{bmatrix*} 0 \\ \vdots \\ 0 \\ 1 \\ 1 \end{bmatrix*}^\top
    \frac{1}{8h}
\left[ \begin{array}{ccrrrr}
    * & * & * & * & * & * \\
    \vdots \, & \vdots \, & \vdots \,& \vdots \, & \vdots \, & \vdots \, \\
    * & * & * & * & * & * \\
    \cdots & 0 & 0 & -1 & -4 & 5 \\
  \cdots & 0 & -1 & -4 & 9 & -4  \\
\end{array} \right]
\begin{bmatrix*}[l]
0 \\ \vdots \\ 0 \\ u(e-3h) \\ u(e-2h) \\ u(e-h) \\ u(e)
\end{bmatrix*} \nonumber \\
& = & 
  \begin{bmatrix*} 1 \\ 1\end{bmatrix*}^\top
    \frac{1}{8h}
\begin{bmatrix*}[r]
  0  & -1 & -4 & 5  \\
  -1 & -4 & 9 & -4
\end{bmatrix*}
\begin{bmatrix*}[l]
u(e-3h) \\ u(e-2h) \\ u(e-h) \\ u(e)
\end{bmatrix*}. \nonumber 
\end{eqnarray}

This expression is identical to the one given in \eqref{Neumann_BC_NL}
for $x=a$.  A careful Taylor expansion yields
\begin{equation} \label{taylor_func_R_L}
\begin{aligned} 
\N_{\rright, \delta=h}^\NL u(e) && = && -u'(e) + \bO(h) = &~ (\nabla u \cdot \bn)(e) + \bO(h) \\ 
\N_{\lleft, \delta=h}^\NL u(e) && = && u'(e) + \bO(h) = &~ (\nabla u \cdot \bn)(e) + \bO(h). 
\end{aligned}
\end{equation}
Hence, the hallmark feature {\bf (HF2)} is satisfied for the NL operator both on the right and left sides of the interface.

\subsection{The Neumann Condition in the Local Case}
For the local problem, depending on the direction, similar Neumann
conditions are enforced at the interface.  Hence, one needs to define local
counterparts of $\N_{\rright, \delta=h}^\NL $ and $\N_{\lleft,
  \delta=h}^\NL$, which we call as $\N_{\lleft}^\L u(e)$ and
$\N_{\rright}^\L u(e)$, respectively.

Using the fact that the only test function that contributes to the BC is either $v=\phi_e^R$ or $v=\phi_e^L$, the corresponding trial functions are $u=\{\phi_e^R, \phi_{e+1}\}$ or $u=\{\phi_{e-1}, \phi_{e}^L\}$ for the right and the left function, respectively.  Invoking the quadratic form notation in \eqref{Neumann_BC_local}, the stiffness matrix arising from the discretization of the local problem dictates
\begin{alignat}{1}
  \N_{\rright}^\L u(e) & = 
  \begin{bmatrix*} 1 \end{bmatrix*}^\top
    \frac{1}{h}
\begin{bmatrix*}[r]
  1& -1
\end{bmatrix*}
\begin{bmatrix*}[l]
u(e) \\ u(e+h)
\end{bmatrix*} \nonumber \\ 
  \N_{\lleft}^\L u(e) & = 
  \begin{bmatrix*} 1 \end{bmatrix*}^\top
    \frac{1}{h}
\begin{bmatrix*}[r]
  -1& 1
\end{bmatrix*}
\begin{bmatrix*}[l]
u(e-h) \\ u(e)
\end{bmatrix*}. \label{taylor_func_L_left}
\end{alignat}
Taylor expansions, simpler than those in \eqref{taylor_func_R_L}, yield
\begin{equation*} 
\begin{aligned}
 \N_{\rright}^\L u(e) && =  && -u'(e) + \bO(h) && = &~ (\nabla u \cdot \bn)(e) + \bO(h) \\
\N_{\lleft}^\L u(e) && = && u'(e) + \bO(h) && = & ~(\nabla u \cdot \bn)(e) + \bO(h).
\end{aligned}
\end{equation*} 
Clearly, the hallmark feature {\bf (HF2)} is satisfied for the local operator on either side of the interface. 

\section{The Equation at the Interface and the Strong Form}  \label{sec:interface_equ}
One way to verify the validity of a weak form discretization is to
check if it captures an approximation of the underlying equation in
strong form.  Taylor expansions are utilized for such verification.
Since the interface is the most important location for the coupled
problem, we first focus on verifying the discretization of the LtL
coupled problem at the interface:
$$
\E_{\cpld}^\LtL: \quad 
\begin{bmatrix*} 1 \end{bmatrix*}^\top
 \frac{1}{h}
\begin{bmatrix*}[r] -1& 1 \end{bmatrix*}
\begin{bmatrix*}[l] u(e-h) \\ u(e) \end{bmatrix*}
+ 
\begin{bmatrix*} 1 \end{bmatrix*}^\top
\frac{1}{h}
\begin{bmatrix*}[r] 1& -1 \end{bmatrix*}
\begin{bmatrix*}[l] u(e) \\ u(e+h) \end{bmatrix*}
= h f(e).
$$ 
Move the $h$ to the left hand side and obtain an equivalent expression:
$$
\E_{\cpld}^\LtL: \quad 
 \frac{1}{h^2}
\begin{bmatrix*} 1 \end{bmatrix*}^\top
\begin{bmatrix*}[r] -1 & 2 & -1 \end{bmatrix*}
\begin{bmatrix*}[l] u(e-h) \\ u(e)  \\ u(e+h) \end{bmatrix*}
= f(e).
$$ 
Using a Taylor expansion, we see that the LtL coupled problem captures the local 
operator in strong form:
\begin{equation} \label{local_discrete_captures_Laplace}
 \frac{1}{h^2}
\begin{bmatrix*} 1 \end{bmatrix*}^\top
\begin{bmatrix*}[r] -1 & 2 & -1 \end{bmatrix*}
\begin{bmatrix*}[l] u(e-h) \\ u(e)  \\ u(e+h) \end{bmatrix*} = -\Delta u(e) + 
\bO(h^2).
\end{equation}
It is noteworthy that the accuracy of the LtL approximation is
$\bO(h^2)$.  Next, we verify if a similar approximation holds for the
LtN coupled problem.

\subsection{Coupled Equation at the Interface} 
Consider the configuration that the local operator and the NL operator
are placed on the left and right of the interface, respectively.  For
given $f$, the coupled operator and the coupled problem become
\begin{equation*} 
  \M_{\cpld} := \left\{
  \begin{array}{rl}
    -\Delta, & \text{in}~\Omega_1 := (a,e)\\
    scl \, \M_\ND, & \text{in}~ \Omega_2 :=(e,b)
  \end{array} \right. \quad \text{and} \quad
  \M_\cpld u = f.
\end{equation*}
Define the coupled equation at the interface as
$$
\N_{\cpld,\delta=h}^\LtN u(e) := \N_{\lleft}^\L u(e) + \N_{\rright,\delta=h}^\NL u(e).
$$
For the case of $\delta = h$, the test functions that contribute to
the interface equation are $v=\phi_e^L$ and $v=\phi_e^R+\phi_{e+1}$ in
$\Omega_1$ and $\Omega_2$, respectively.  Since the coupled system is
obtained by appending the two subproblems, the test function involved
is the sum of those functions, namely,
$$\psi_{\cpld}^{\delta=h}=\phi_e^L + \big(\phi_e^R+\phi_{e+1} \big)= \phi_e + \phi_{e+1}.$$
Hence, the load vector entry at the interface is obtained by testing $f$ with $\psi_{\cpld}^{\delta=h}$.
This brings us to the following definition:
\begin{equation*}
r_{\cpld,\delta=h}^\LtN (e) := \Trp\big( \braket{f|\psi_{\cpld}^{\delta=h}} \big),
\end{equation*}
where $\Trp\big( \braket{f|\psi_{\cpld}^{\delta=h}} \big)$ is the Taylor expansion at $x=e$ of the trapezoidal rule approximation of $\braket{f|\psi_{\cpld}^{\delta=h}}$.
The equation at the interface becomes
$$\E_{\cpld, \delta=h}^\LtN: \quad \N_{\cpld,\delta=h}^\LtN u(e) = r_{\cpld,\delta=h}^\LtN (e).$$
Using the trapezoidal rule, one gets
\begin{alignat*}{1}
\braket{f|\psi_{\cpld}^{\delta=h}} = h [ \frac12 & f(e-h) \psi_{\cpld}^{\delta=h}(e-h) +
f(e) \psi_{\cpld}^{\delta=h}(e) + \\
& f(e+h) \psi_{\cpld}^{\delta=h}(e+h) + 
\frac12 f(e+2h) \psi_{\cpld}^{\delta=h}(e+2h) ] + \bO(h^2).
\end{alignat*}
Using 
\begin{equation*}
\psi_{\cpld}^{\delta=h}(e-h) = \psi_{\cpld}^{\delta=h}(e+2h) = 0 \quad \text{and} \quad
\psi_{\cpld}^{\delta=h}(e) = \psi_{\cpld}^{\delta=h}(e+h) = 1,
\end{equation*}
one obtains 
\begin{equation*}
\braket{f|\psi_{\cpld}^{\delta=h}} = h \big( f(e) + f(e+h) \big) + \bO(h^2).
\end{equation*}
Apply a first order Taylor expansion on  $f(e+h)$ and obtain
\begin{equation*}
\braket{f|\psi_{\cpld}^{\delta=h}} = 2h f(e) + \bO(h^2).
\end{equation*}
This expression is what we define as $r_{\cpld,\delta=h}^\LtN (e)$, hence
\begin{equation*}
r_{\cpld,\delta=h}^\LtN (e) = 2h f(e).
\end{equation*}
Using \eqref{taylor_func_L_left} and \eqref{quadratic_form_right}, the equation at the interface finally becomes
$$\E_{\cpld, \delta=h}^\LtN: \quad  
  \N_{\lleft}^\L +   \N_{\rright, \delta=h}^\NL = 2h f(e).$$
We would like to relate $\E_{\cpld, \delta=h}^\LtN$ to the
discretization of the local operator in strong form.  For that we move
the $2h$ to the left hand side and obtain an equivalent expression:
\begin{equation} \label{interface_equ_equiv}
\E_{\cpld, \delta=h}^\LtN: \quad  
\frac{1}{2h} \N_{\lleft}^\L +   \frac{1}{2h} \N_{\rright, \delta=h}^\NL = f(e). 
\end{equation}
Rewriting \eqref{interface_equ_equiv}, we arrive at a  remarkable result:  The left hand side of \eqref{interface_equ_equiv} is an approximation of the local operator.  More precisely, 
\begin{eqnarray}
&& \frac{1}{2h^2} \begin{bmatrix*} 1 \end{bmatrix*}^\top
\begin{bmatrix*}[r]
  -1& 1
\end{bmatrix*}
\begin{bmatrix*}[l]
u(e-h) \\ u(e)
\end{bmatrix*}
+ 
  \frac{1}{16h^2} 
\begin{bmatrix*} 1 \\ 1\end{bmatrix*}^\top
\begin{bmatrix*}[r]
  5 & -4 & -1 & 0  \\
  -4 & 9 & -4 & -1
\end{bmatrix*}
\begin{bmatrix*}[l]
u(e) \\ u(e+h) \\ u(e+2h) \\ u(e+3h)
\end{bmatrix*} \nonumber \\
&& =  -\Delta u(e) + \bO(h). 
\label{nonlocal_discrete_captures_Laplace_delta_h}
\end{eqnarray}
It is nontrivial to conclude that so many terms in
\eqref{nonlocal_discrete_captures_Laplace_delta_h} would lead to an
approximation of the local operator.  For that result, we expanded
each term by symbolic computation.  The equation
\eqref{nonlocal_discrete_captures_Laplace_delta_h} represents the case
$\delta=h$.  For larger horizon sizes, the symbolic computation is
even more involved.  For the general case of $\delta=Rh$ case, the
same Taylor expansion with a leading error term of $\bO(h)$ holds.
It is instructive to setup the interface equation for the most common
horizon choice $\delta = 3h$.  Similar to
\eqref{quadratic_form_right}, define
\begin{equation*}
\N_{\rright, \delta=3h}^\NL u(e) := 
\begin{bmatrix*} 1 \\ 1 \\ 1 \\ 1 \end{bmatrix*}^\top
\frac{1}{216h}
\begin{bmatrix*}[r]
  36 & 0 & -23 & -12 & -1 & 0 & 0  & 0 \\
0 & 49  &  -12 & -24 & -12 & - 1 & 0 & 0 \\
-23 & -12 & 71 & 0 & -23 & -12 & -1 & 0 \\
-12 & -24 & 0 & 72 & 0 & -23 & -12 & -1
\end{bmatrix*}
\begin{bmatrix*}[l]
u(e) \\ u(e+h) \\ u(e+2h) \\ u(e+3h) \\ u(e+4h) \\ u(e+5h) \\ u(e+6h)
\end{bmatrix*}.
\end{equation*}
The test function related to the interface equation is the following sum:
\begin{alignat*}{1}
\psi_{\cpld}^{\delta=3h} & =\phi_e^L + \big( \phi_e^R+\phi_{e+1} + \phi_{e+2}+ \phi_{e+3} \big) \\
& = \phi_e + \phi_{e+1} + \phi_{e+2}+ \phi_{e+3}.
\end{alignat*}
Apply a first order Taylor expansion and obtain
\begin{equation*}
\braket{f|\psi_{\cpld}^{\delta=3h}} = 4h f(e) + \bO(h^2).
\end{equation*}
Hence, the equation at the interface becomes
\begin{equation*}
\E_{\cpld, \delta=3h}^\LtN: ~ \N_{\lleft}^\L +  \N_{\rright, \delta=3h}^\NL u(e) =  4h f(e).
\end{equation*}
Move the $4h$ factor to the left hand side of the equation and arrive at
\begin{equation*}
\E_{\cpld, \delta=3h}^\LtN: ~\frac{1}{4h} \N_{\lleft}^\L +  
\frac{1}{4h} \N_{\rright, \delta=3h}^\NL u(e) = f(e).
\end{equation*} 
Eventually, the equation at the interface approximates the local operator:
\begin{equation} \label{nonlocal_discrete_captures_Laplace_delta_3h}
\frac{1}{4h} \N_{\lleft}^\L +  \frac{1}{4h} \N_{\rright, \delta=3h}^\NL u(e) 
=  -\Delta u(e) + \bO(h). 
\end{equation}
Note that similar to \eqref{local_discrete_captures_Laplace}, the
equations \eqref{nonlocal_discrete_captures_Laplace_delta_h}  and 
\eqref{nonlocal_discrete_captures_Laplace_delta_3h} establish that
the coupled equation at the interface captures the local operator in
strong form.  But, we observe that the accuracy of the approximation
is $\bO(h)$, one order lower than the LtL coupled approximation in
\eqref{local_discrete_captures_Laplace}.  We believe that our coupling
method works because at the interface, we always end up with an
equation that approximates the local operator.

\iftrue

\subsection{Matrices in Color} We display the arising matrices in color from the discretization of the coupled problems in weak form.  The matrices for the L-NL-L and NL-L-NL coupled configurations are shown in Figs.~\ref{fig:L-NL-L_mtrx} and \ref{fig:NL-L-NL_mtrx}.  

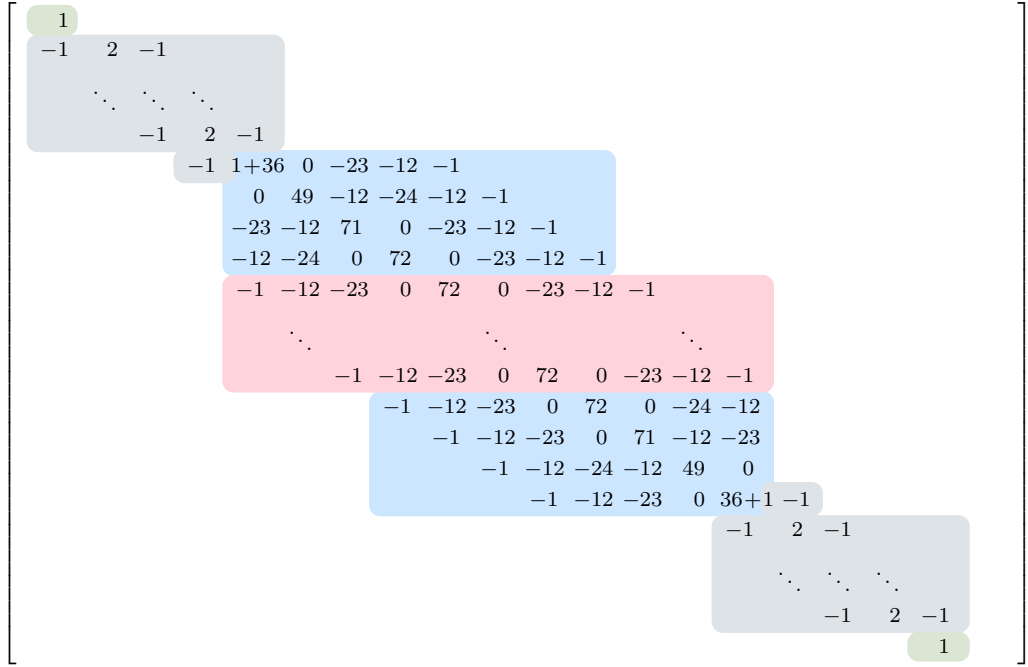
\begin{figure}[tb]
\centering
\scriptsize
\begin{tikzpicture}
  \matrix (m)[
    matrix of math nodes, 
    nodes in empty cells,
    nodes={text width={width(999)}, 
    align=right},
    right delimiter=\rbrack,left delimiter=\lbrack
  ] {
1 \\
-1 & 2 & -1 & & \\
& \ddots  & \ddots & \ddots  & & & \\
& & -1 & 2 & -1 \\
& & & -1 & {1\!+\!36} & 0 & -23 & -12 & -1 & & & & \\
& & &  & 0 & 49 & -12 & -24 & -12 & -1 & & & \\
& & &  & -23 & -12 & 71 & 0 & -23 & -12 & -1 & & \\
& & &  & -12 & -24 & 0 & 72 & 0 & -23 & -12 & -1 & \\
& & &  & -1 & -12  & -23  & 0  & 72  & 0  & -23  &  -12 & -1 & \\
& & &  &  & \ddots &  & & & \ddots &  &  &  & \ddots \\
& & & & &  & -1 & -12  & -23  & 0  & 72  & 0  & -23  &  -12 & -1 &  \\
& & & & &  & & -1 & -12 & -23 & 0 & 72 & 0 & -24 & -12 & &  \\
& & & & &  & & & -1 & -12 & -23 & 0 & 71 & -12 & -23 & & \\
& & & & &  & & &   & -1 & -12 & -24 & -12 & 49 & 0 &  &  \\
& & & & &  & & &  & & -1 & -12 & -23 & 0 & {36\!+\!1} & \;\; -1 & \\
& & & & & &  & & &  & &  &  &  & -1 &  2 & -1  &  \\
& & & & & &  & & &  & &  &  &  &  &  \ddots & \ddots  & \ddots \\
& & & & & &  & & &  & &  &  &  &  &   & -1  & 2 & -1 & \\
 & & & & &  & & &  & &  &  &  &  &   &   &  &  & 1 \\
} ;
\begin{pgfonlayer}{myback}
\fhighlight[azure!20]{m-5-5}{m-8-12}
\fhighlightL[azure!20]{m-11-8}{m-15-15} 
\fhighlightL[awesome!20]{m-8-5}{m-11-15}
\fhighlightL[cadetgrey!30]{m-1-1}{m-4-5}
\fhighlightL[cadetgrey!30]{m-4-4}{m-5-4}
\fhighlightL[cadetgrey!30]{m-15-15}{m-18-19}
\fhighlightL[cadetgrey!30]{m-14-16}{m-15-16}
\fhighlight[asparagus!30]{m-1-1}{m-1-1}
\fhighlightL[asparagus!30]{m-18-19}{m-19-19}
\end{pgfonlayer}
\end{tikzpicture}
\caption{The matrix arising from the discretization of the L-NL-L coupled problem with the choice of the flat-top kernel and a horizon of $\delta=3h$.  For clarity,  only the integer valued entries are shown and the scalings are skipped.  The actual values of the stiffens matrix are obtained by multiplying the local and nonlocal blocks with $\frac{1}{h}$ and $\frac{3}{\delta^3}\frac{h^2}{24}=\frac{1}{216 h}$, respectively}
\label{fig:L-NL-L_mtrx}
\end{figure}

\begin{figure}[tb]
\centering
\scriptsize
\begin{tikzpicture}
  \matrix (m)[
    matrix of math nodes, 
    nodes in empty cells,
    nodes={text width={width(999)}, 
    align=right},
    right delimiter=\rbrack,left delimiter=\lbrack
  ] {
48 & 24 & -1 & -6 & -1 \\
24 & 95 & 12 & -22 & -12 & -1 \\
-1 & 12 & 73 & 0 & -23 & -12 & -1 \\
-6 & -22 & 0 & 72 & 0 & -23 & -12 & -1 \\
-1 & -12 & -23 & 0 & 72 & 0 & -23 & -12 & -1 \\
& \ddots &  &  &  & \ddots &  &  &  & \ddots \\
& & -1 & -12 & -23 & 0 & 72 & 0 & -23 & -12 & -1 \\
& & & -1 & -12 & -23 & 0 & 72 & 0 & -24 & -12 \\
& & & & -1 & -12 & -23 & 0 & 71 & -12 & -23 \\
& & & & & -1 & -12 & -24 & -12 & 49 & 0 \\
& & & & & & -1 & -12 & -23 & 0 & {\!36+\!1} & \;\;-1 & \\
& & & & & & & & &  -1 & 2 & -1 & & & & & & \\
& & & & & & & & &   & \ddots & \ddots & \ddots \\
& & & & & & & & &   &  & -1 & 2 & -1 \\
& & & & & & & & &   &  & & -1 & {1\!+\!36} & 0 & -23 & -12 & -1  \\
& & & & & & & & &   &  & &  & 0  & 49 & -12 & -24 & -12 & -1 \\
& & & & & & & & &   &  & &  & -23  & -12 & 71 & 0 & -23 & -12 & -1 \\
& & & & & & & & &   &  & &  &  -12 & -24 & 0 & 72 & 0 & -23 & -12 & -1 \\
& & & & & & & & & & & & &-1 & -12 & -23 & 0 & 72 & 0 & -23 & -12 & -1 \\
 & & & & & & & & & & & & & & \ddots &  &  &  & \ddots &  &  &  & \ddots \\
& & & & & & & & & & & & & & & -1 & -12 & -23 & 0 & 72 & 0 & -23 & -12 & -1 \\
& & & & & & & & & & & & & & &  & -1 & -12 & -23 & 0 & 72 & 0 & -22 & -6 \\
& & & & & & & & & & & & & & &  &  & -1 & -12 & -23 & 0 & 73 & 12 & -1 \\
& & & & & & & & & & & & & & &  &  &  & -1 & -12 & -22 & 12 & 95 & 24 \\
& & & & & & & & & & & & & & &  &  &  &  & -1 & 6 & -1 & 24 & 48 \\
} ;
\begin{pgfonlayer}{myback}
\fhighlight[azure!20]{m-5-1}{m-7-11}
\fhighlightL[azure!30]{m-18-14}{m-21-24}
\fhighlightL[awesome!20]{m-7-4}{m-11-11}
\fhighlightL[awesome!20]{m-14-14}{m-18-21}
\fhighlight[cadetgrey!30]{m-11-12}{m-11-12}
\fhighlightL[cadetgrey!30]{m-11-10}{m-14-14}
\fhighlightL[cadetgrey!30]{m-14-13}{m-15-13}
\fhighlight[asparagus!30]{m-1-1}{m-4-8}
\fhighlightL[asparagus!30]{m-21-17}{m-25-24}
\end{pgfonlayer}
\end{tikzpicture}
\caption{The matrix arising from the discretization of the NL-L-NL coupled problem with the choice of the flat-top kernel and a horizon of $\delta=3h$.  For clarity,  only the integer valued entries are shown and the scalings are skipped.  The actual values of the stiffens matrix are obtained by multiplying the local and nonlocal blocks with $\frac{1}{h}$ and $\frac{3}{\delta^3}\frac{h^2}{24}=\frac{1}{216 h}$, respectively}
\label{fig:NL-L-NL_mtrx}
\end{figure}

\section{Boundary Treatment in Local and Nonlocal Formulations} \label{sec:bdry_treatment_compa_cond}

In local problems, BCs are auxiliary constraints to the governing
equation.  Whereas in NL problems of interest, BCs are part of the
governing equation.  \magenta{Hence, local and NL problems are fundamentally
different in the way the boundary data enters the formulation.  In the
NL problem, the boundary data is already in the formulation before the
integration step of the weak form.}  Whereas in the local problem, it
enters the formulation after the integration step.  Let us recall local
weak formulation in detail.  First, one prepares the strong equation
to the weak formulation by multiplying with the test function $v$:
$$ 
-\Delta u(x) \, v(x) = f(x) \, v(x), \quad x \in \Omega.
$$
Integrate both sides:
$$
\int_\Omega -\Delta u(x) \, v(x) \diff x =  \int_\Omega f(x) \, v(x) \diff x.
$$
Apply integration by parts:
$$
\int_{\partial \Omega} -\nabla u(x) \cdot \bn(x) \, v(x) \diff s + \int_{\Omega} \nabla u(x) \cdot \nabla v(x) \diff x 
= \int_\Omega f(x) \, v(x) \diff x 
$$
Since the integration step is completed, the boundary data now enters the formulation:
$$
\begin{aligned}
\int_{\Omega} \nabla u(x) \cdot \nabla v(x) \diff x = & \int_{\partial \Omega} \nabla u(x) \cdot \bn(x) \, v(x) \diff s + \int_\Omega f(x) \, v(x) \diff x \\
= & \int_{\partial \Omega}  g(x) \, v(x) \diff s + \int_\Omega f(x) \, v(x) \diff x,
\end{aligned}
$$
where $g = \nabla u \cdot \bn$ is the prescribed Neumann data.  For
the NL formulation, the treatment on the boundary triggers an important
relation, namely, compatibility conditions.  We explain these next.

\subsection{Compatibility Conditions of the Nonlocal Problem} \label{subsec:compa_cond}
Consider the strong form of the NL problem with $\delta>0$:
\begin{equation} \label{NL_prob}
scl \; \M_\DN w(x) = f(x), \quad x \in \Omega.
\end{equation}
While the NL problem is posed for $x \in \Omega$, it has an
implication for $x \in \partial \Omega$.  This is due to the fact that
the convolution operator $\K_\DN$ has a continuous extension to the
boundary; see \eqref{cont_ext}.  Once the boundary data 
\begin{equation} \label{bdry_data_w}
\lim_{x \to a} w(x) ~~\text{and} ~~ \lim_{x \to b} w'(x)
\end{equation}
 are provided as part of the BVP, the operator $\M_\DN$ creates
compatibility conditions between the solution $w$ and the forcing
function $f$:
\begin{subequations} \label{compa_cond_DN}
\begin{alignat}{2}
\lim_{x \to a} scl \; \M_\DN w(x)  & = && \lim_{x \to a} f(x) \label{compa_cond_DN_strong_1}\\
\lim_{x \to b} \frac{\diff}{\diff x} scl \; \M_\DN w(x) & = && \lim_{x \to b} \frac{\diff f}{\diff x} (x) \label{compa_cond_DN_strong_2}.
\end{alignat}
\end{subequations}
Using the Hilbert-Schmidt\footnote{The Hilbert-Schmidt property allows for uniform convergence of series of functions.  Hence, limits can be interchanged, such as  those involving $\lim_{x \to \partial \Omega}$, to evaluate the boundary value.} property associated with the operator $\M_\DN$, we arrive at the compatibility conditions stated explicitly:
\begin{subequations} \label{compa_cond_DN_strong_explicit}
\begin{alignat}{2}
scl \, c \lim_{x \to a} w(x)  & = && \lim_{x \to a} f(x) \label{compa_cond_DN_strong_explicit_1}\\
scl \, c\lim_{x \to b} w'(x) & = && \lim_{x \to b} f'(x). \label{compa_cond_DN_strong_explicit_2}
\end{alignat}
\end{subequations}
Thus, the NL problem \eqref{NL_prob} is posed for $x \in \Omega$, but
due to compatibility conditions, it governs an equation for $x \in \overline{\Omega}$.
Recalling \eqref{scaling}, since we have 
$$scl \, c = \frac{3}{\delta^3} \, 2 \delta = \frac{6}{\delta^2},$$
it is more useful to rewrite \eqref{compa_cond_DN_strong_explicit} as
\begin{equation*} 
\begin{split}
\begin{aligned}
\lim_{x \to a} w(x) & = \delta^2/6 \lim_{x \to a} f(x) \\
\lim_{x \to b} w'(x) & = \delta^2/6 \lim_{x \to b} f'(x). 
\end{aligned}
\end{split}
\end{equation*}

Since the weak form of the governing equation \eqref{NL_prob} is
utilized, we study the compatibility conditions in weak form.  As an
initial step, one prepares the strong equation to the weak form by
multiplying with the test function $v$:
\begin{equation} \label{NL_equ_weak}
\big[scl \; \M_\DN w(x) \big] v(x) = \big[ f(x) \big] v(x), \quad x \in \Omega.
\end{equation}
Recalling the domain and range of the governing operator $\M_\DN$ in \eqref{dom_range}, we already assume that $w, f \in L^2(\Omega)$.  In the weak formulation, the inner product version of \eqref{NL_equ_weak} will eventually appear:
\begin{equation} \label{NL_equ_weak_inner_prod}
\braket{scl \; \M_\DN w| v} = \braket{f|v}.
\end{equation}
The self-adjointness of $\M_\DN$ forces us to guarantee the existence of the swapped version of \eqref{NL_equ_weak_inner_prod}:
\begin{equation*} 
\braket{w, scl \; \M_\DN v} = \braket{f| v}.
\end{equation*}
As a result, we also assume that the test function $v$ comes from
$L^2(\Omega)$, the domain of $\M_\DN$.  Similar to \eqref{bdry_data_w}, the boundary data of $v$, i.e.,
\begin{equation*} 
\lim_{x \to a} v(x) \quad \text{and} \quad \lim_{x \to b} v'(x)
\end{equation*}
should also be provided to set up the BVP in weak form.

For boundary data, the formulation demands that each function
appearing in \eqref{NL_equ_weak} and their derivatives have limits as
$x \to a$ and $x \to b$, respectively.  More precisely, for $\BC=\DN$,
since BCs employ $\lim_{x \to a}$ and $\lim_{x \to b}\frac{\diff}{\diff x}$, 
one would naturally expect the existence of
\begin{equation*}
\lim_{x \to a} w(x), \lim_{x \to a} v(x), \lim_{x \to a} f(x) \quad \text{and} \quad
\lim_{x \to b} w'(x), \lim_{x \to b} v'(x), \lim_{x \to b} f'(x).
\end{equation*}
But, due to the product rule applied as part of the Neumann condition at $x=b$, the formulation additionally demands the existence of 
\begin{equation} \label{additional_lims}
\lim_{x \to b} w(x), \lim_{x \to b} v(x), ~~\text{and} ~~\lim_{x \to b} f(x).
\end{equation}
Since all assumptions on the existence of limits are in place,  of we are ready to state the compatibility conditions in weak form:
\begin{subequations} \label{compa_cond_DN_weak}
\begin{alignat}{2}
\lim_{x \to a} scl \; \M_\DN w(x)  \, v(x) & = && \lim_{x \to a} f(x) \, v(x) \label{compa_cond_DN_weak_1} \\
\lim_{x \to b} \frac{\diff}{\diff x} \big[ scl \; \M_\DN w(x) \; v(x) \big] & = && \lim_{x \to b} \frac{\diff}{\diff x} \big[ f(x) \, v(x) \big]. \label{compa_cond_DN_weak_2}
\end{alignat}
\end{subequations}

To enable division, we further assume that the following boundary
limits of $v$,
$$\lim_{x \to a} v(x) ~~ \text{and} ~~ \lim_{x \to b} v(x),$$
are nonzero.  Otherwise the compatibility conditions are trivially
satisfied.  Distribute $\lim_{x \to a}$ to obtain
$$
\lim_{x \to a} scl \; \M_\DN w(x)  \lim_{x \to a} v(x) = \lim_{x \to a} f(x) \lim_{x \to a} v(x). 
$$ 
After division by $\lim_{x \to a} v(x)$ , we immediately see that
\eqref{compa_cond_DN_weak_1} reduces to
\eqref{compa_cond_DN_strong_1}.  

The compatibility condition with the derivative is more involved.
Before taking $\lim_{x \to b}$, apply the differentiation product rule
in \eqref{compa_cond_DN_weak_2}:
$$
\frac{\diff}{\diff x} \big[ scl \, \M_\DN w(x) \big] \, v(x) + 
\big[ scl \; \M_\DN w(x) \big] \, v'(x) = f'(x) \, v(x) + f(x) v'(x).
$$
Distribute $\lim_{x \to b}$ to obtain
\begin{eqnarray*}
&& \lim_{x \to b} \frac{\diff}{\diff x} \big[ scl \, \M_\DN w(x) \big] \lim_{x \to b} v(x) + 
\lim_{x \to b} \big[ scl \; \M_\DN w(x) \big] \lim_{x \to b} v'(x) = \\
&& \lim_{x \to b} f'(x) \lim_{x \to b} v(x) + \lim_{x \to b} f(x) \lim_{x \to b} v'(x).
\end{eqnarray*}
Using \eqref{NL_prob}, one obtains
\begin{equation*} 
\lim_{x \to b} \big[ c \, scl \, w'(x) \big] \lim_{x \to b} v(x) + 
\lim_{x \to b} f(x) \lim_{x \to b} v'(x) =
\lim_{x \to b} f' \lim_{x \to b} v(x) + \lim_{x \to b} f(x) \lim_{x \to b} v'(x).
\end{equation*}
After cancellation, one arrives at
\begin{equation*}
\lim_{x \to b} \big[ c \, scl \, w'(x) \big] \lim_{x \to b} v(x) = 
\lim_{x \to b} f'(x) \lim_{x \to b} v(x).
\end{equation*}
After division by $\lim_{x \to b} v(x)$ , we immediately see that 
\eqref{compa_cond_DN_weak_2} reduces to \eqref{compa_cond_DN_strong_explicit_2}.  
Consequently, we showed that the compatibility conditions in weak and strong forms are identical. 

\section{Rectification with Harmonic Functions} \label{sec:rectification}
In a coupling problem, both the BCs of the local problem and the
forcing function $f$ are given.  The BC values are not necessarily the
boundary values of $f$.  This situation creates a discrepancy for the
NL problem.  As explained in Sec.~\ref{subsec:compa_cond}, the BCs
are dictated by the boundary values of $f$.  In order to enforce the
BCs of the local problem, the NL solution needs, what we call,
a rectification.

Consider the local problem with mixed BCs
\begin{equation*} 
\left\{ \begin{aligned} 
-\Delta u & = f \\ u(a) & = \alpha \\ u'(b) & = \beta. 
\end{aligned} \right.
\end{equation*}
To motivate the rectification process in the NL case, we present an analog scenario in the local setting.  Assume that we find ourselves in a situation that only homogeneous BCs are allowed for the computation of $u$.  This would obviously lead to a ``wrong'' solution, which we denote as $w$ of the following problem:
\begin{equation} \label{wrong_solu_local}
\left\{ \begin{aligned}
-\Delta w & = f \\ w(a) & = 0 \\ w'(b) & = 0.
\end{aligned} \right.
\end{equation}
\magenta{One important question arises: Is it possible to rectify $w$ to obtain $u$?  The answer is yes, and it is due to a well-known decomposition of $u$.}

Assume that we solve the following additional problem with the correct BC:
\begin{equation*} 
\left\{ \begin{aligned} 
-\Delta H & = 0 \\ H(a) & = \alpha \\ H'(b) & = \beta.
\end{aligned} \right.
\end{equation*}
The function $H$ is called  the harmonic extension\footnote{The discrete harmonic extension is an important idea in domain decomposition and is heavily used in the DDM literature 
\cite{Quarteroni:1999:DDforPDEs,Toselli:2005:DDBook}.}.
When the solution $H$ is added to $w$, the sum gives the correct solution.  Hence, we say that $H$ rectifies the wrong solution $w$.  We essentially utilized the aforementioned decomposition of $u$:
\begin{equation} \label{u_decomp_local}
u = w + H.
\end{equation}

In the light of the decomposition \eqref{u_decomp_local}, we state the involved problems together:
\begin{equation} \label{local_prob_decomp}
\left\{ 
\begin{aligned} -\Delta u & = f \\ u(a) & = \alpha \\ u'(b) & = \beta \end{aligned} 
\right. ~ = ~ 
\left\{ 
\begin{aligned} -\Delta w & = f \\ w(a) & = 0 \\ w'(b) & = 0 \end{aligned} 
\right. ~ + ~ 
\left\{ 
\begin{aligned} -\Delta H & = 0 \\ H(a) & = \alpha \\ H'(b) & = \beta. \end{aligned} 
\right.
\end{equation}
The proof of the decomposition \eqref{u_decomp_local} is simply due to the linearity of $-\Delta_\DN$.  Next, we explain how one can rectify the NL solution.

\subsection{Rectification of the Nonlocal Solution}

Denote the local operator with mixed BCs by $-\Delta_\DN$.  The
self-adjoint governing operator $\M_\DN$ is constructed by using
functional calculus and is a function of the local operator, i.e.,
$\M_\DN = F(-\Delta_\DN)$ for some bounded function $F$ defined on the
spectrum of $-\Delta_\DN$.  Roughly speaking, the scaled operator $scl
\, \M_\DN$ is constructed as a generalization of $-\Delta_\DN$.  Our
rectification process is based on this observation.  In fact, for the
eigenvalues $\lambda_k, ~k=1, 2, \ldots$ of $\M_\DN$,  one can
show that \cite[Sec.~4]{aksoyluGazonas2020_dispersion}
$$\lim_{\delta \to 0} \lambda_k(scl \; \M_\DN) = \lambda_k(-\Delta_\DN).$$

In \eqref{local_prob_decomp}, we replace the problem
\eqref{wrong_solu_local} with its NL generalization using the same BC
as $\delta \to 0$ in the following way:
\begin{equation} \label{NL_prob_decomp}
\left\{ 
\begin{aligned} -\Delta u & =  f \\ u(a) & =  \alpha \\ u'(b) & =  \beta \end{aligned} 
\right. ~ \approx ~
\left\{ \begin{aligned}
scl \; \M_\DN w & =  f \\
w(a) & =  \delta^2/6 \, f(a) \\
w'(b) & =  \delta^2/6 \, f'(b)
\end{aligned} \right.  ~ + ~
\left\{ 
\begin{aligned} -\Delta H & =  0 \\ H(a) & =  \alpha \\ H'(b) & =  \beta. \end{aligned} 
\right.
\end{equation}
Note that the both the Dirichlet and Neumann BCs approach zero as $\delta \to 0$, i.e.,
$$
\lim_{\delta \to 0} w(a) = \lim_{\delta \to 0} \delta^2/6 \, f(a) = 0 \quad \text{and} \quad
\lim_{\delta \to 0} w'(b) =  \lim_{\delta \to 0} \delta^2/6 \, f'(b) = 0.
$$
\magenta{Consequently, the rationale of our rectification process can be stated as follows:  As $\delta \to 0$, the ``wrong'' NL problem in \eqref{NL_prob_decomp} converges to the
``wrong'' local problem in \eqref{local_prob_decomp}, hence, rectify it with a harmonic extension.}

So far, we discussed the rectification process in strong form.  Since we are interested in the weak form, we utilize the weak version of the decomposition in \eqref{u_decomp_local}:
\begin{equation} \label{harmonic_decomp_weak} 
\big(w(x)+H(x)\big)  v(x) = u(x)  v(x).
\end{equation}
Since $H$ is a harmonic function, it is a linear polynomial in 1D.  Define
$$H(x) := -(a_0 + a_1 x).$$
We prefer to insert a minus sign for ease of algebra, which soon is going to become clear. 

For the Dirichlet BC, the test function associated with $x=a$ is $v=\phi_a$.  Apply the Dirichlet BC to \eqref{harmonic_decomp_weak} and obtain:
\begin{eqnarray*}
\big(w(a) + H(a)\big) \phi_a(a) & = & u(a) \phi_a(a) \\
\delta^2/6 \, f(a) + H(a) & = & u(a),
\end{eqnarray*}
which implies
\begin{equation} \label{rec_lhs_Dir}
a_0 + a \, a_1 = \delta^2/6 \, f(a) - u(a).
\end{equation}

For the Neumann BC, the test function associated with $x=b$ is $v=\phi_b$.  Apply the Neumann BC to \eqref{harmonic_decomp_weak} and obtain:
\begin{equation*}
\frac{\diff}{\diff x}\big[\big( w(x) + H(x)\big) \phi_b(x) \big]\big|_{x=b} =
\frac{\diff}{\diff x}\big[u(x) \phi_b(x) \big] \big|_{x=b}.
\end{equation*}

Apply the product rule:
$$
\big[w'(x) + H'(x) \big] \phi_b(x) + \big[w(x) + H(x) \big] \phi_b'(x) = 
u'(x) \phi_b(x) + u(x) \phi_b'(x)
$$
Evaluate at $x=b$ and substitute $w'(b) = \delta^2/6 f'(b)$:
$$
\big[\delta^2/6 f'(b)+ H'(b) \big] + \big[w(b) + H(b) \big] \frac{1}{h} 
= u'(b) + u(b) \frac{1}{h},
$$
which implies
\begin{equation} \label{rec_lhs_Neu}
\frac{1}{h} a_0 + (1+\frac{b}{h}) a_1  = \big[\delta^2/6 f'(b) - u'(b)\big] + \big[w(b)-u(b)\big]/h.
\end{equation}
Combining \eqref{rec_lhs_Dir} and \eqref{rec_lhs_Neu}
\begin{equation*} 
\begin{bmatrix*}[c] 1 & a \\ & \\ \displaystyle\frac{1}{h} & 1+\displaystyle\frac{b}{h} \end{bmatrix*}
\begin{bmatrix*}[c] a_0 \\ a_1 \end{bmatrix*} = 
\begin{bmatrix*}[l] \delta^2/6 \, f(a) -\alpha \\ \\ \big[\delta^2/6 \, f'(b) - \beta \big]
+ \big[w(b) -u(b) \big]/h \end{bmatrix*}.
\end{equation*}
Rectification takes places after the solution $w$ is obtained, hence, the value $w(b)$ is known.  Since at $x=b$, the Neumann BC is enforced, the value $u(b)$ is not known.  Rectification in weak form demands this extra information; see \eqref{additional_lims}.

\subsection{The Case of $\BC=\NN$}
Solving a pure Neumann problem is more challenging than solving a
problem with at least one Dirichlet BC.  The operator $\M_\NN$ has a
nontrivial kernel which is spanned by the eigenfunction
\begin{equation} \label{e0_NN}
e_0^\NN(x) = \sqrt{\frac{1}{L}}.
\end{equation}
Namely,
\begin{equation*} 
\text{ker}(\M_\NN) = \text{span}(e_0^\NN).
\end{equation*}
Denote the eigenfunctions of the $\M_\NN$ by $e_k^\NN$  and define the space
\begin{equation} \label{one_perp}
\Lo^2(\Omega) := \big(e_0^\NN\big)^\perp = \text{span}(\big\{e_k^\NN: k \in \mathbb{N} \setminus \{0\} \big\}).
\end{equation}
Thanks to \eqref{e0_NN} and \eqref{one_perp}, the operator $\M_\NN$ is
a bijection if its domain and range are restricted to functions that
are orthogonal to the constant function.  See the unisolvent
discussion in \cite{aksoyluCelikerDiehl2024_implementation}.
Consequently, for the case of $\BC=\NN$, we end up with the following
the rectification system:
$$\begin{bmatrix*}[c] -\displaystyle\frac{1}{h} & 1 - \displaystyle\frac{a}{h} \\ & \\ \displaystyle\frac{1}{h} & 1+\displaystyle\frac{b}{h} \end{bmatrix*}
\begin{bmatrix*}[c] a_0 \\ a_1 \end{bmatrix*} = 
\begin{bmatrix*}[l] \big[\delta^2/6 \, f'(a) - \alpha \big]
+ \big[w(a) -u(a) \big]/h \\ \\ \big[\delta^2/6 \, f'(b) - \beta \big]
+ \big[w(b) -u(b) \big]/h \end{bmatrix*}.
$$
For the pure Neumann problem, the rectification process calls for the function values of the solution, i.e., $u(a)$ and $u(b)$, at the boundary.  The need for these values is an inevitable additional cost of compatibility conditions in weak form; see Remark~\ref{rem:no_free_lunch_thm}.

\section{Numerical Experiments} \label{sec:numerical_experiments}

\begin{table}[t]
\caption{Number of elements employed on the domain $(a,b) = (-2.00,3.25)$ for 3-subdomain, 2-subdomain, and a single domain configuration and the corresponding grid sizes}
\label{table:no_elements_grid_size}
\begin{center}
\begin{tabular}{cc@{\hspace{.2cm}}cc@{\hspace{.2cm}}cc@{\hspace{.2cm}}cc}
\toprule
& \multicolumn{3}{c}{3SD} &\multicolumn{2}{c}{2SD} &\multicolumn{1}{c}{1SD} 
\vspace{-.1in} \\
& \uline{3} &\uline{2} &\uline{1}\\
Grid &$N_1$ &$N_2$ &$N_3$ &$N_1$ &$N_2$ &$N$ &$h$ \\
\toprule
4 &64     &64     &64     &96     &96     &192   &2.73E-02 \\
5 &128   &128   &128   &192   &192   &384   &1.37E-02 \\
6 &256   &256   &256   &384   &384   &768   &6.84E-03 \\
7 &512   &512   &512   &768   &768   &1536 &3.42E-03 \\
8 &1024 &1024 &1024 &1536 &1536 &3072 &1.71E-03 \\
\bottomrule
\end{tabular}
\end{center}
\end{table}

For discretization of the local and NL problems, the finite element
method and the Galerkin projection are utilized, respectively, with a
nodal linear basis.  We compute an approximate solution $u_h$ and
report the $L^2$-norm of the error $u-u_h$.  The ``Grid'' column
indicates the number of elements used in the discretization.  More
explicitly, Grid = $i$ contains $24 \times 2^{i-1}$, $i=4, \ldots, 8$,
elements.  In coupled problems, subdomains equally share the elements
of the single domain.  More precisely, in 3-subdomain coupled
configurations, subdomain $\Omega_k, k=1, 2, 3$ contains $8 \times
2^{i-1}$ elements. In 2-subdomain coupled configurations, subdomain
$\Omega_k, k=1, 2$ contains $12 \times 2^{i-1}$ elements. For a fair
comparison, we kept the grid size equal for all configuration at all
levels.  See the details in Table~\ref{table:no_elements_grid_size}.

The exact solutions in each BC case are given below:
\begin{equation*}
\begin{aligned}
&\DD: && u(x) = \cos(\pi x) \\
&\DN: && u(x) = \exp(x) + \exp(-x) \\
&\NN1: && u(x) = \cos(\pi x) - \beta_1 \sin(2 \pi x) - \beta_2 \cos(3 \pi x) \\
&\NN2: && u(x) = \sqrt{2/L} \cos(2 \pi (a-x) /L),
\end{aligned}
\end{equation*}
where
\begin{equation*}
\beta_1 = \frac{-16 \big(\sin(\pi a) - \sin(\pi b)\big)}{5 \big(\cos(2 \pi a) - \cos(2 \pi b) \big)}  \quad \text{and} \quad
\beta_2 =  \frac{-9 \big(\sin(\pi a) - \sin(\pi b)\big)}{5 \big(\sin(3 \pi a) - \sin(3 \pi b)\big)},
\end{equation*} 
which guarantees the orthogonality to the constant function required
by the pure Neumann problem.  

We elaborate on the design of the experiments.  In order to establish
that our coupling method works for an arbitrary solution, we chose exact
solutions from different families of functions.  The exact solution in
test cases $\DD$ and $\DN$ are oscillatory and have exponential growth,
respectively.  The pure Neumann test cases $\NN1$ and $\NN2$ are
designed in such a way that both the exact solution and the forcing
function are orthogonal to the constant function, thereby,
guaranteeing a solution to the singular system.  The exact solution in
test case $\NN1$ is a nontrivial sinusoidal function.  The function in
$\NN2$ is an eigenfunction of the operator $\M_\NN$.

We report the numerical experiment in Tables \ref{tab:L-NL-L}, \ref{tab:NL-L-NL}, \ref{tab:NL-NL}, \ref{tab:L-L} and \ref{tab:NL}.  The corresponding solutions are shown in Figures \ref{fig:L-NL-L_coupling_DD1_DN1}, \ref{fig:L_NL_L_coupling_NN1_NN2}, \ref{fig:NL_L_NL_coupling_DD1_DN1}, \ref{fig:NL-L-NL_coupling_NN1_NN2}, \ref{fig:NL_NL_coupling_DD1_DN1}, \ref{fig:NL_NL_coupling_NN1_NN2}.  The kernel of choice is the flat-top given in \eqref{flat_top_kernel} with a horizon of $\delta=3h$.

\begin{table}[htbp]
\caption{History of convergence of L-NL-L coupling with $(a,e_{12},e_{23},b) = (-2.00,-0.25,1.50,3.25)$}
\label{tab:L-NL-L}
\begin{center}
\begin{tabular}{ccc@{\hspace{.2cm}}cc@{\hspace{.2cm}}cc@{\hspace{.2cm}}cc@{\hspace{.2cm}}cc}
\toprule
& \multicolumn{2}{c}{$\DD$} &\multicolumn{2}{c}{$\DN$} &\multicolumn{2}{c}{$\NN1$} &\multicolumn{2}{c}{$\NN2$}
\vspace{-.1in} \\
& \uline{2} &\uline{2} &\uline{2} &\uline{2}\\
Grid &Error &Rate &Error &Rate &Error &Rate &Error &Rate\\
\toprule
4 &2.22E-01 &0.95 &5.47E-01 &1.01 &2.89E+00 &1.04 &7.45E-02 &1.00 \\
5 &1.13E-01 &0.97 &2.72E-01 &1.01 &1.43E+00 &1.02 &3.73E-02 &1.00 \\
6 &5.72E-02 &0.99 &1.36E-01 &1.00 &7.12E-01 &1.01 &1.87E-02 &1.00 \\
7 &2.87E-02 &0.99 &6.78E-02 &1.00 &3.55E-01 &1.00 &9.33E-03 &1.00 \\
8 &1.44E-02 &1.00 &3.39E-02 &1.00 &1.77E-01 &1.00 &4.67E-03 &1.00 \\
\bottomrule
\end{tabular}
\end{center}
\end{table}

\begin{table}[htbp]
\caption{History of convergence of NL-L-NL coupling with $(a,e_{12},e_{23},b) = (-2.00,-0.25,1.50,3.25)$}
\label{tab:NL-L-NL}
\begin{center}
\begin{tabular}{ccc@{\hspace{.2cm}}cc@{\hspace{.2cm}}cc@{\hspace{.2cm}}cc@{\hspace{.2cm}}cc@{\hspace{.2cm}}cc}
\toprule
& \multicolumn{2}{c}{$\DD$} &\multicolumn{2}{c}{$\DN$} &\multicolumn{2}{c}{$\NN1$} &\multicolumn{2}{c}{$\NN2$}
\vspace{-.1in} \\
& \uline{2} &\uline{2} &\uline{2} &\uline{2}\\
Grid &Error &Rate &Error &Rate &Error &Rate &Error &Rate\\
\toprule
4 &2.56E-01 &1.00 &2.00E+00 &1.00 &3.42E+00 &0.99 &9.22E-02 &1.01 \\
5 &1.28E-01 &1.00 &1.00E+00 &1.00 &1.73E+00 &0.99 &4.59E-02 &1.01 \\
6 &6.41E-02 &1.00 &5.00E-01 &1.00 & 8.67E-01 &0.99 &2.29E-02 &1.00 \\
7 &3.21E-02 &1.00 &2.50E-01 &1.00 &4.35E-01 &1.00 &1.14E-02 &1.00 \\
8 &1.60E-02 &1.00 &1.25E-01 &1.00 &2.18E-01 &1.00 &5.72E-03 &1.00 \\
\bottomrule
\end{tabular}
\end{center}
\end{table}

\begin{table}[htbp]
\caption{History of convergence of NL-NL coupling with $(a,e_{12},b) = (-2.00,0.62,3.25)$}
\label{tab:NL-NL}
\begin{center}
\begin{tabular}{ccc@{\hspace{.2cm}}cc@{\hspace{.2cm}}cc@{\hspace{.2cm}}cc@{\hspace{.2cm}}cc}
\toprule
& \multicolumn{2}{c}{$\DD$} &\multicolumn{2}{c}{$\DN$} &\multicolumn{2}{c}{$\NN1$} &\multicolumn{2}{c}{$\NN2$}
\vspace{-.1in} \\
& \uline{2} &\uline{2} &\uline{2} &\uline{2}\\
Grid &Error &Rate &Error &Rate &Error &Rate &Error &Rate\\
\toprule
4 &3.48E-01 &0.92 &3.49E+00 &1.00 & 3.00E+00 &1.05 &8.40E-04 &2.00 \\
5 &1.79E-01 &0.96 &1.75E+00 &1.00 &1.49E+00 &1.01 & 2.10E-04 &2.00 \\
6 &9.08E-02 &0.98 &8.73E-01 &1.00 & 7.41E-01 &1.00 & 5.25E-05 &2.00 \\
7 &4.57E-02 &0.99 &4.36E-01 &1.00 & 3.70E-01 &1.00 & 1.31E-05 &2.00 \\
8 &2.30E-02 &0.99 &2.18E-01 &1.00 & 1.85E-01 &1.00 & 3.28E-06 &2.00 \\
\bottomrule
\end{tabular}
\end{center}
\end{table}

\begin{table}[htbp]
\caption{History of convergence of L-L coupling with $(a,e_{12},b) = (-2.00,0.62,3.25)$}
\label{tab:L-L}
\begin{center}
\begin{tabular}{ccc@{\hspace{.2cm}}cc@{\hspace{.2cm}}cc@{\hspace{.2cm}}cc@{\hspace{.2cm}}cc}
\toprule
& \multicolumn{2}{c}{$\DD$} & \multicolumn{2}{c}{$\DN$} &\multicolumn{2}{c}{$\NN1$} &\multicolumn{2}{c}{$\NN2$}
\vspace{-.1in} \\
& \uline{2} &\uline{2} &\uline{2} &\uline{2}\\
Grid &Error &Rate &Error &Rate &Error &Rate &Error &Rate\\
\toprule
4 &2.01E-03 &2.00 &7.30E-03 &2.00 & 9.42E-02 &2.00 &1.78E-04 &2.00 \\
5 &5.04E-04 &2.00 &1.83E-03 &2.00 & 2.35E-02 &2.00 &4.46E-05 &2.00 \\
6 &1.26E-04 &2.00 &4.56E-04 &2.00 & 5.88E-03 &2.00 &1.12E-05 &2.00 \\
7 &3.15E-05 &2.00 &1.14E-04 &2.00 & 1.47E-03 &2.00 &2.79E-06 &2.00 \\
8 &7.87E-06 &2.00 &2.85E-05 &2.00 & 3.67E-04 &2.00 &6.97E-07 &2.00 \\
\bottomrule
\end{tabular}
\end{center}
\end{table}

\begin{table}[htbp]
\caption{History of convergence of the nonlocal problem with $(a,b) = (-2.00,3.25)$}
\label{tab:NL}
\begin{center}
\begin{tabular}{ccc@{\hspace{.2cm}}cc@{\hspace{.2cm}}cc@{\hspace{.2cm}}cc@{\hspace{.2cm}}cc}
\toprule
& \multicolumn{2}{c}{$\DD$} &\multicolumn{2}{c}{$\DN$} &\multicolumn{2}{c}{$\NN1$} &\multicolumn{2}{c}{$\NN2$}
\vspace{-.1in} \\
& \uline{2} &\uline{2} &\uline{2} &\uline{2} \\
Grid &Error &Rate &Error &Rate &Error &Rate &Error &Rate\\
\toprule
4 &1.29E-02 &1.99 &1.75E-02 &2.07 &1.83E-01 &2.24 &8.40E-04 &2.00 \\
5 &3.23E-03 &2.00 &4.26E-03 &2.04 &4.23E-02 &2.11 &2.10E-04 &2.00 \\
6 &8.09E-04 &2.00 &1.05E-03 &2.02 &1.02E-02 &2.05 &5.25E-05 &2.00 \\
7 &2.02E-04 &2.00 &2.61E-04 &2.01 &2.51E-03 &2.02 &1.31E-05 &2.00 \\
8 &5.06E-05 &2.00 &6.50E-05 &2.00 &6.22E-04 &2.01 &3.28E-06 &2.00 \\
\bottomrule
\end{tabular}
\end{center}
\end{table}

\begin{figure}[tb]
\centering
\begin{subfigure}[b]{0.47\textwidth}
\includegraphics[width=\textwidth]{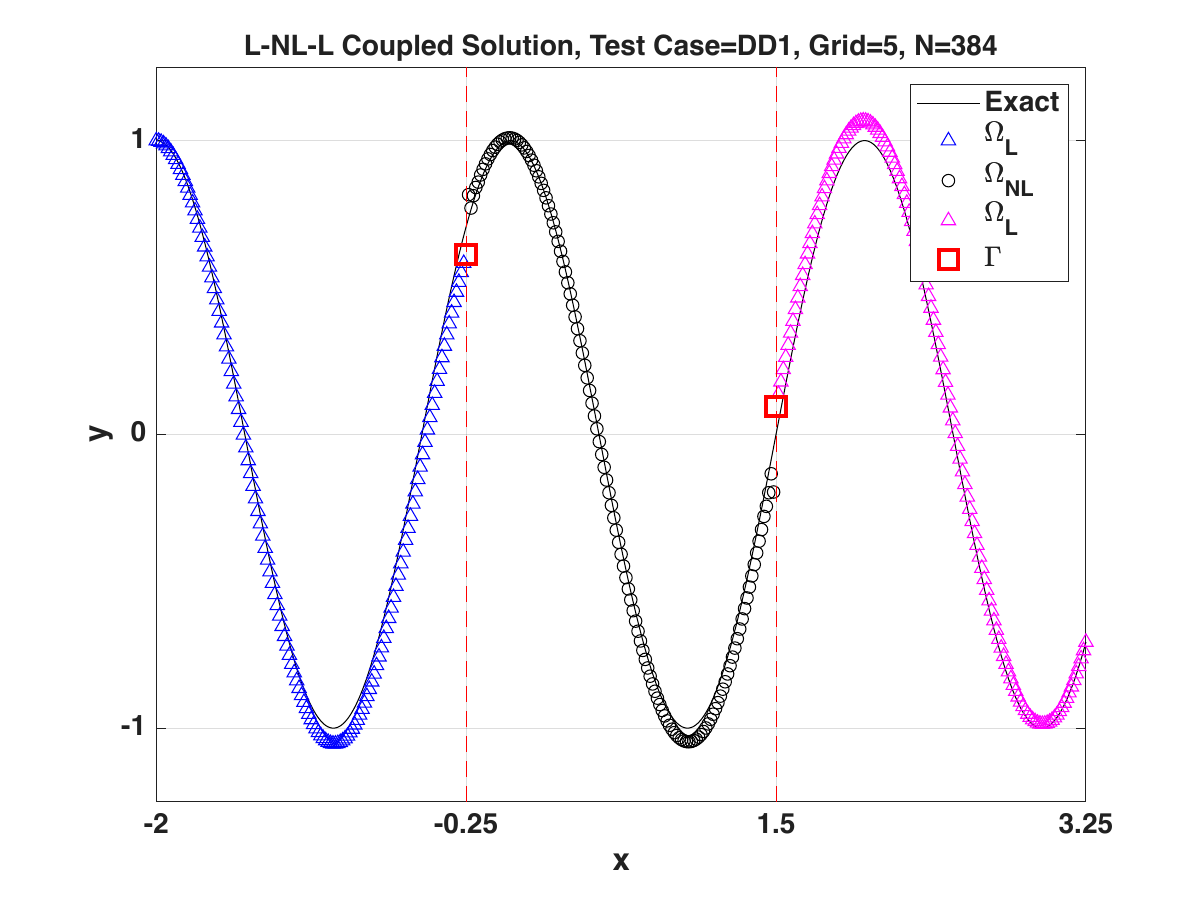}
\end{subfigure}
\begin{subfigure}[b]{0.47\textwidth}
\includegraphics[width=\textwidth]{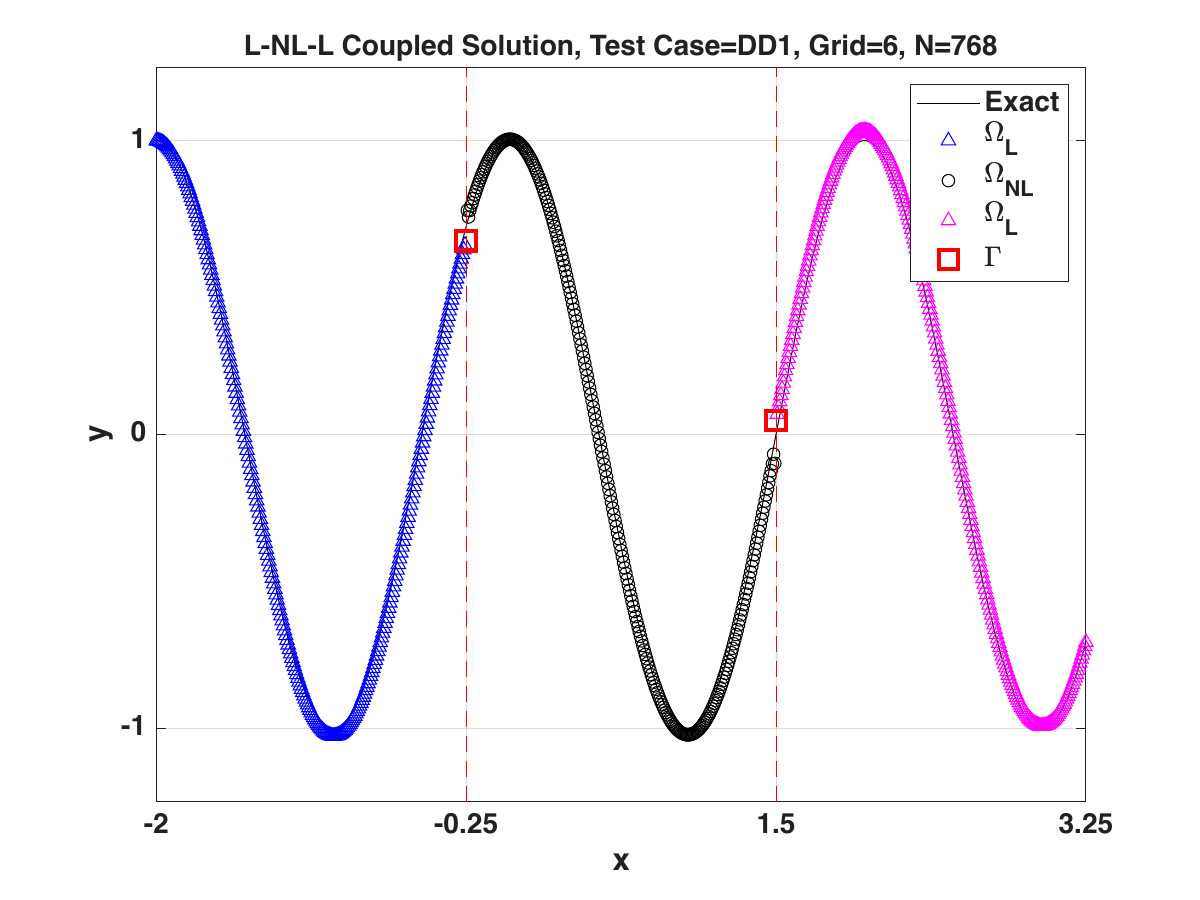}
\end{subfigure}
\begin{subfigure}[b]{0.47\textwidth}
\includegraphics[width=\textwidth]{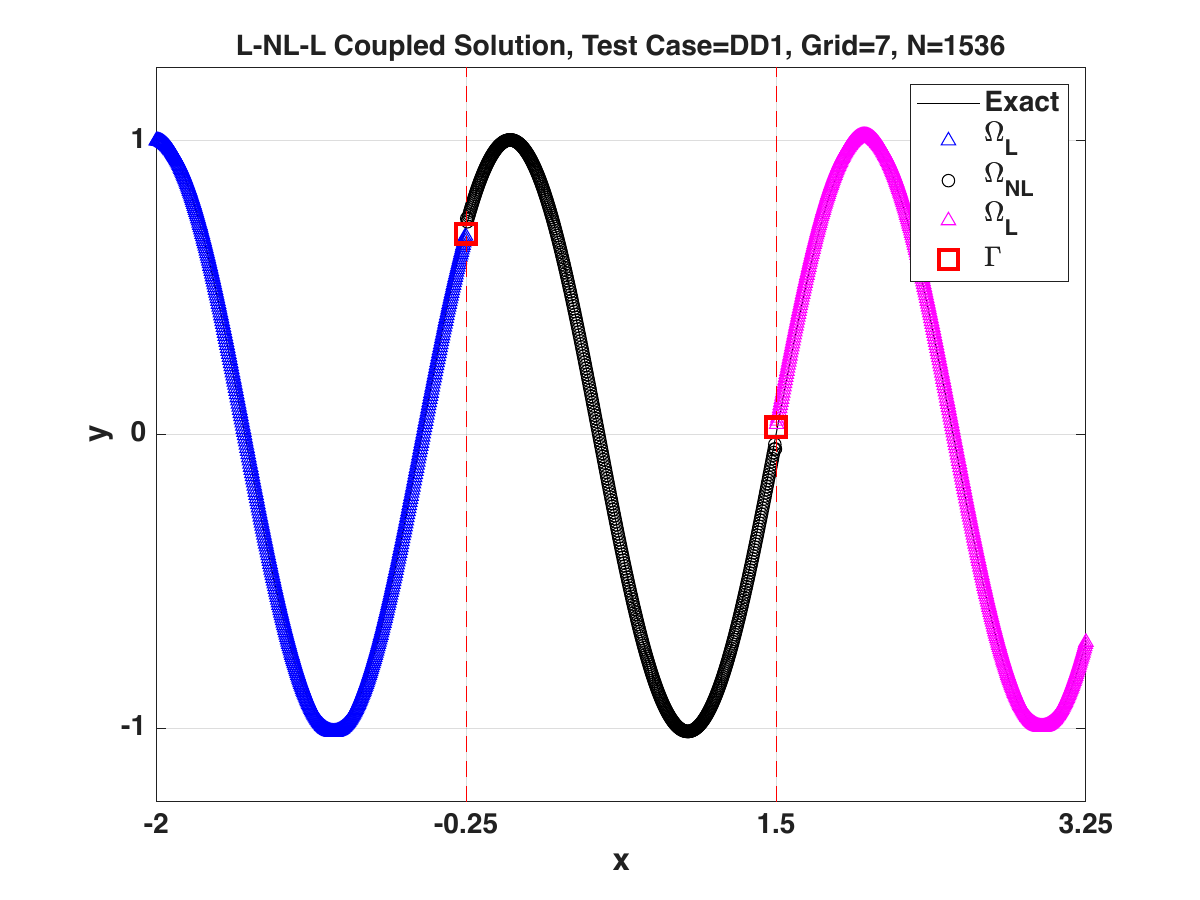}
\end{subfigure}
\begin{subfigure}[b]{0.47\textwidth}
\includegraphics[width=\textwidth]{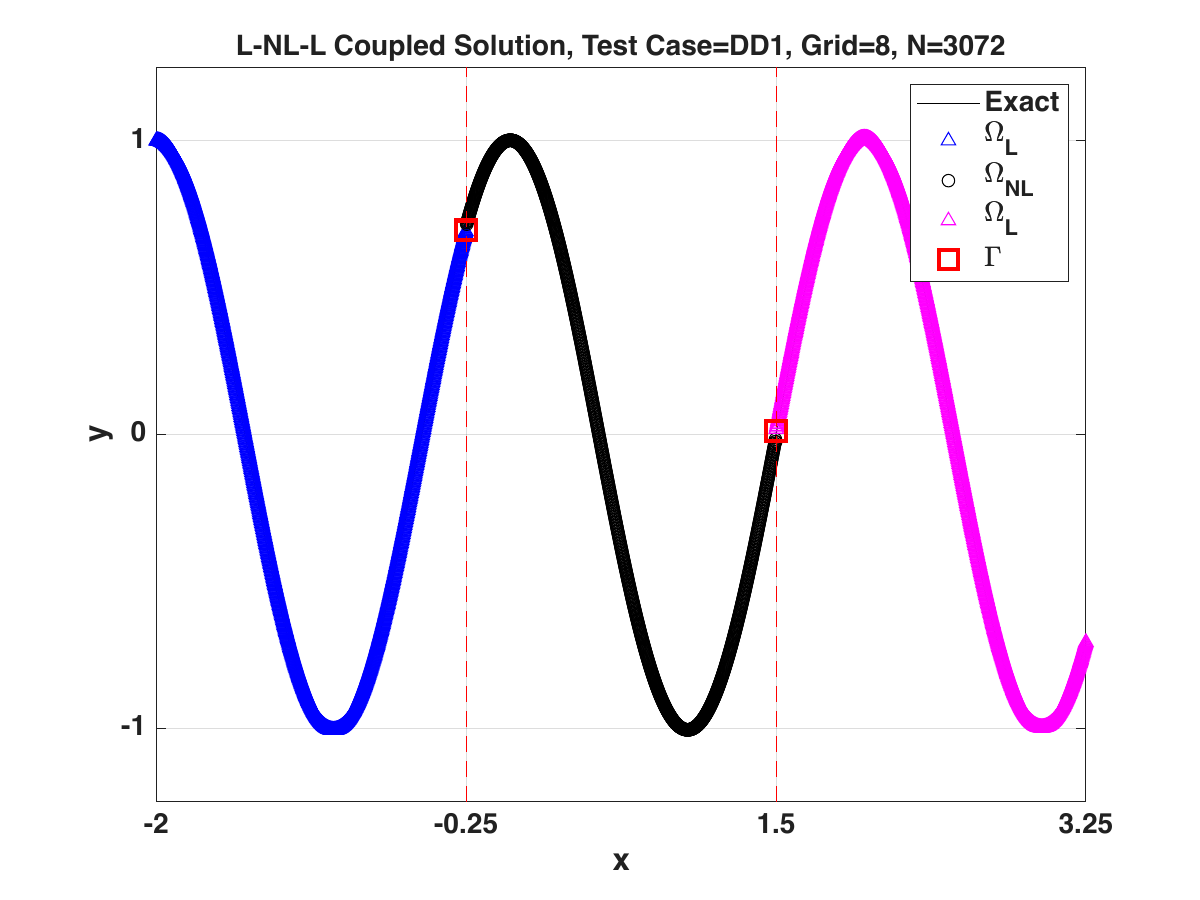}
\end{subfigure}
\begin{subfigure}[b]{0.47\textwidth}
\includegraphics[width=\textwidth]{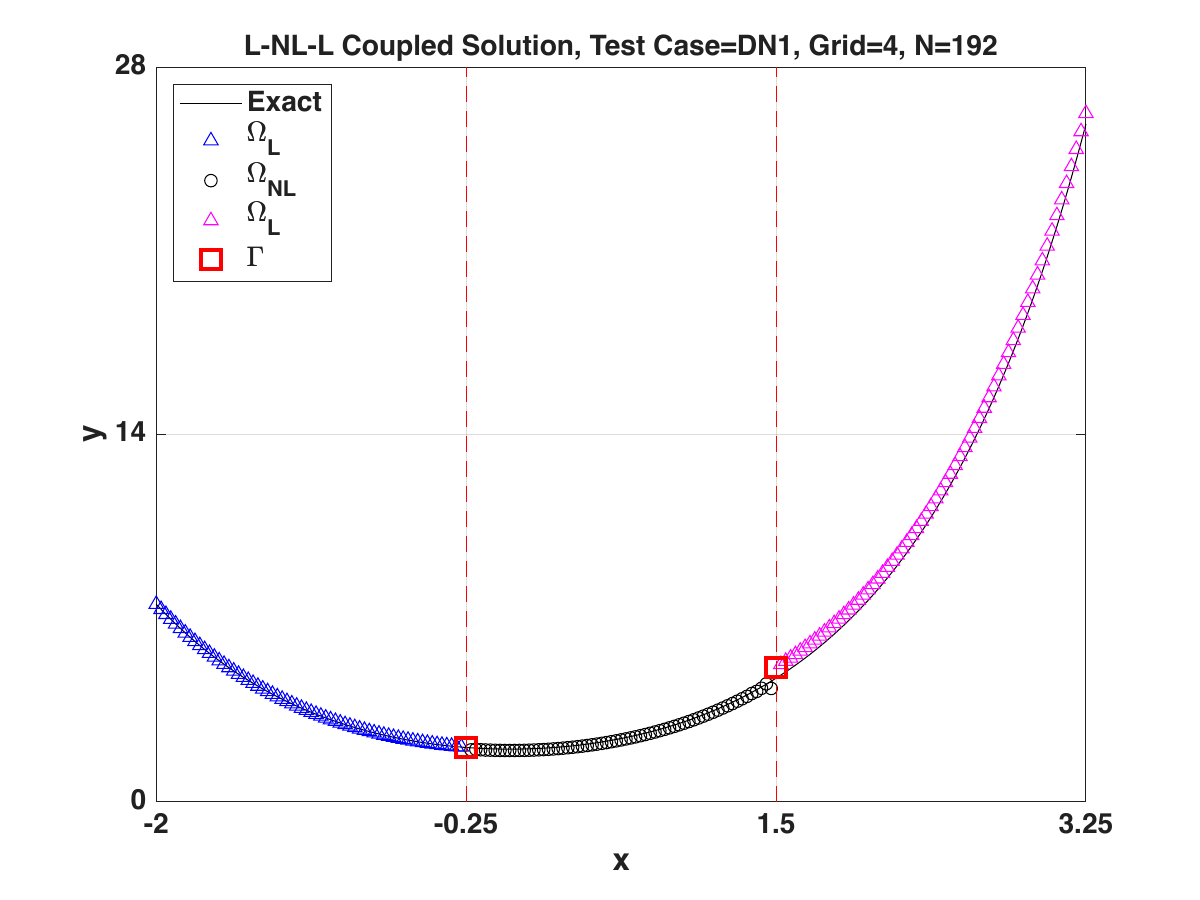}
\end{subfigure}
\begin{subfigure}[b]{0.47\textwidth}
\includegraphics[width=\textwidth]{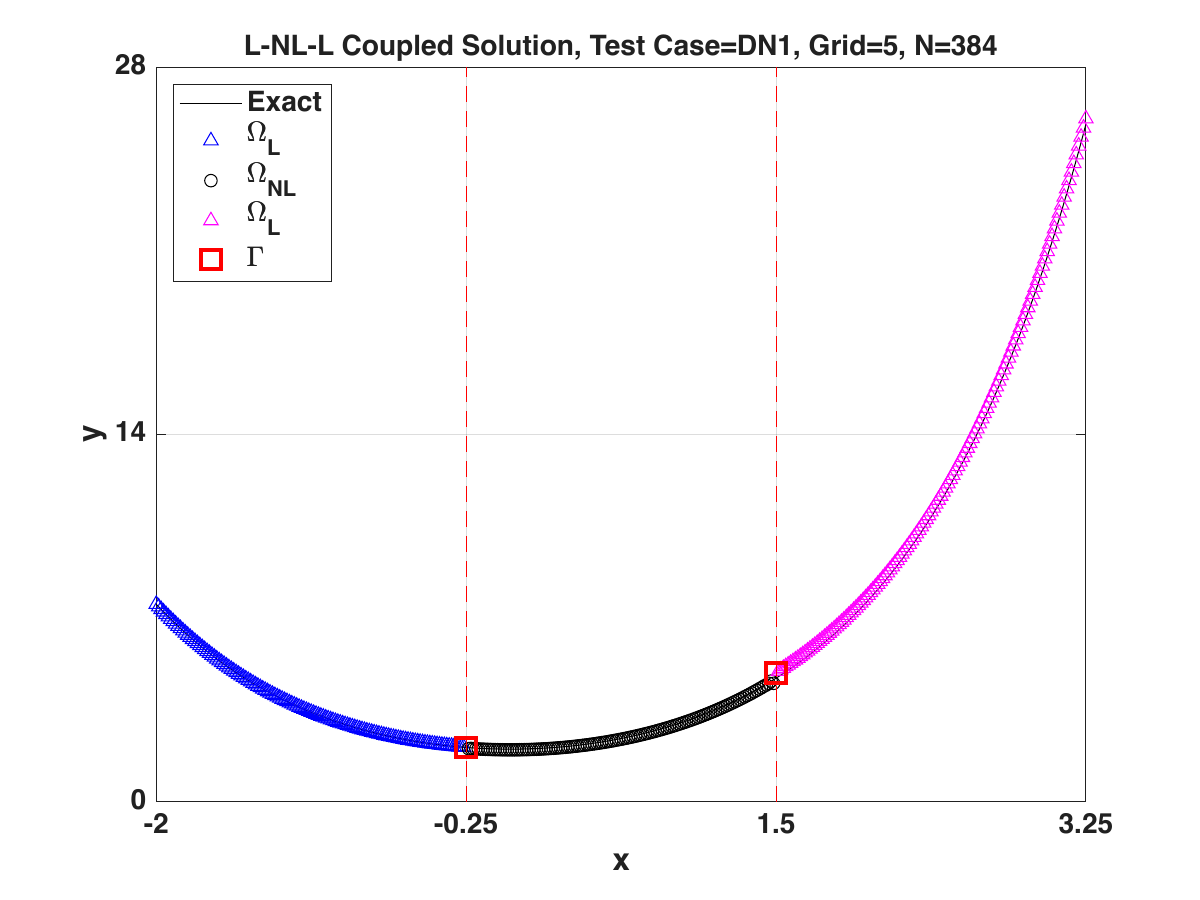}
\end{subfigure}
\begin{subfigure}[b]{0.47\textwidth}
\includegraphics[width=\textwidth]{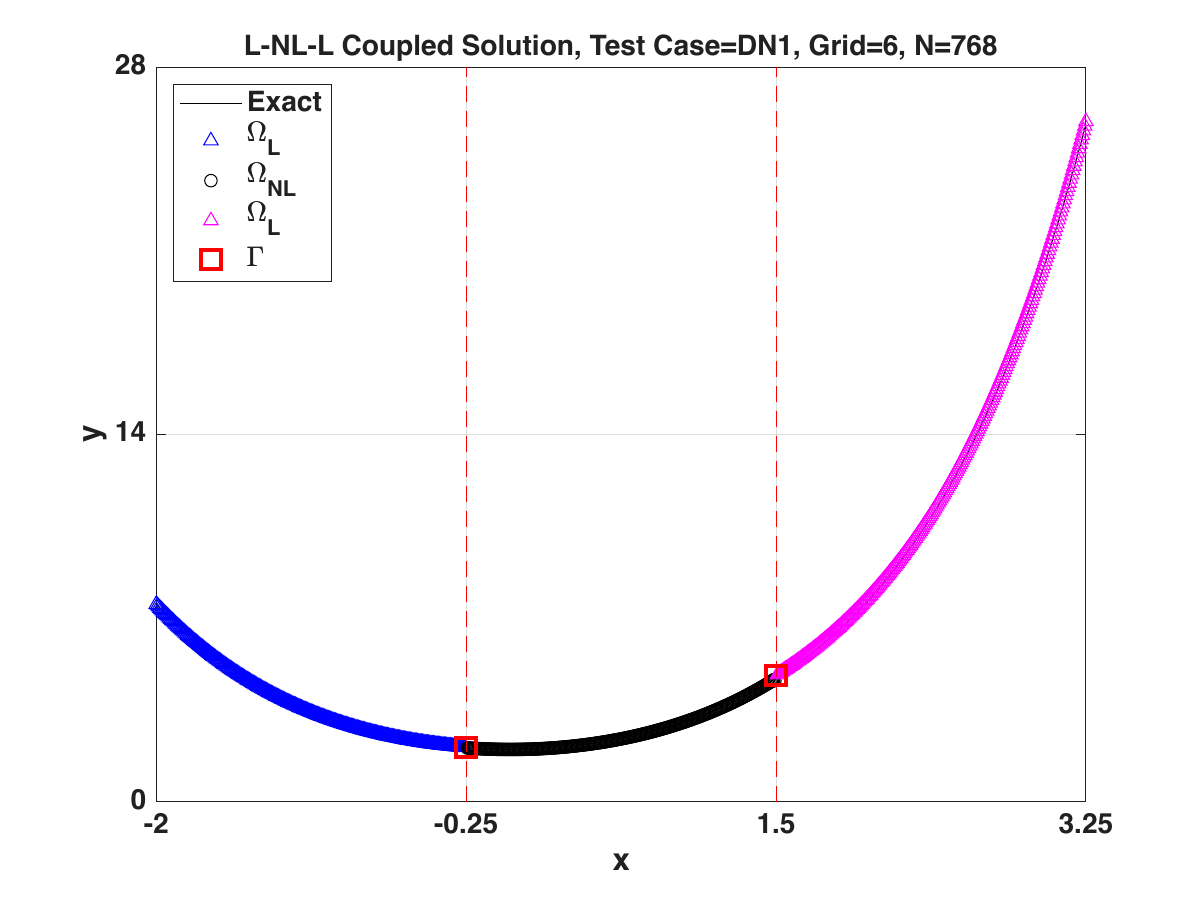}
\end{subfigure}
\begin{subfigure}[b]{0.47\textwidth}
\includegraphics[width=\textwidth]{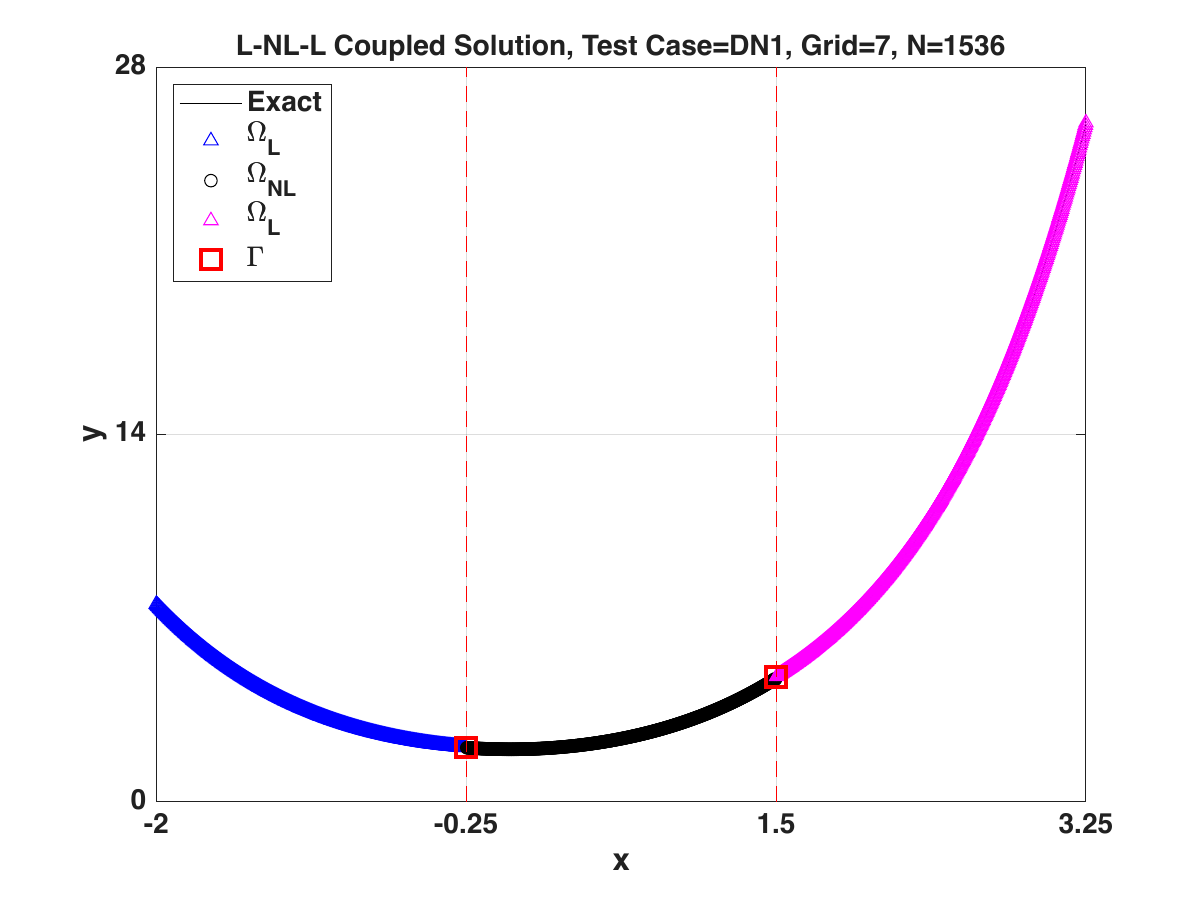}
\end{subfigure}
\caption{L-NL-L coupling with $\BC=\DD$ (top 2 rows) $\BC=\DN$ (bottom 2 rows)}
\label{fig:L-NL-L_coupling_DD1_DN1}
\end{figure}

\begin{figure}[tb]
\centering
\begin{subfigure}[b]{0.47\textwidth}
\includegraphics[width=\textwidth]{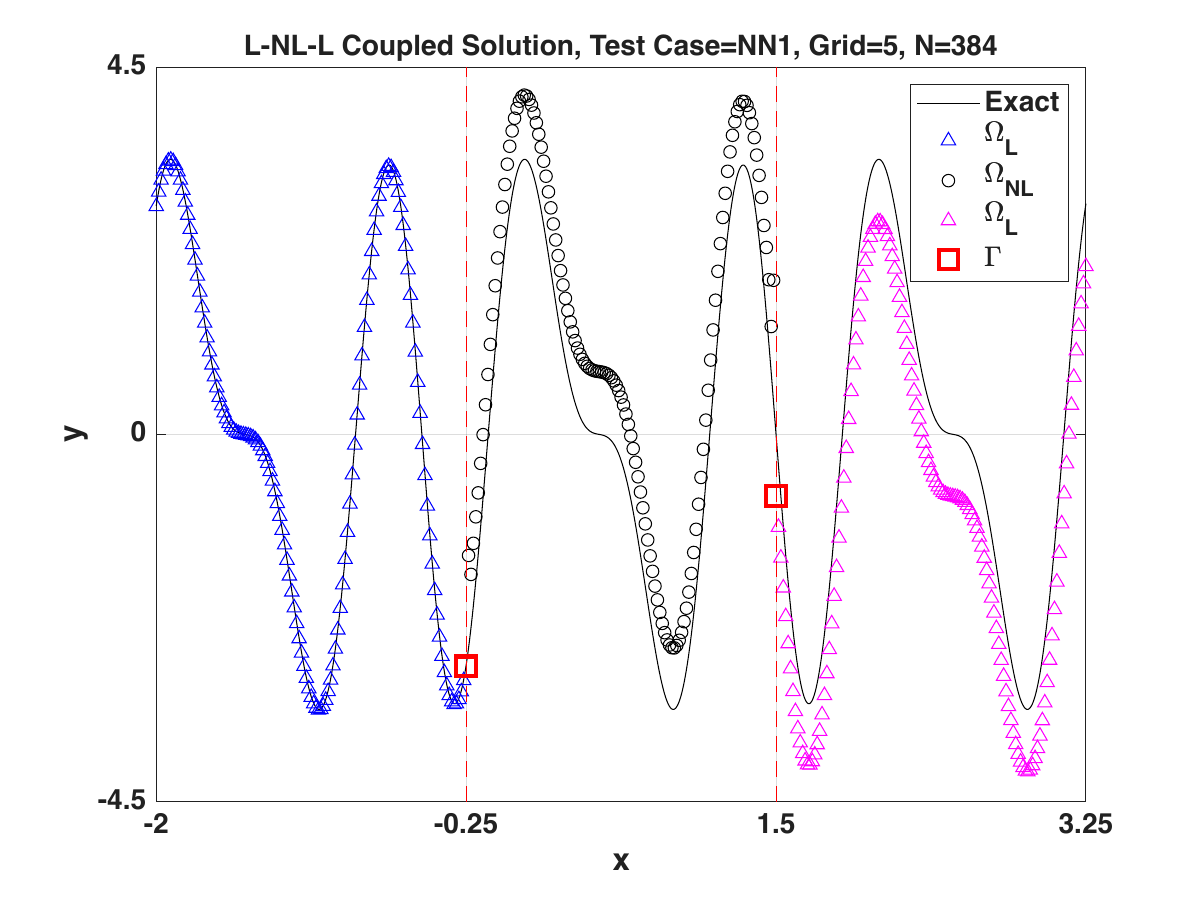}
\end{subfigure}
\begin{subfigure}[b]{0.47\textwidth}
\includegraphics[width=\textwidth]{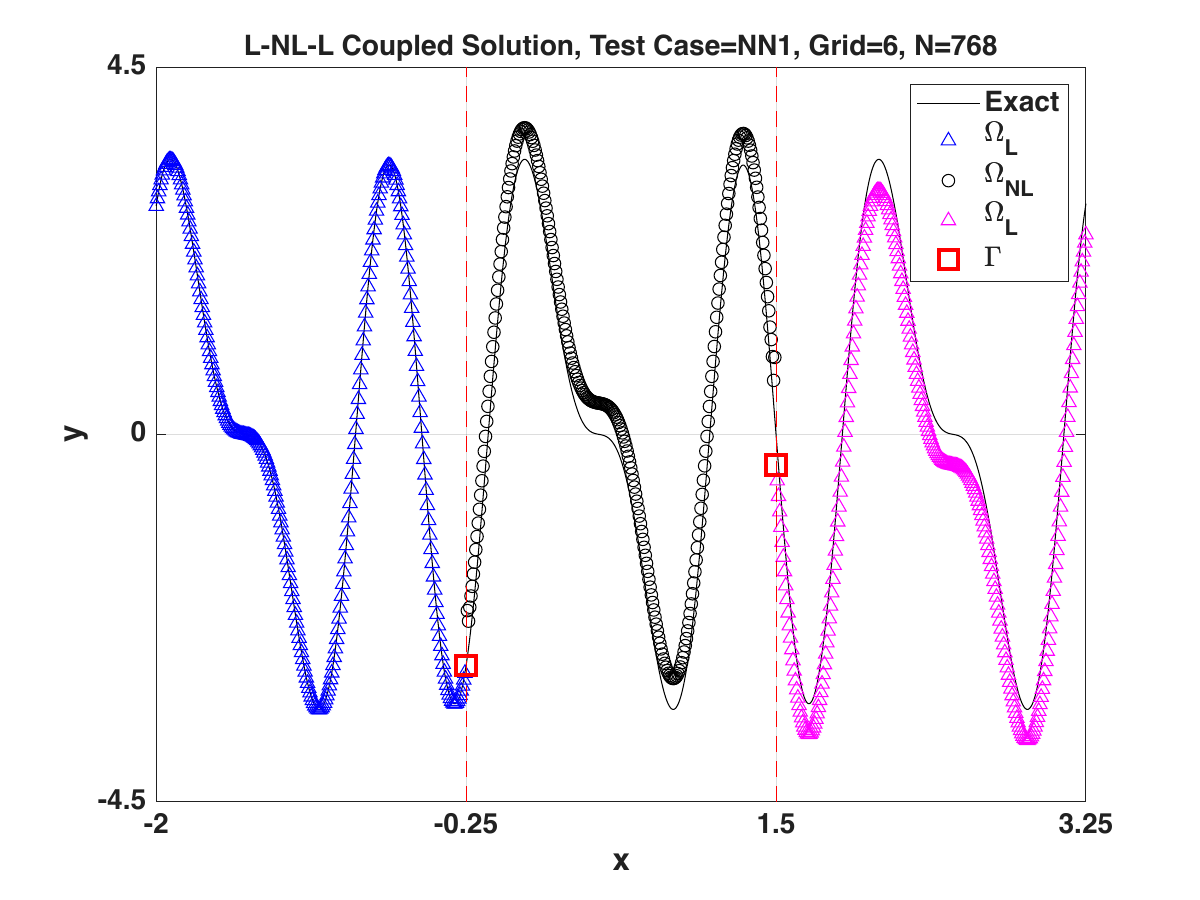}
\end{subfigure}
\begin{subfigure}[b]{0.47\textwidth}
\includegraphics[width=\textwidth]{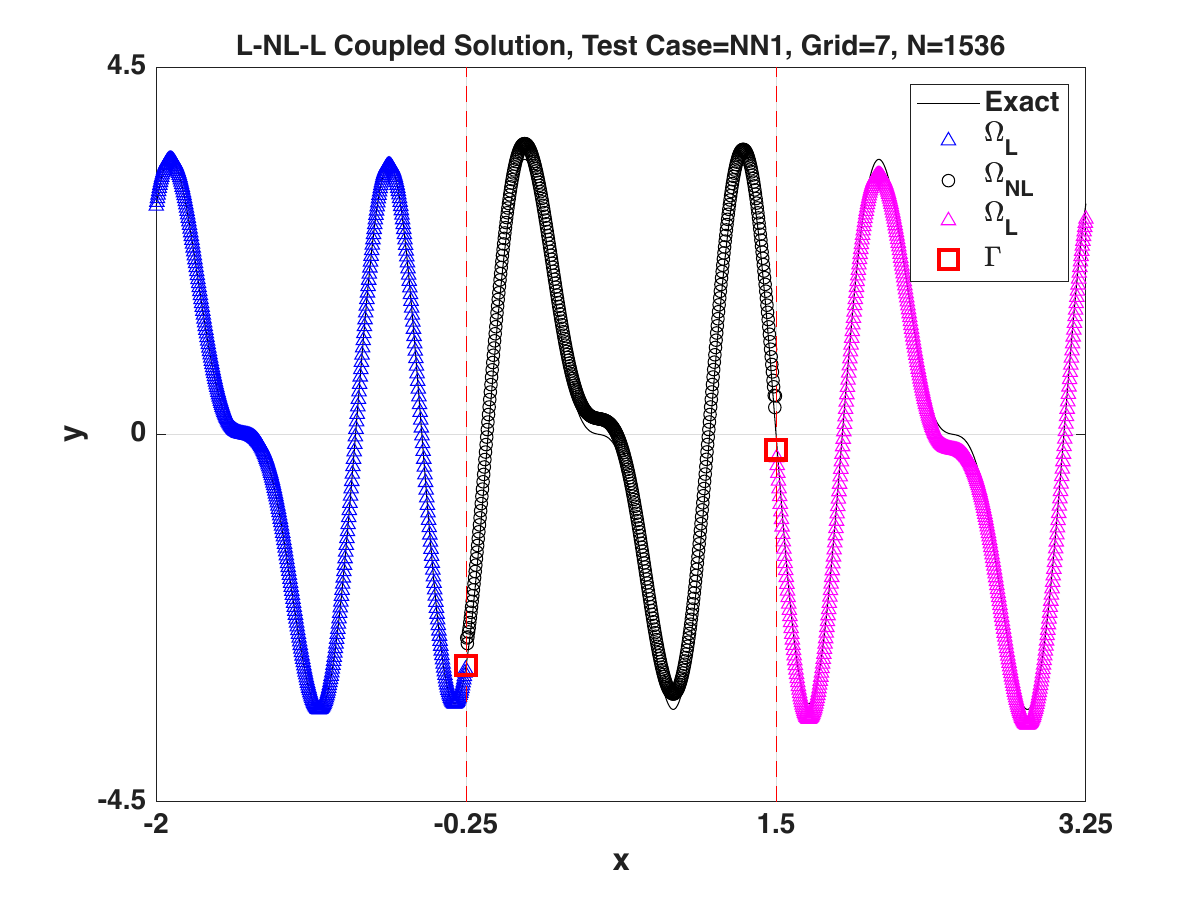}
\end{subfigure}
\begin{subfigure}[b]{0.47\textwidth}
\includegraphics[width=\textwidth]{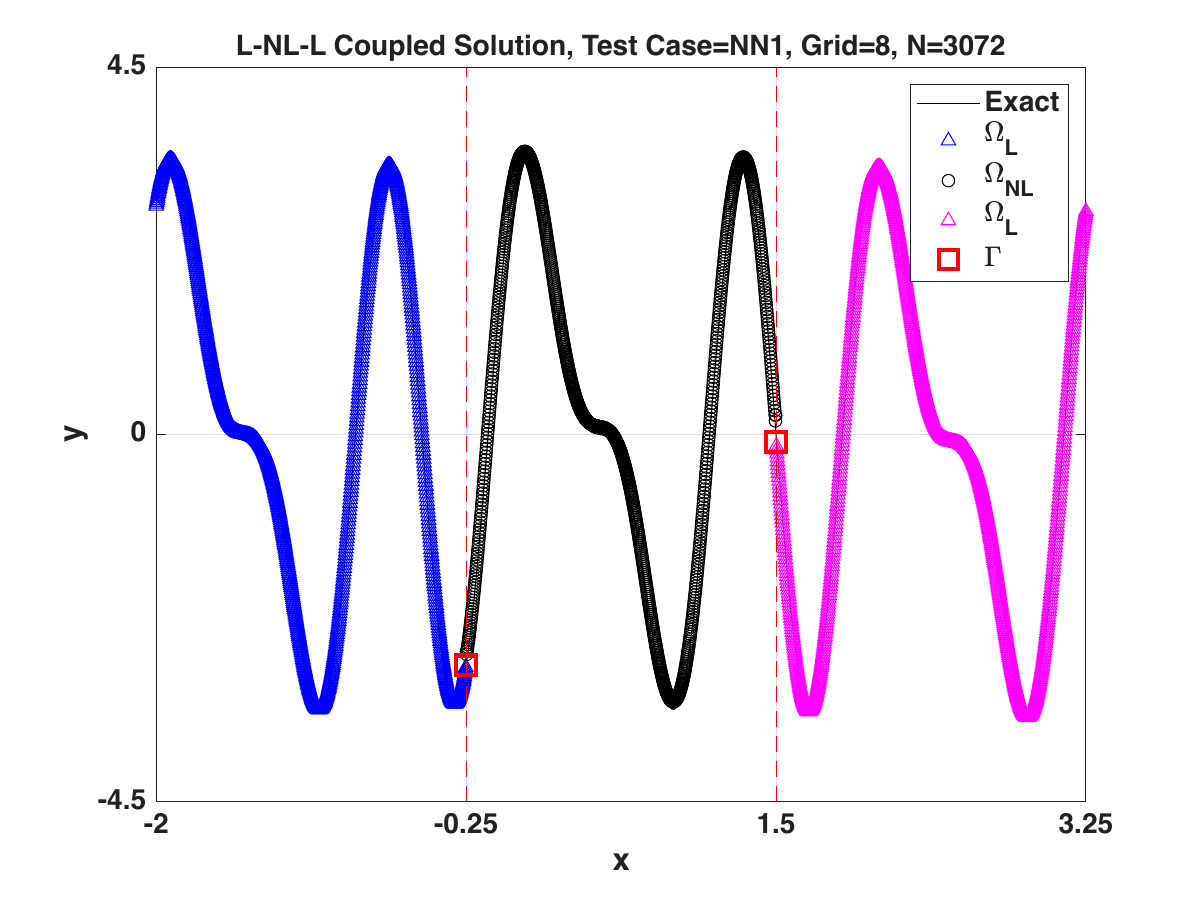}
\end{subfigure}
\begin{subfigure}[b]{0.47\textwidth}
\includegraphics[width=\textwidth]{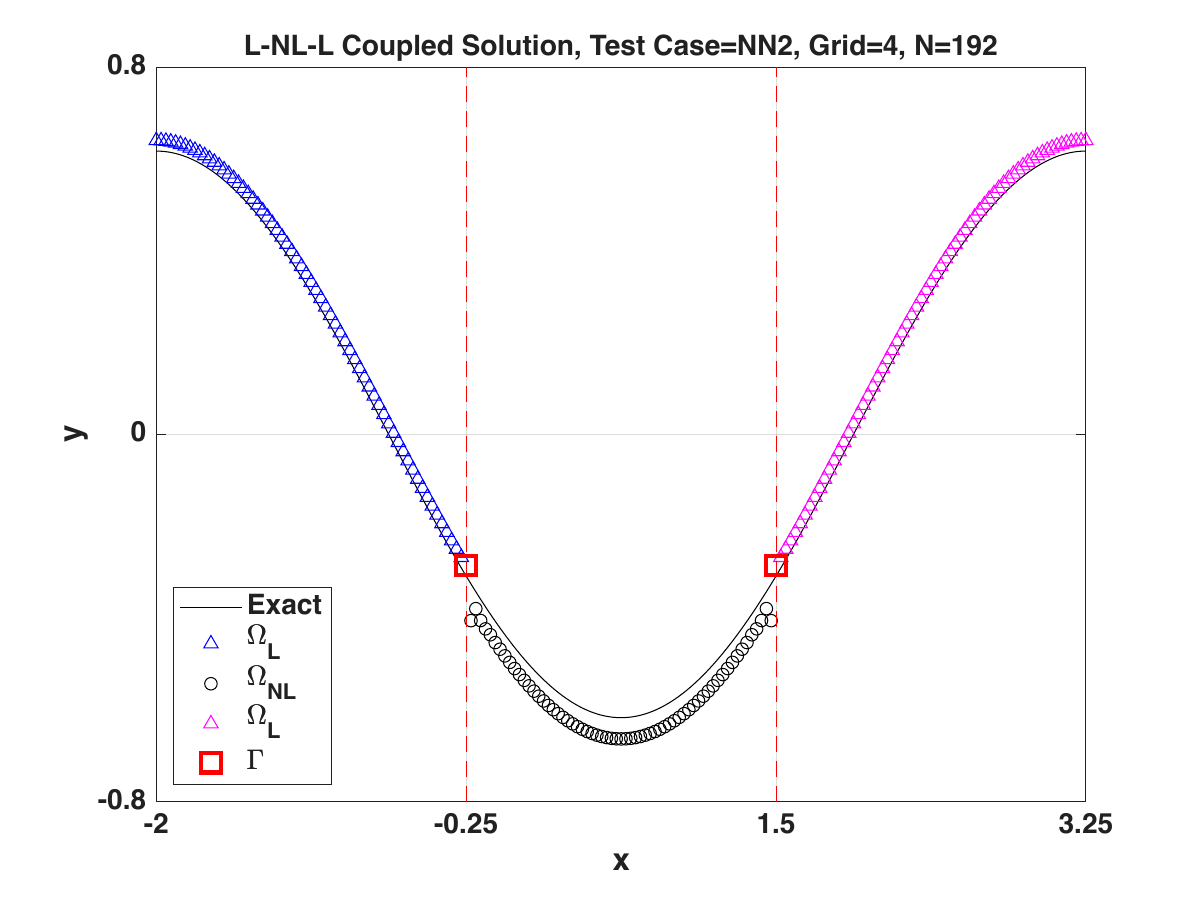}
\end{subfigure}
\begin{subfigure}[b]{0.47\textwidth}
\includegraphics[width=\textwidth]{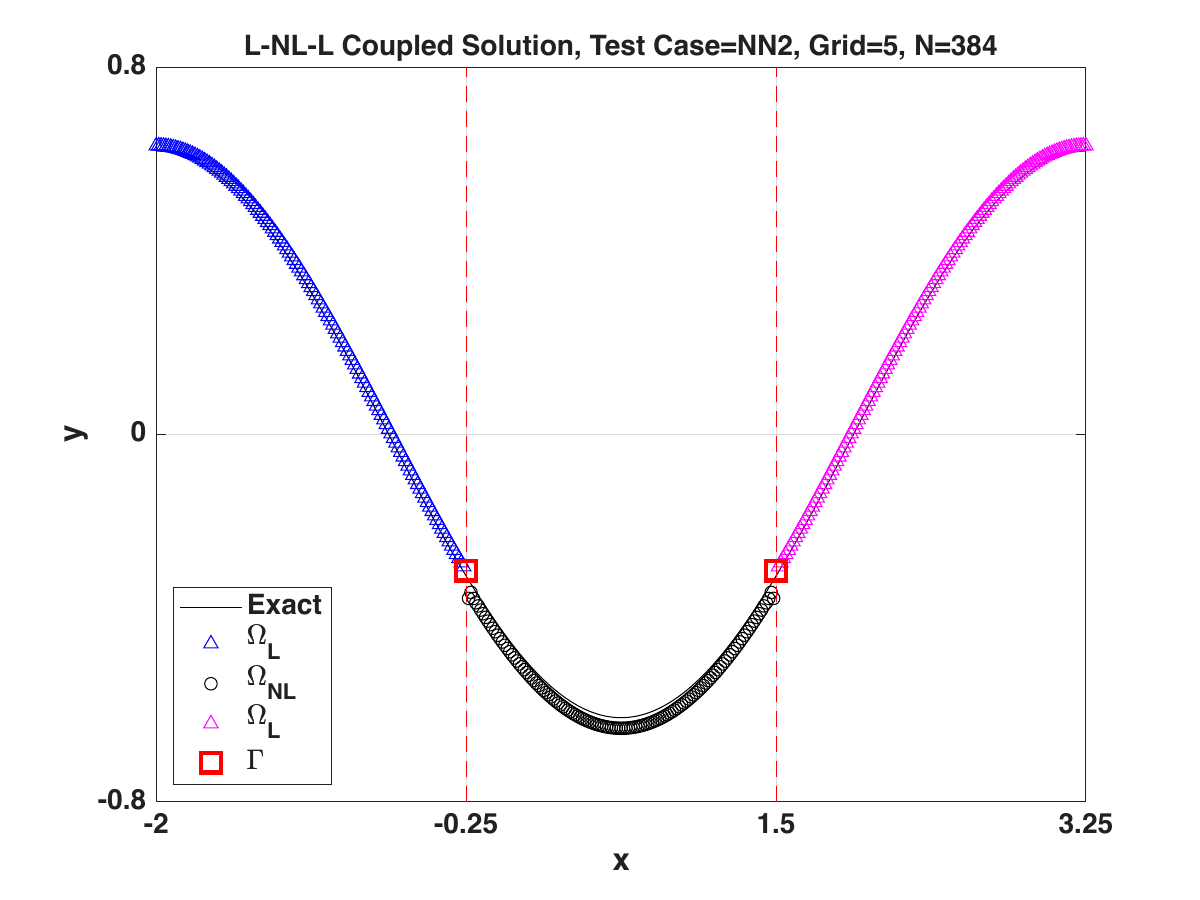}
\end{subfigure}
\begin{subfigure}[b]{0.47\textwidth}
\includegraphics[width=\textwidth]{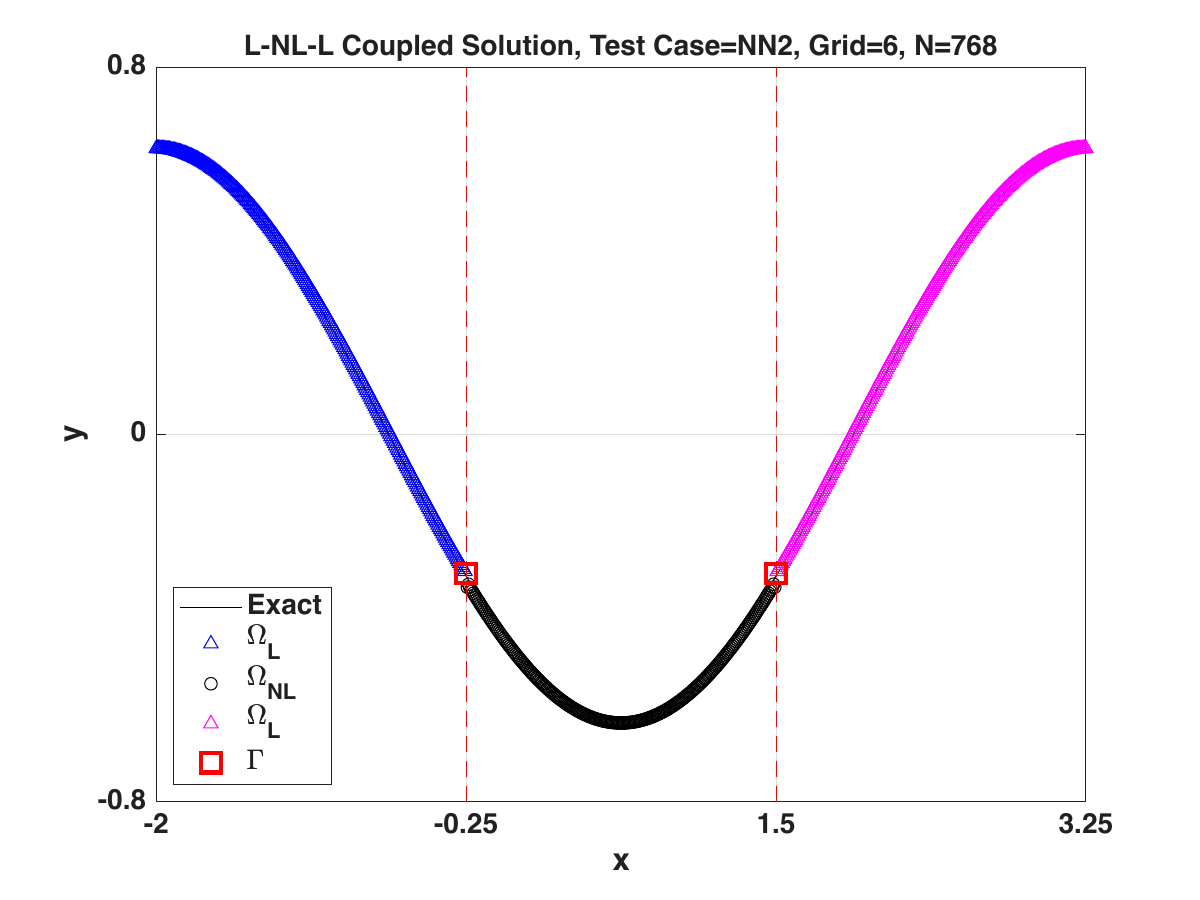}
\end{subfigure}
\begin{subfigure}[b]{0.47\textwidth}
\includegraphics[width=\textwidth]{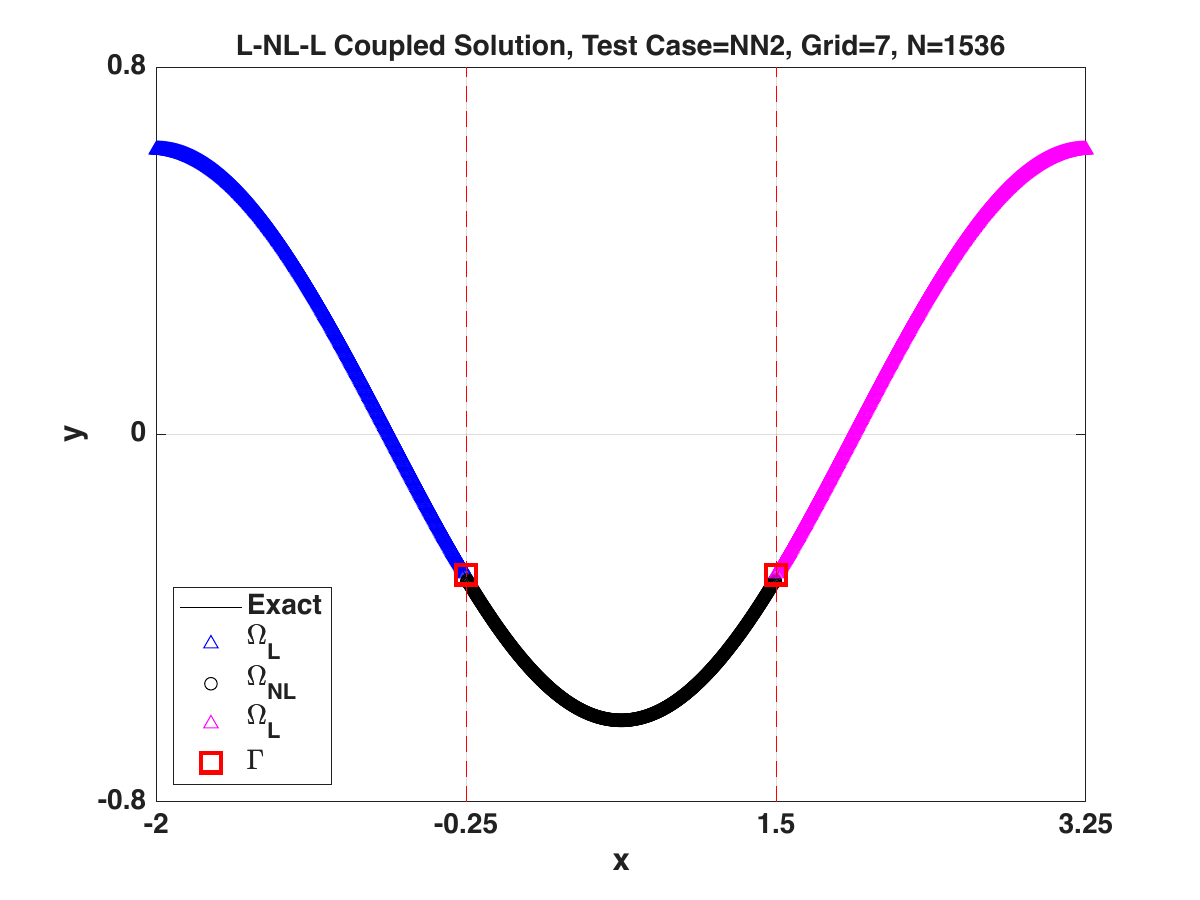}
\end{subfigure}
\caption{L-NL-L coupling with $\BC=\NN$, $\NN1$ (top 2 rows) and $\NN2$ (bottom 2 rows)}
\label{fig:L_NL_L_coupling_NN1_NN2}
\end{figure}


\begin{figure}[tb]
\centering
\begin{subfigure}[b]{0.47\textwidth}
\includegraphics[width=\textwidth]{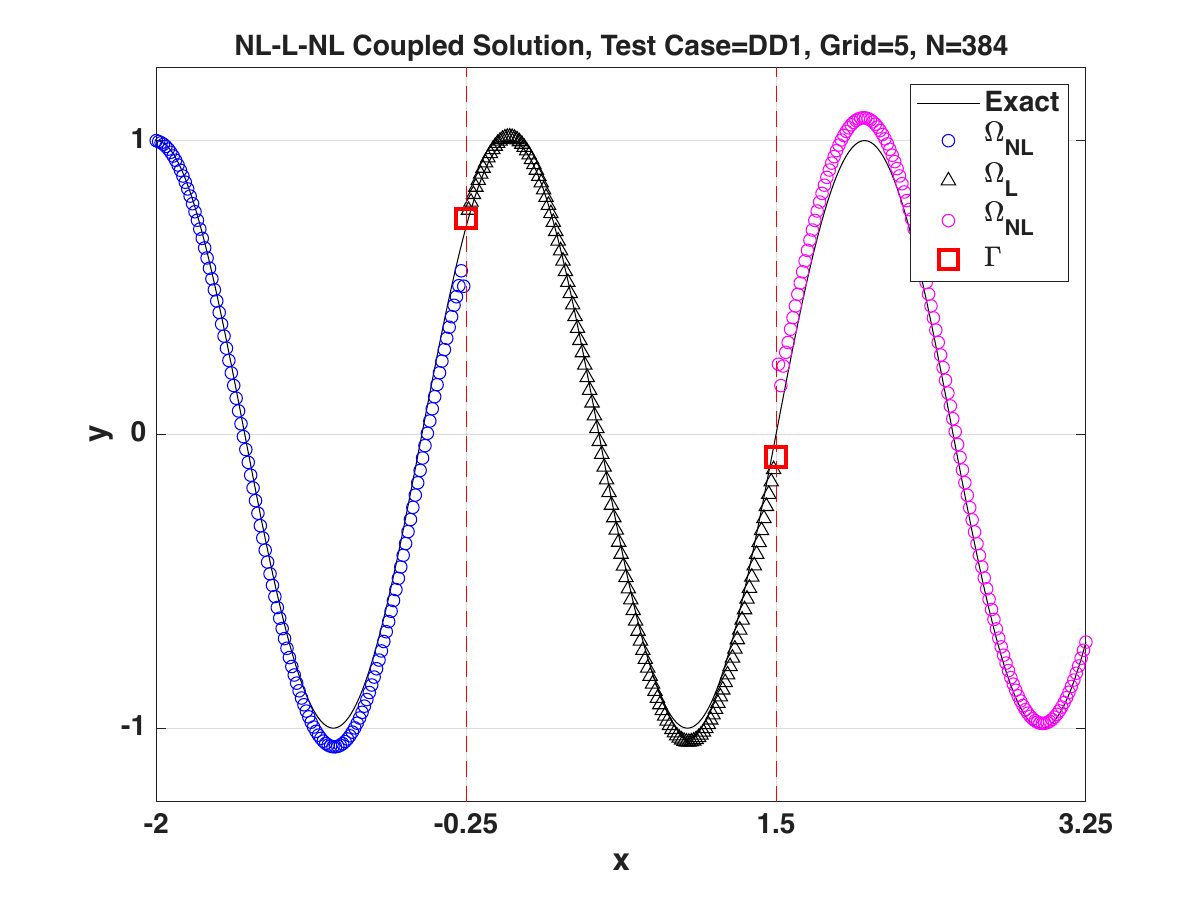}
\end{subfigure}
\begin{subfigure}[b]{0.47\textwidth}
\includegraphics[width=\textwidth]{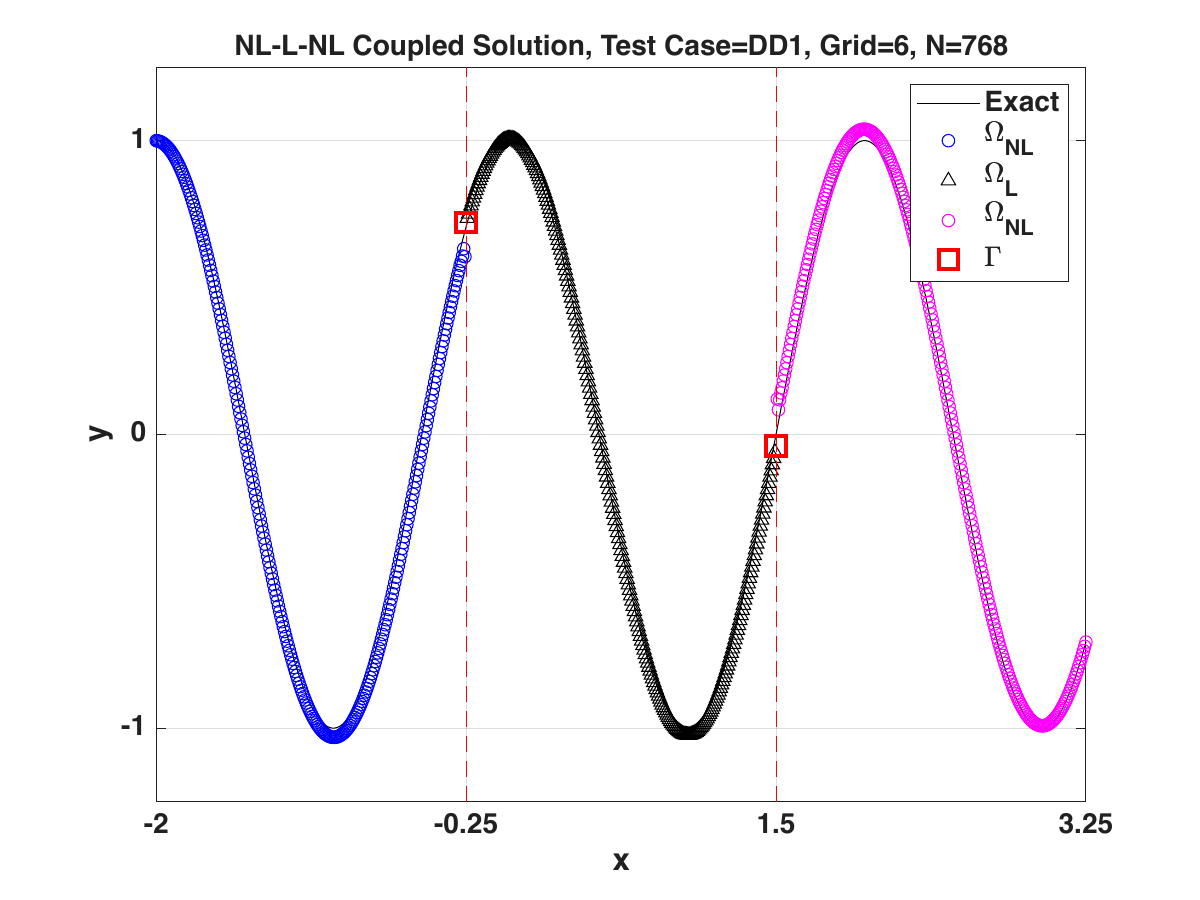}
\end{subfigure}
\begin{subfigure}[b]{0.47\textwidth}
\includegraphics[width=\textwidth]{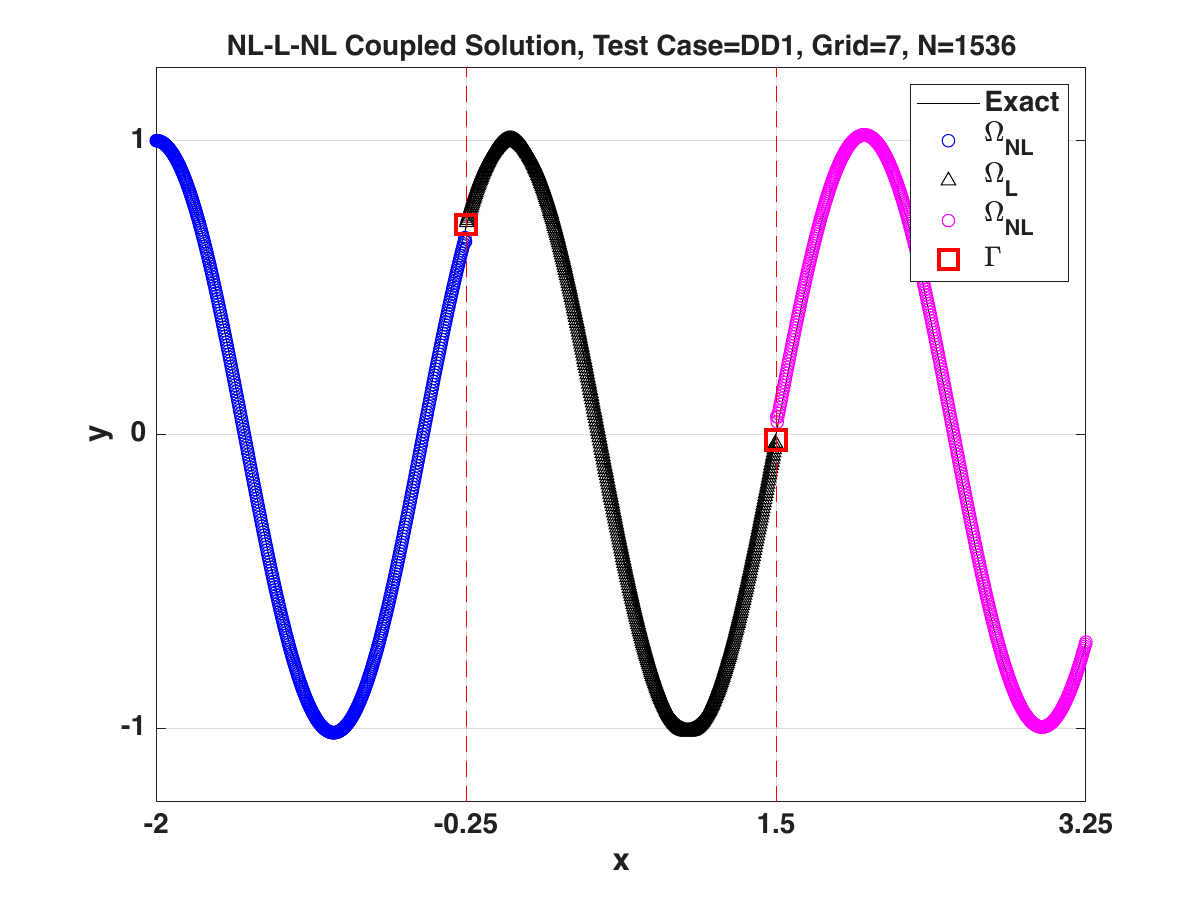}
\end{subfigure}
\begin{subfigure}[b]{0.47\textwidth}
\includegraphics[width=\textwidth]{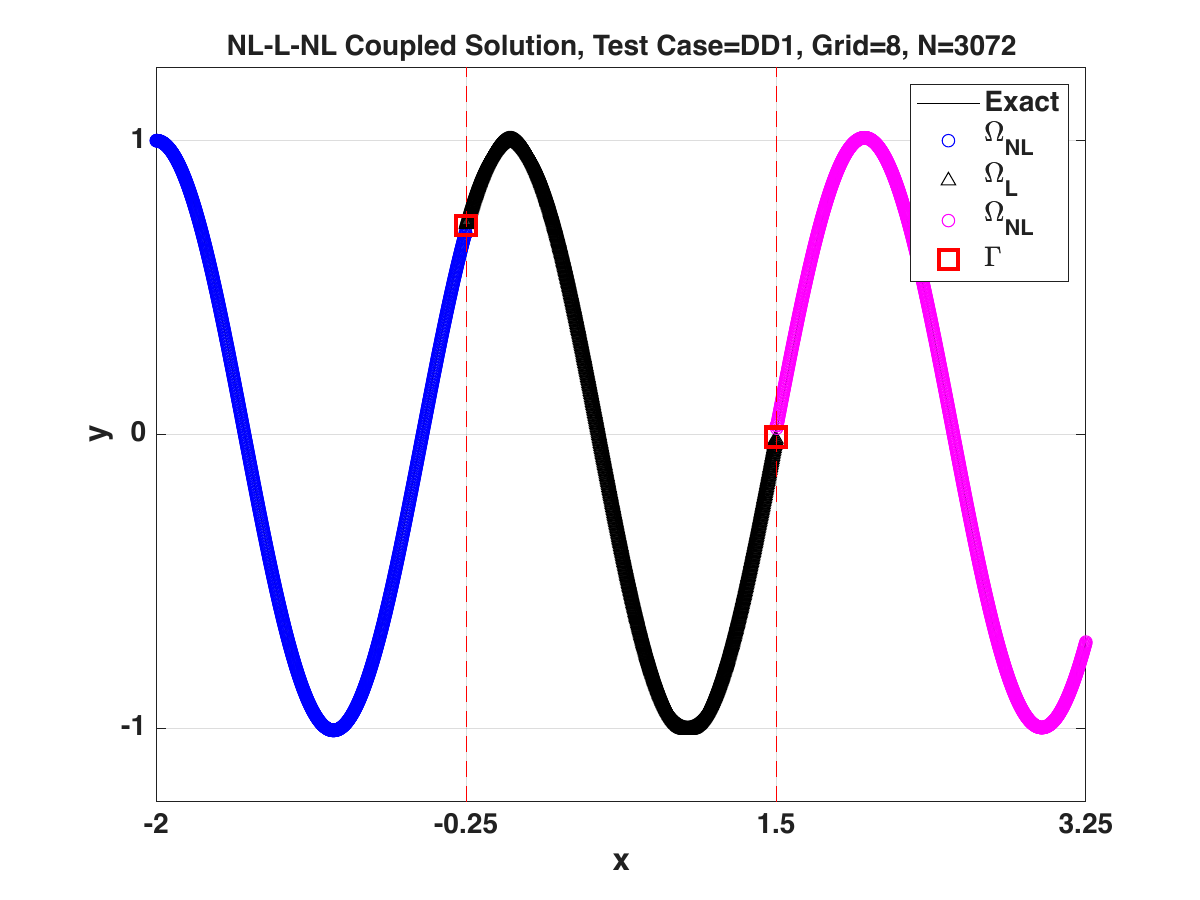}
\end{subfigure}
\begin{subfigure}[b]{0.47\textwidth}
\includegraphics[width=\textwidth]{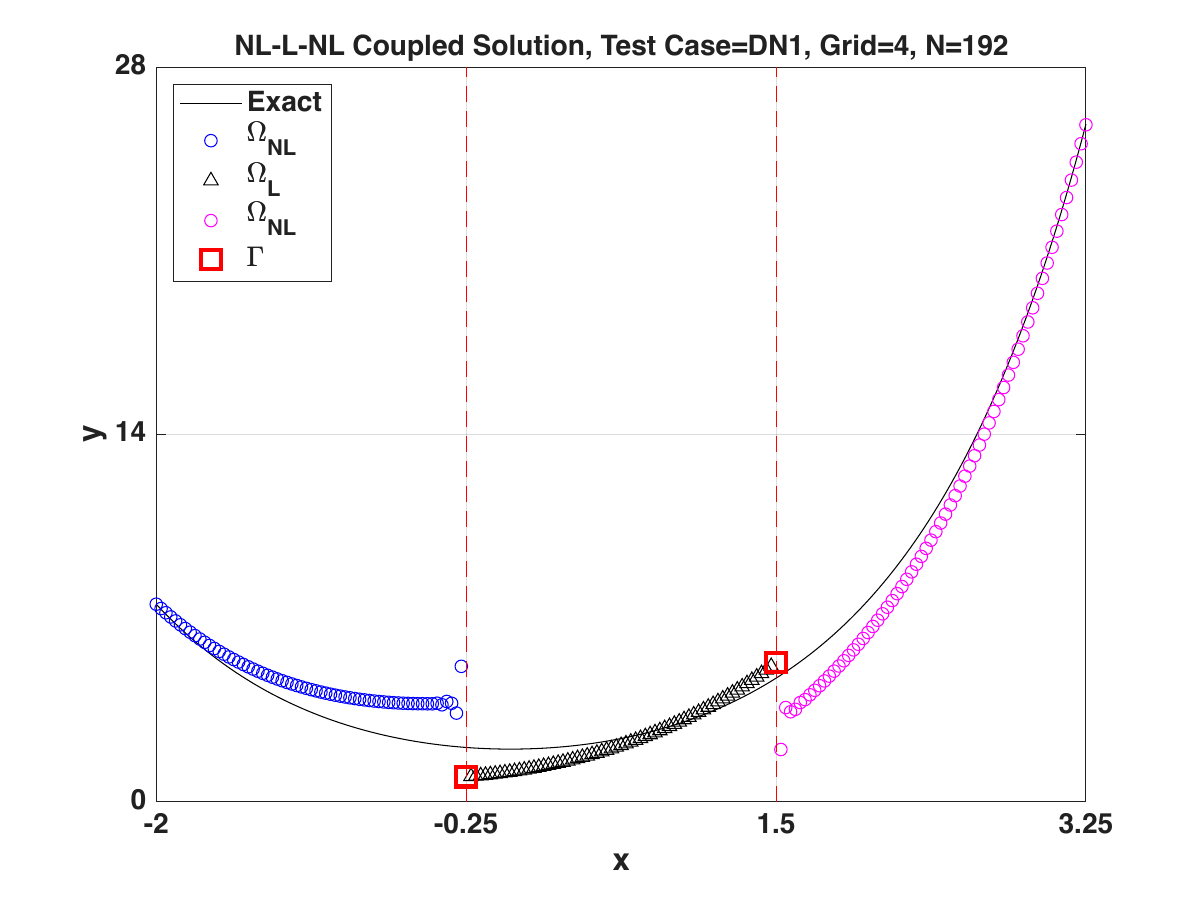}
\end{subfigure}
\begin{subfigure}[b]{0.47\textwidth}
\includegraphics[width=\textwidth]{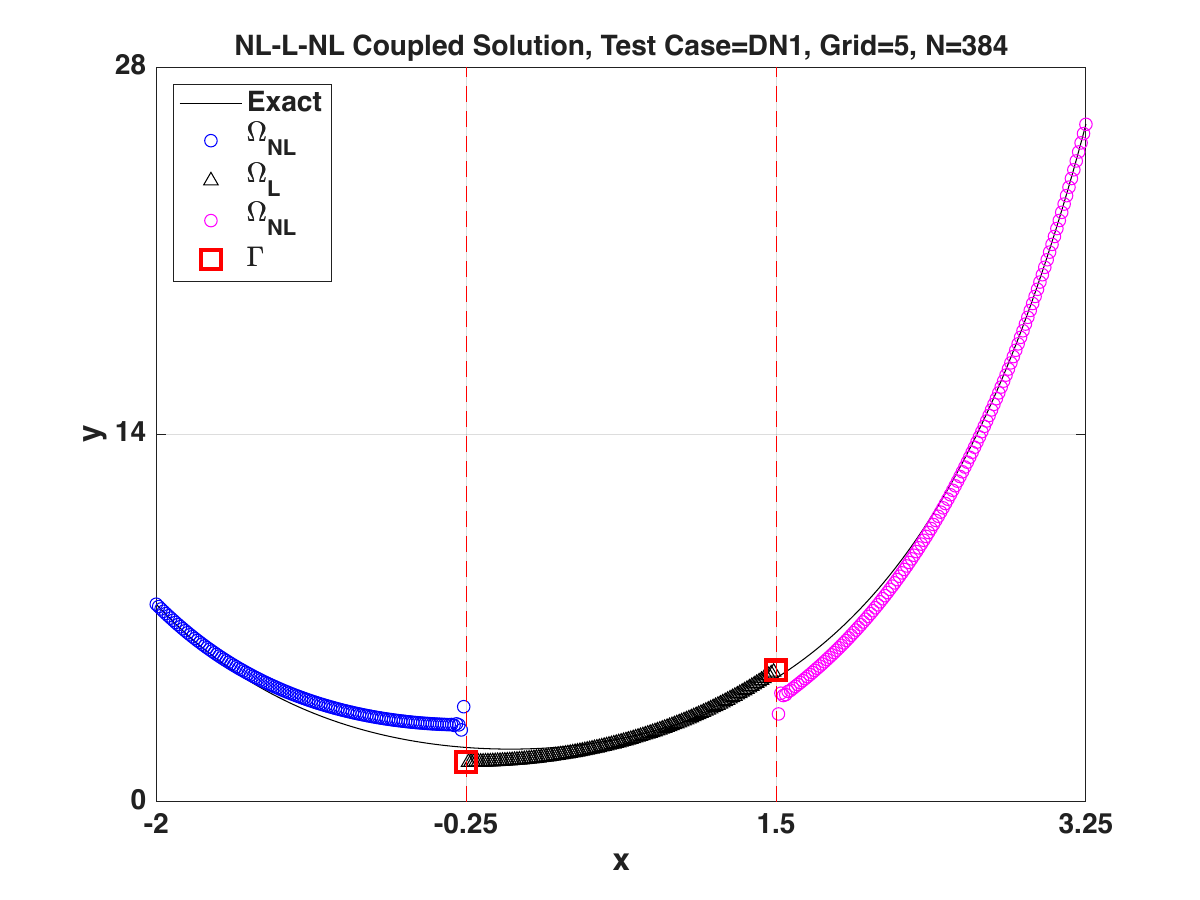}
\end{subfigure}
\begin{subfigure}[b]{0.47\textwidth}
\includegraphics[width=\textwidth]{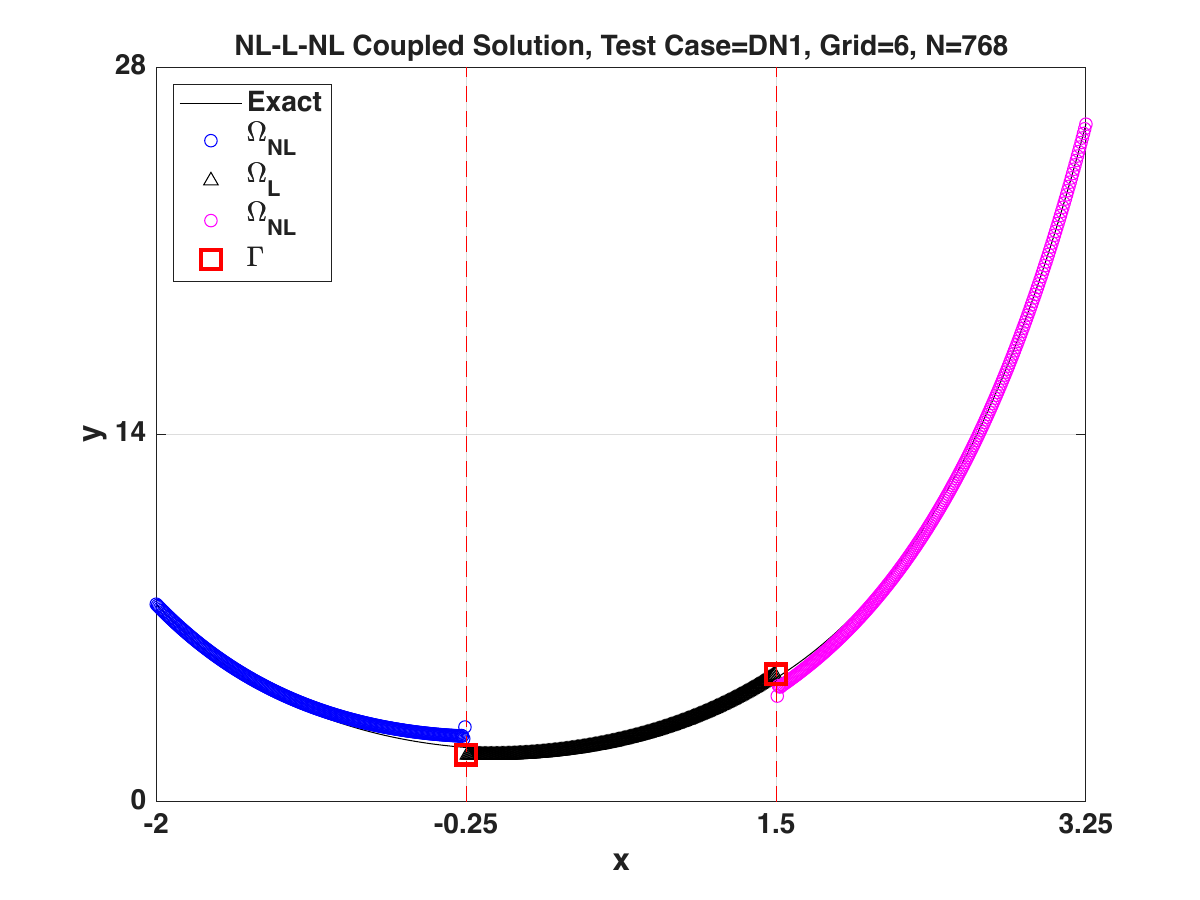}
\end{subfigure}
\begin{subfigure}[b]{0.47\textwidth}
\includegraphics[width=\textwidth]{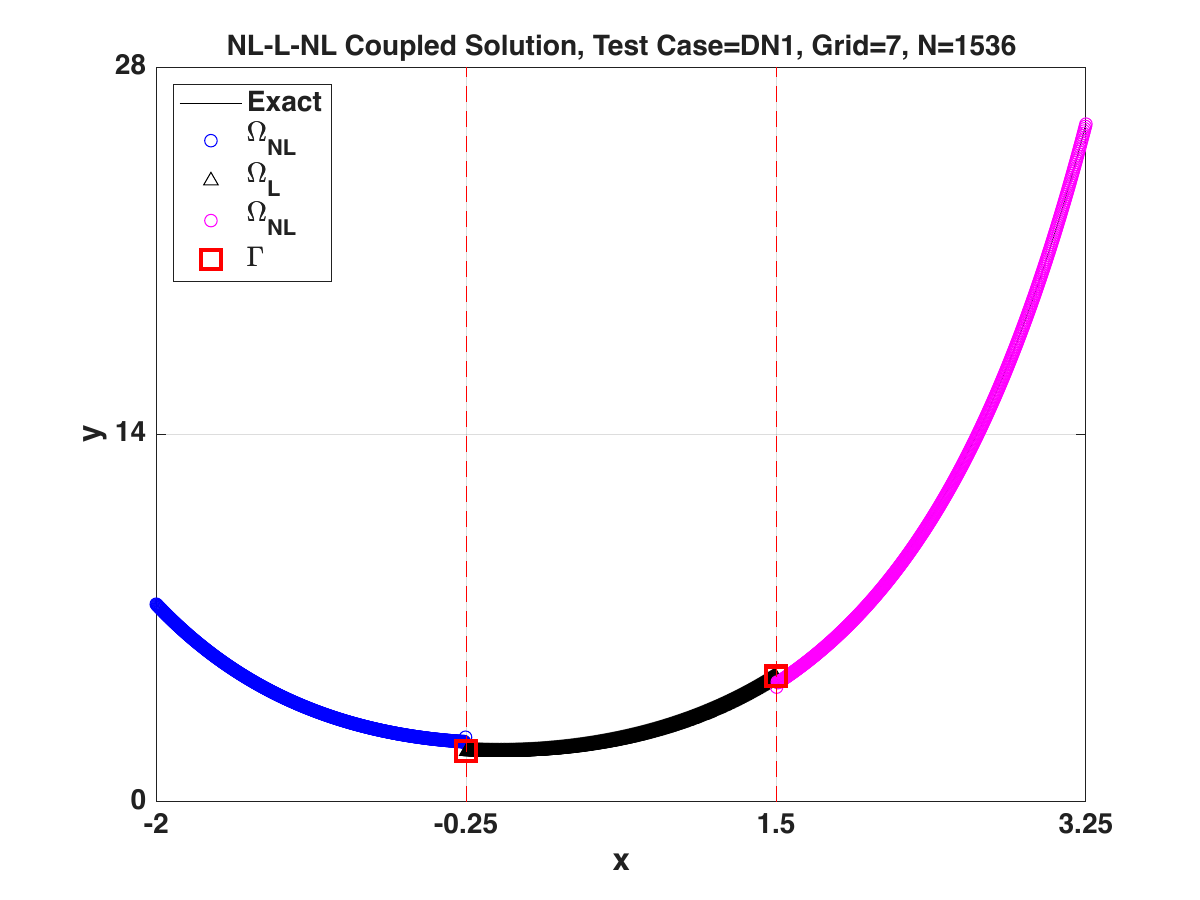}
\end{subfigure}
\caption{NL-L-NL coupling with $\BC=\DD$ (top 2 rows) $\BC=\DN$ (bottom 2 rows)}
\label{fig:NL_L_NL_coupling_DD1_DN1}
\end{figure}

\begin{figure}[tb]
\centering
\begin{subfigure}[b]{0.47\textwidth}
\includegraphics[width=\textwidth]{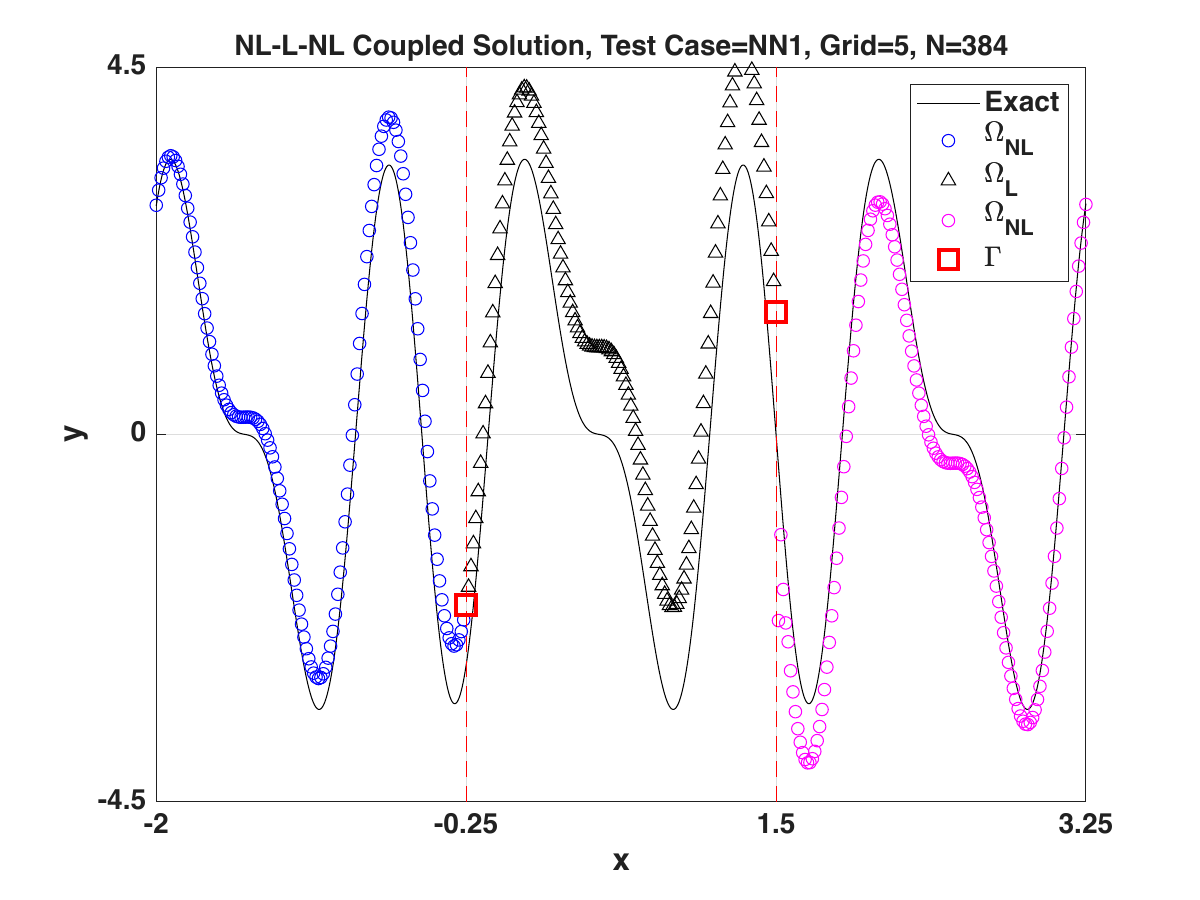}
\end{subfigure}
\begin{subfigure}[b]{0.47\textwidth}
\includegraphics[width=\textwidth]{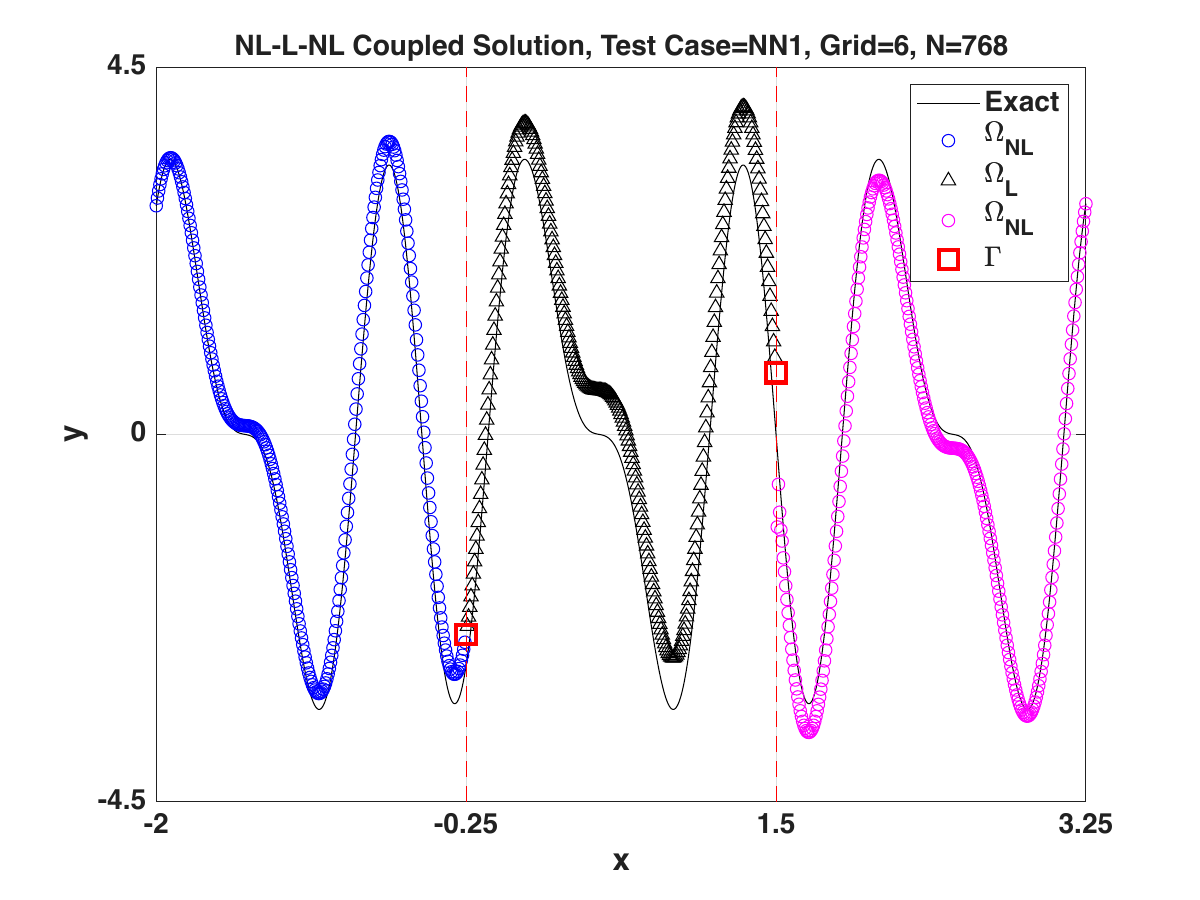}
\end{subfigure}
\begin{subfigure}[b]{0.47\textwidth}
\includegraphics[width=\textwidth]{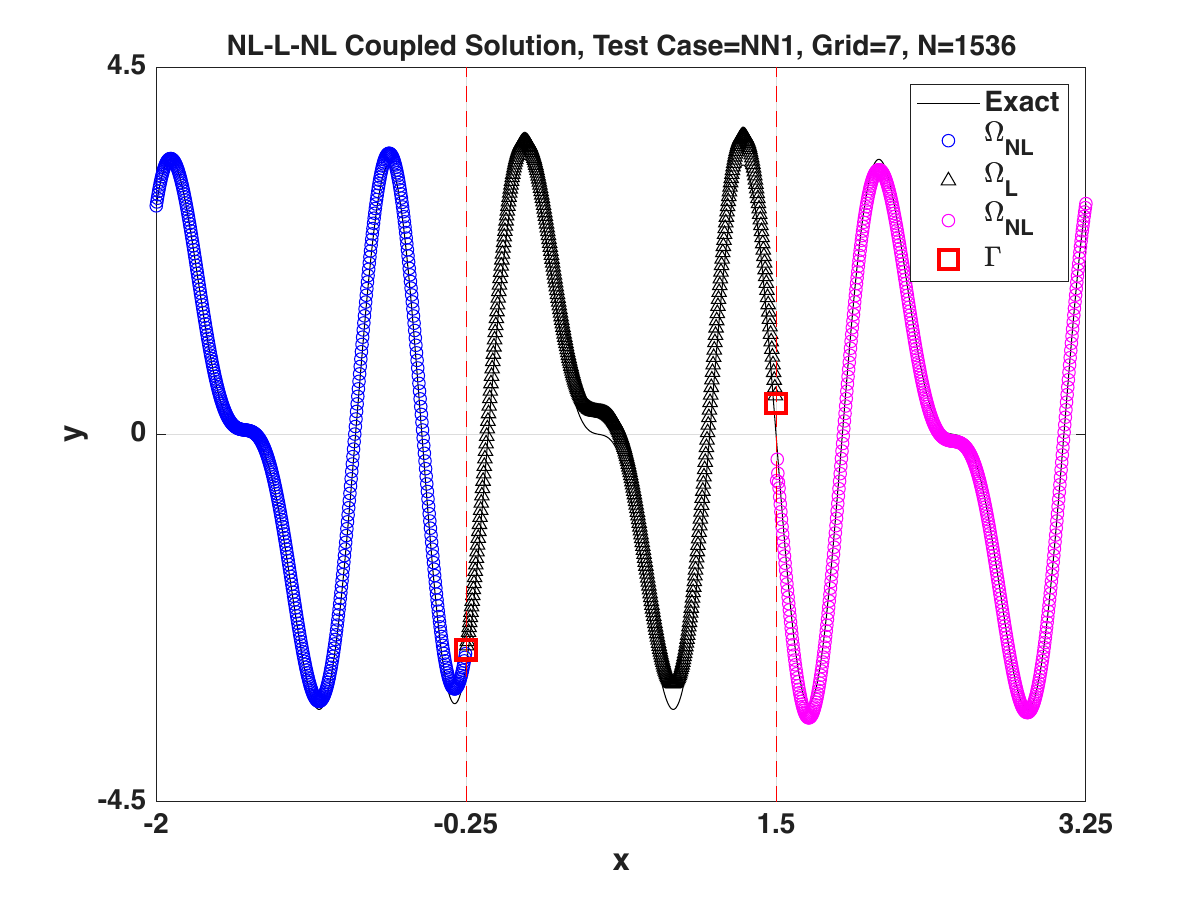}
\end{subfigure}
\begin{subfigure}[b]{0.47\textwidth}
\includegraphics[width=\textwidth]{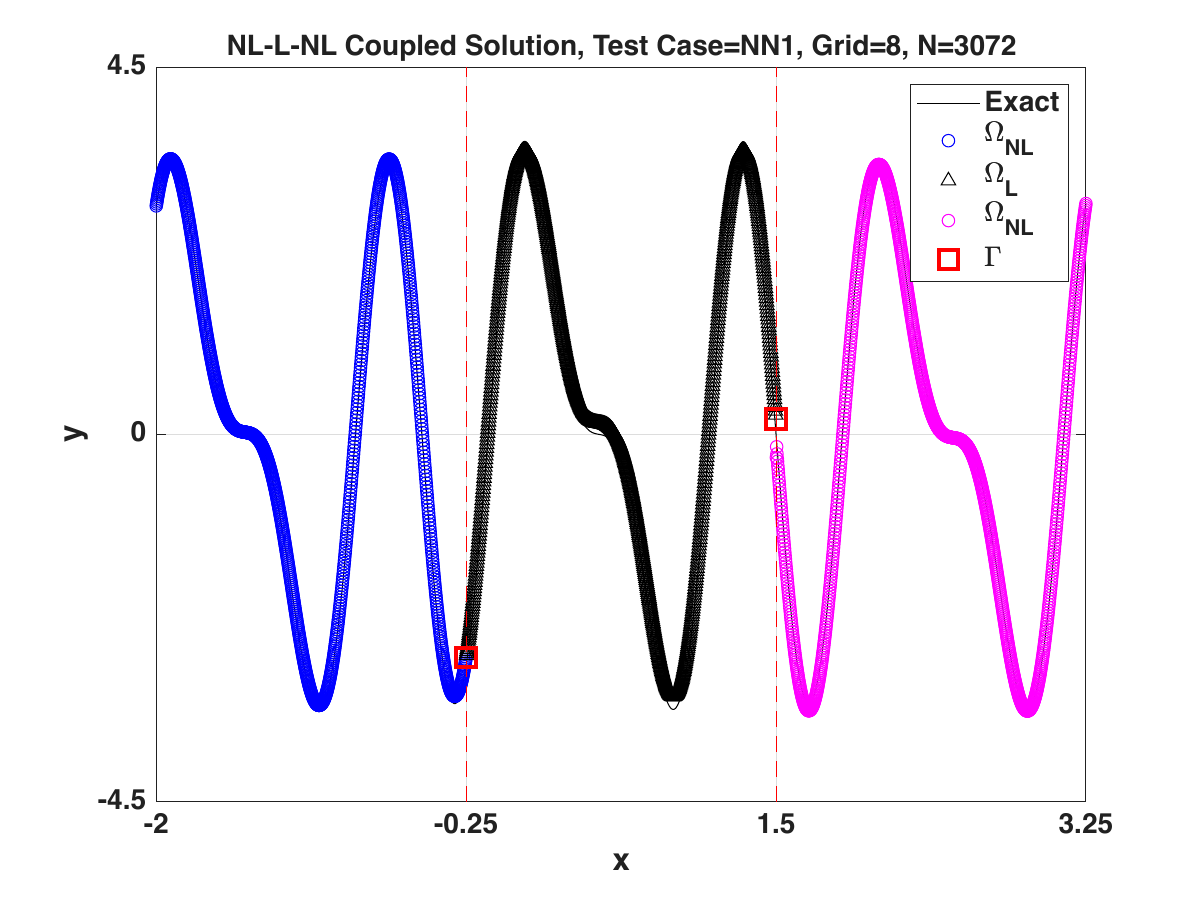}
\end{subfigure}
\begin{subfigure}[b]{0.47\textwidth}
\includegraphics[width=\textwidth]{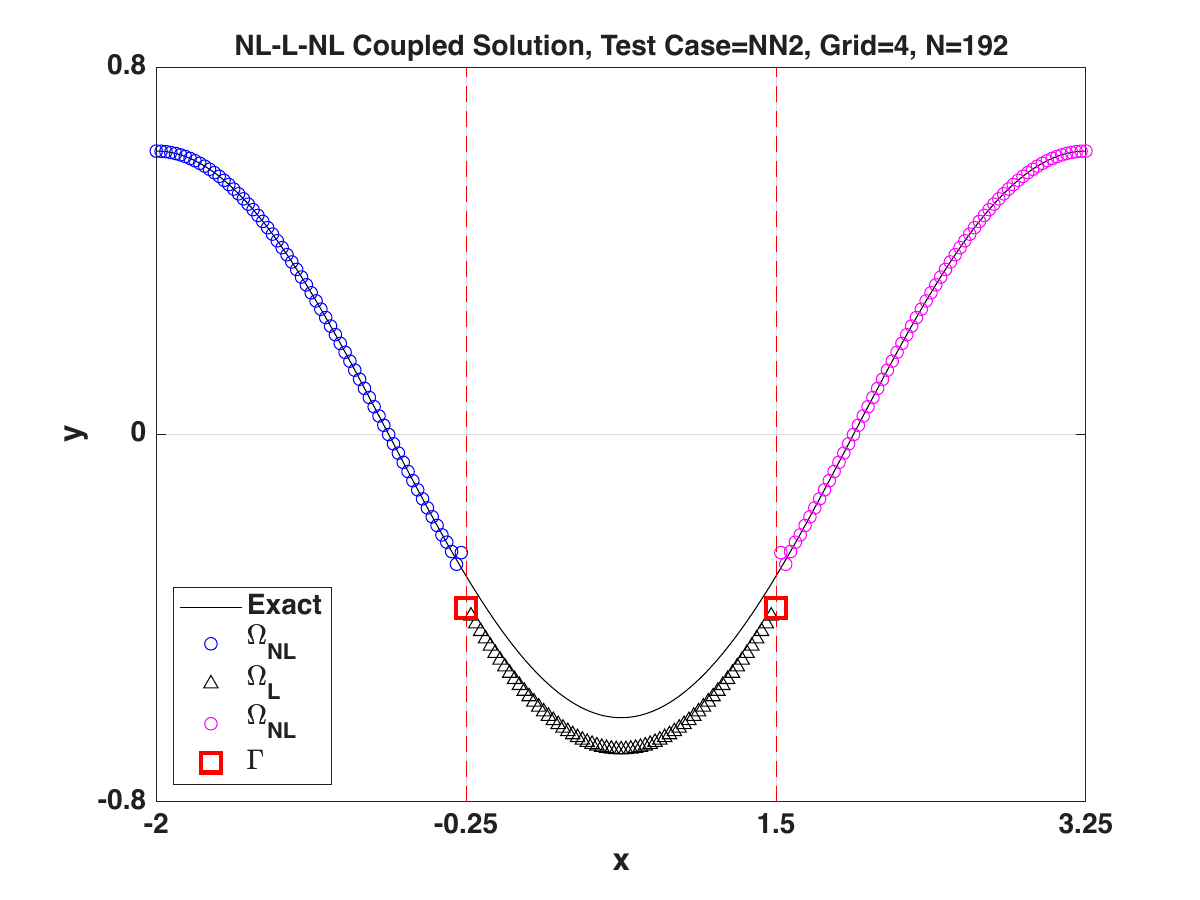}
\end{subfigure}
\begin{subfigure}[b]{0.47\textwidth}
\includegraphics[width=\textwidth]{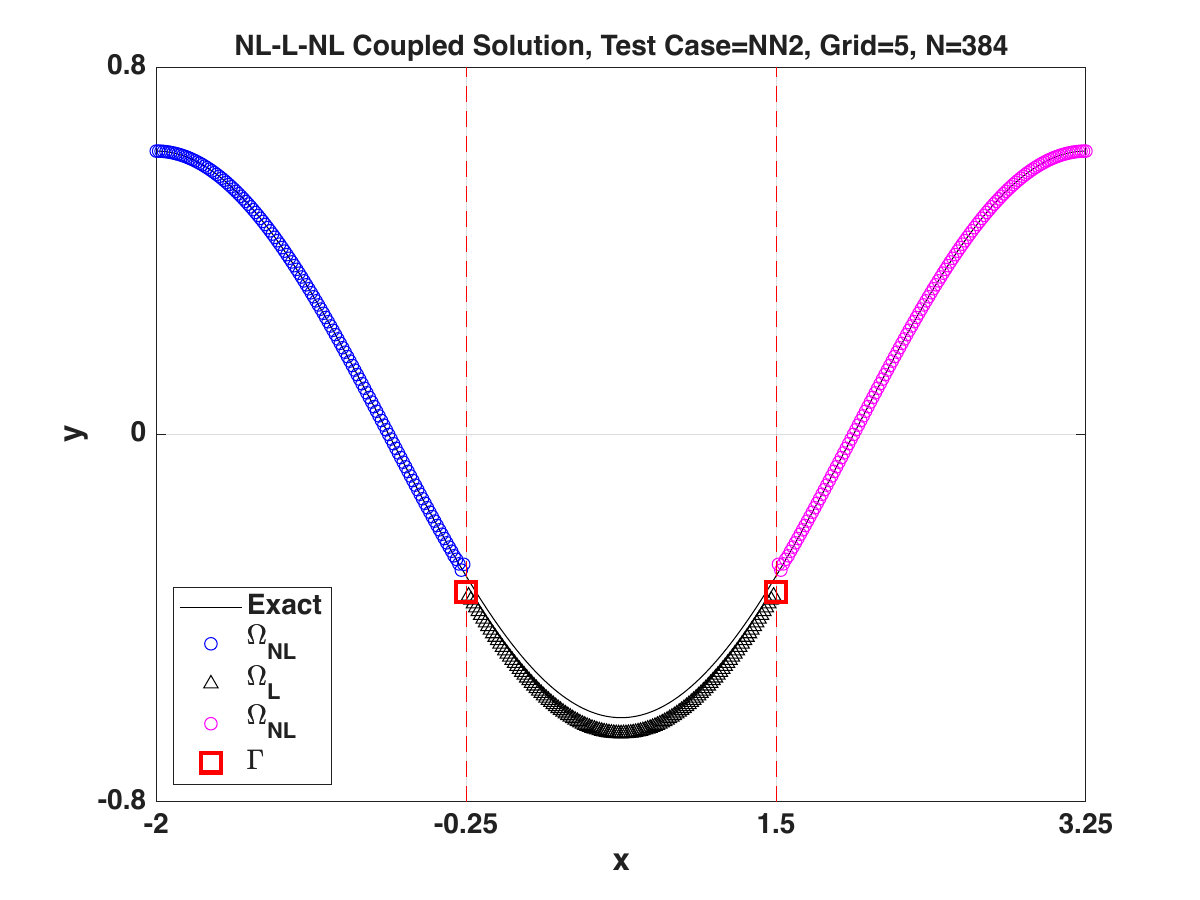}
\end{subfigure}
\begin{subfigure}[b]{0.47\textwidth}
\includegraphics[width=\textwidth]{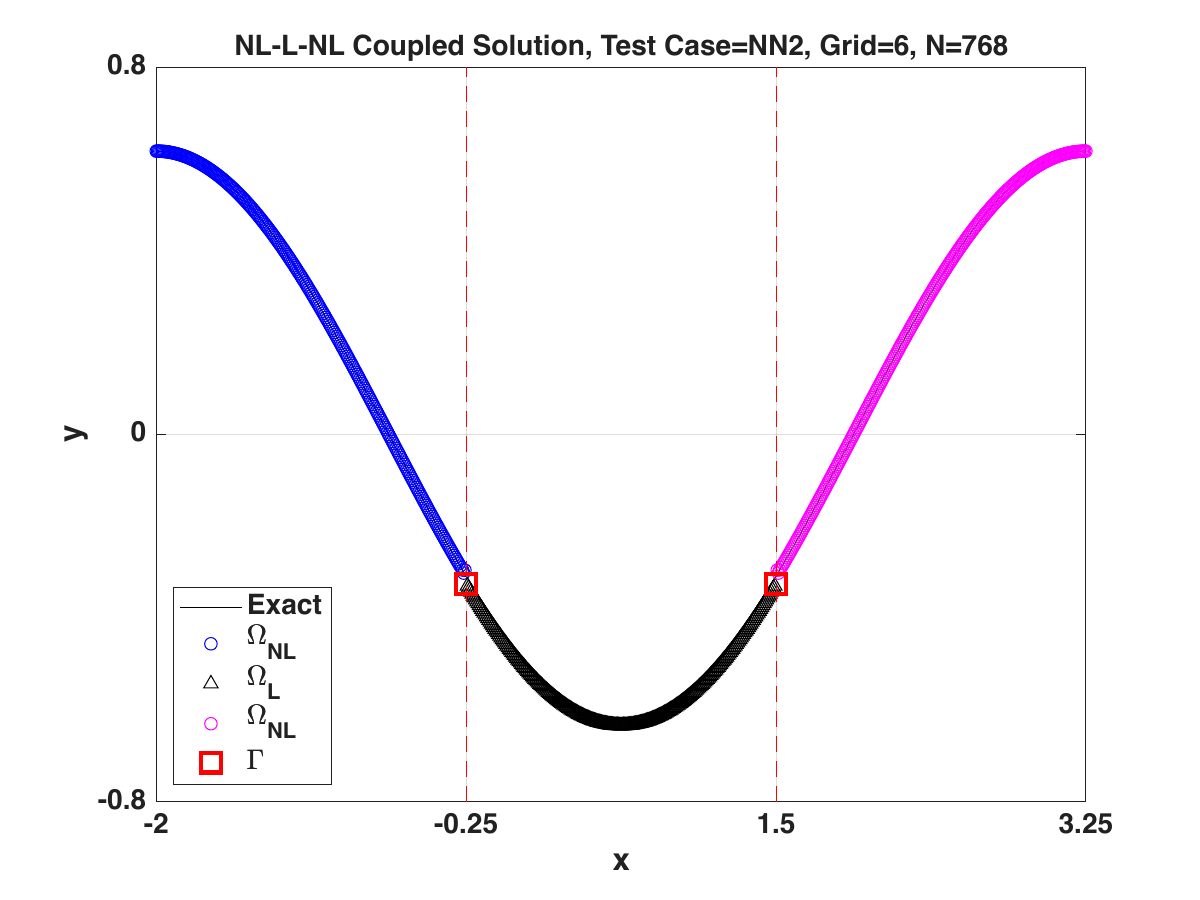}
\end{subfigure}
\begin{subfigure}[b]{0.47\textwidth}
\includegraphics[width=\textwidth]{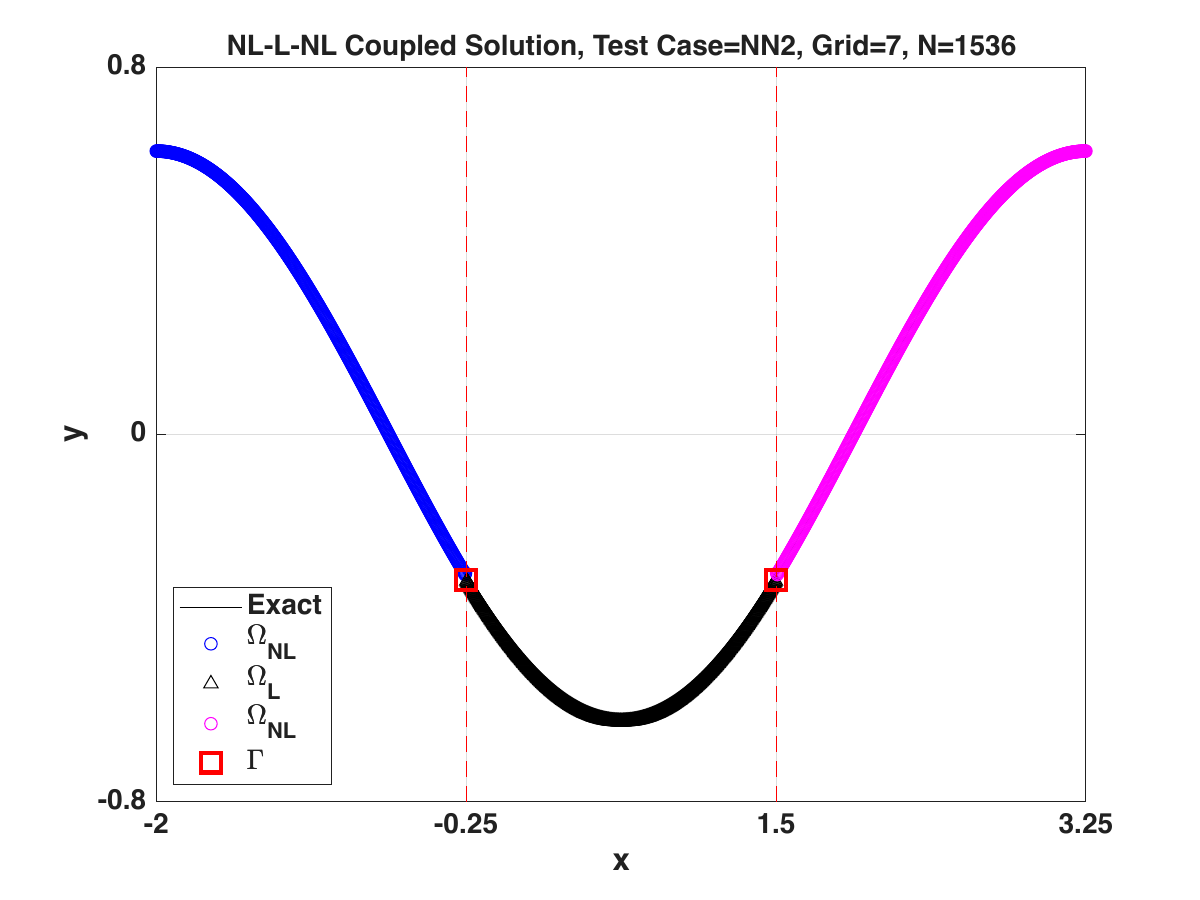}
\end{subfigure}
\caption{NL-L-NL coupling with $\BC=\NN$, $\NN1$ (top 2 rows) and $\NN2$ (bottom 2 rows)}
\label{fig:NL-L-NL_coupling_NN1_NN2}
\end{figure}


\begin{figure}[tb]
\centering
\begin{subfigure}[b]{0.47\textwidth}
\includegraphics[width=\textwidth]{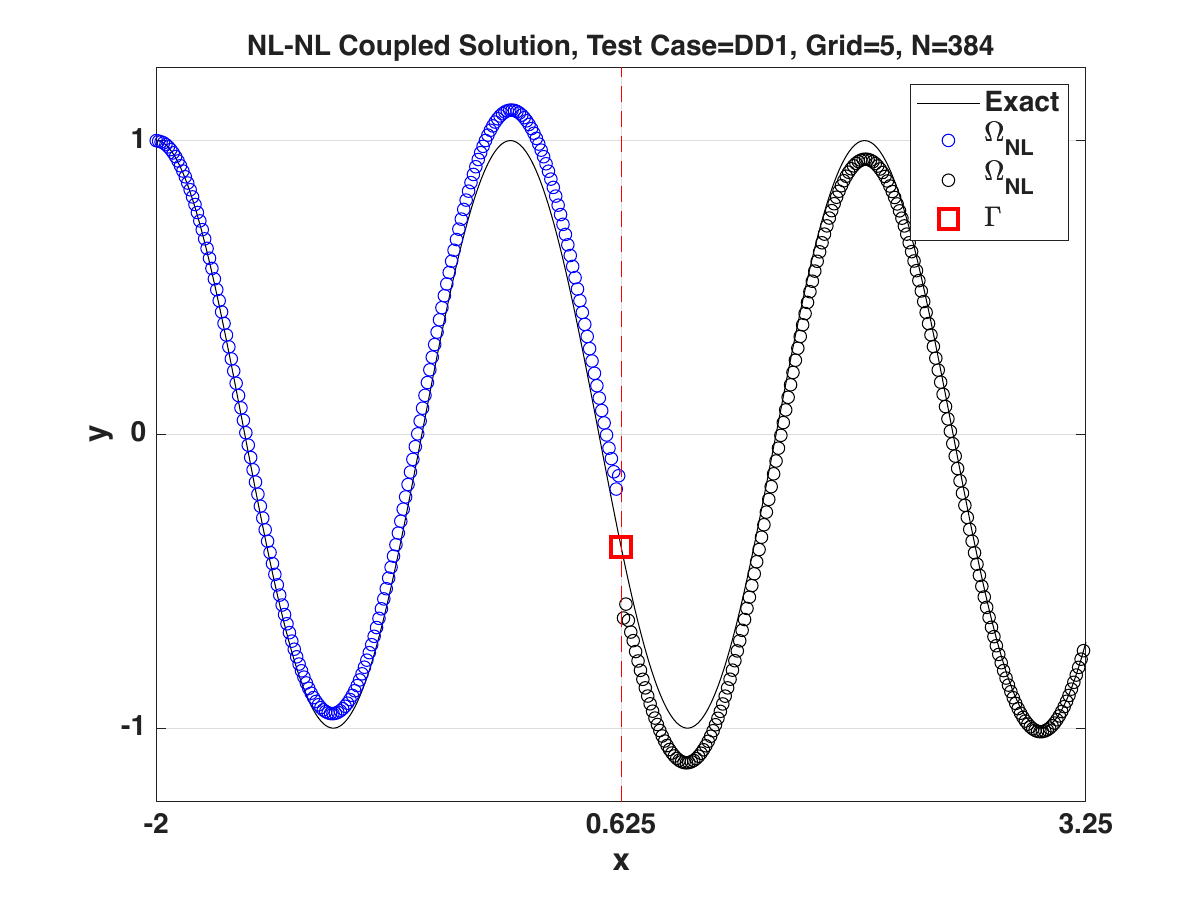}
\end{subfigure}
\begin{subfigure}[b]{0.47\textwidth}
\includegraphics[width=\textwidth]{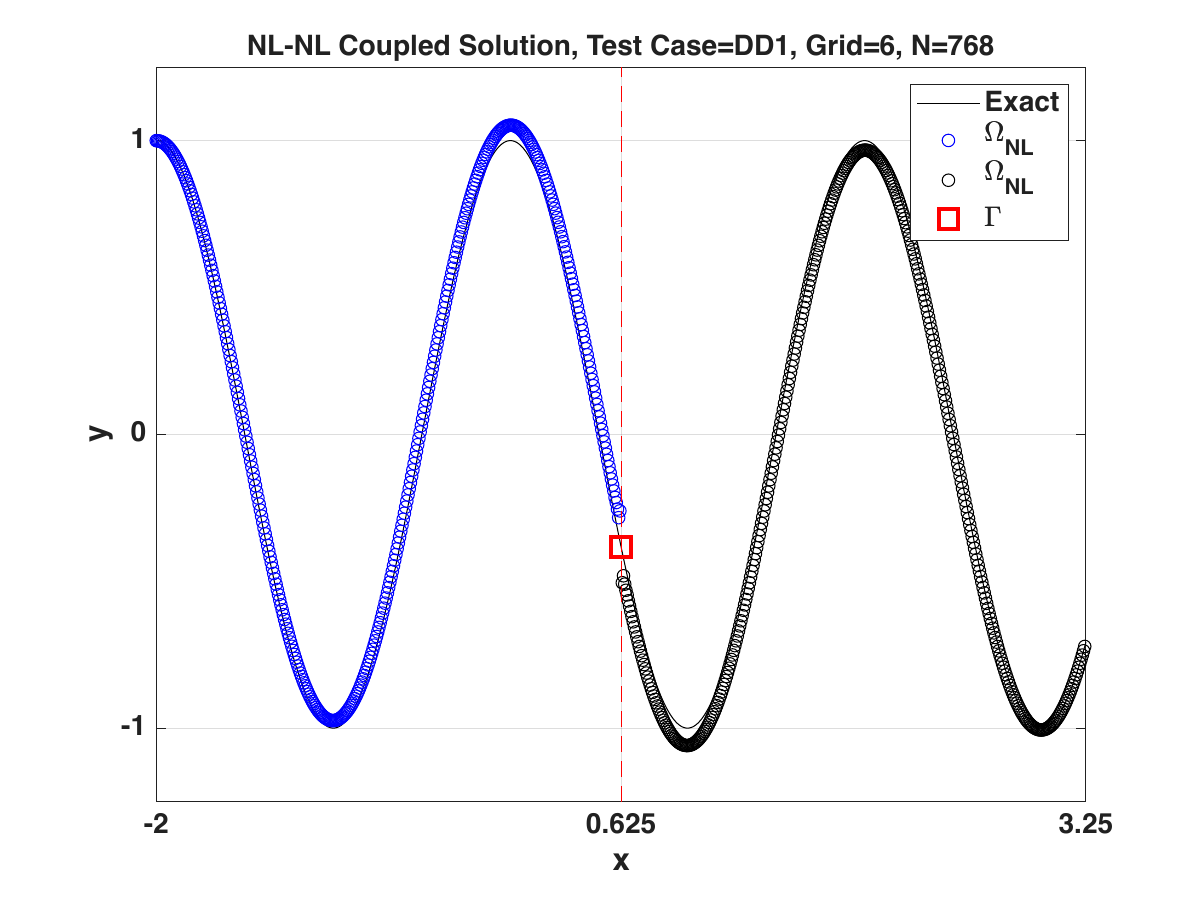}
\end{subfigure}
\begin{subfigure}[b]{0.47\textwidth}
\includegraphics[width=\textwidth]{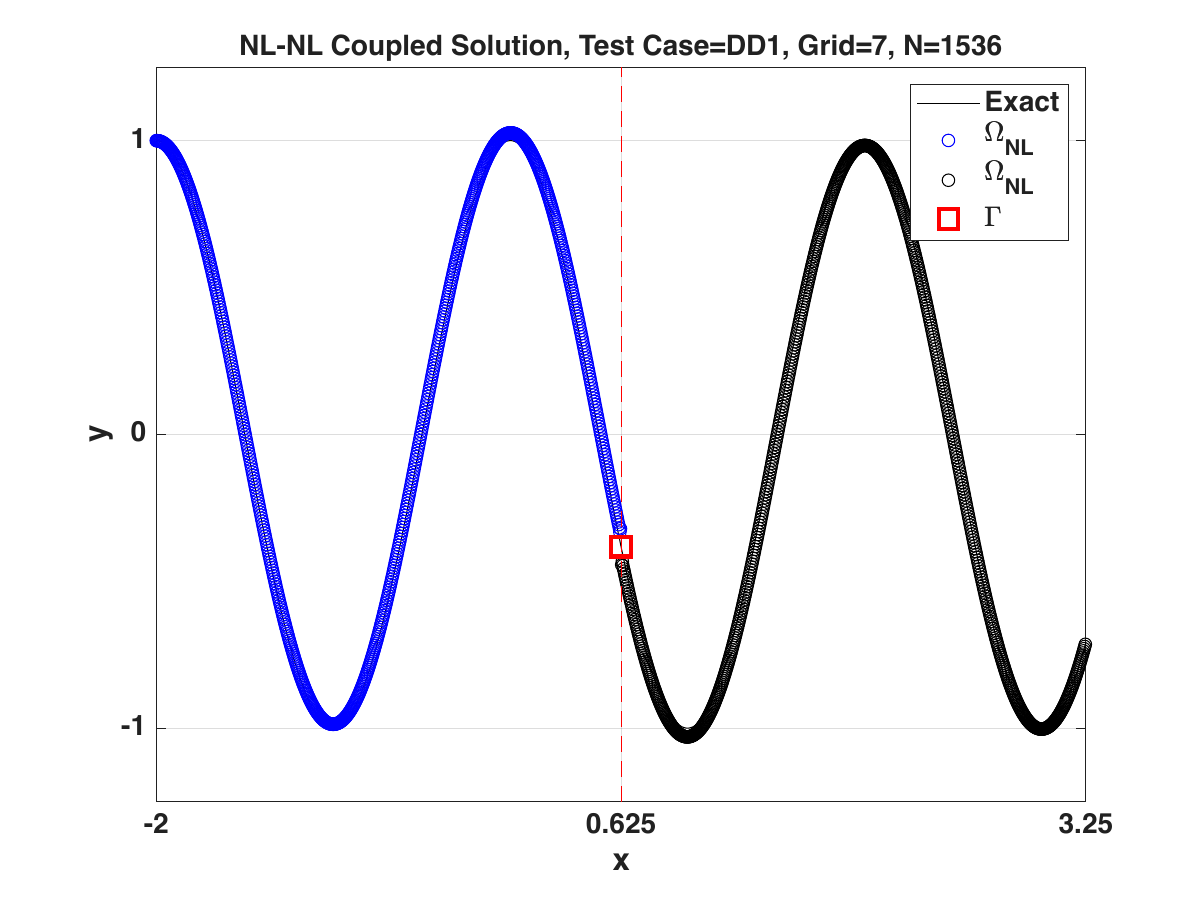}
\end{subfigure}
\begin{subfigure}[b]{0.47\textwidth}
\includegraphics[width=\textwidth]{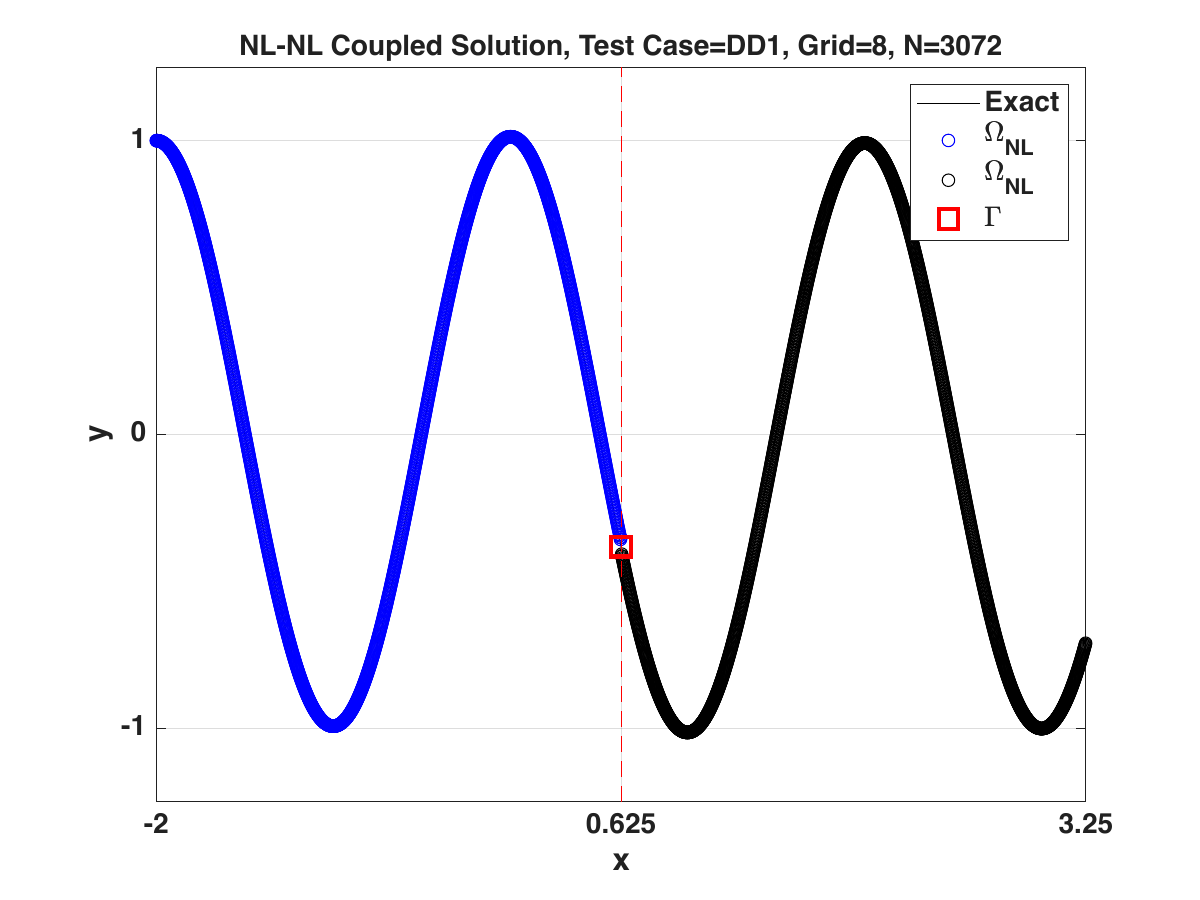}
\end{subfigure}
\caption{NL-NL coupling with $\BC=\DD$}
\begin{subfigure}[b]{0.47\textwidth}
\includegraphics[width=\textwidth]{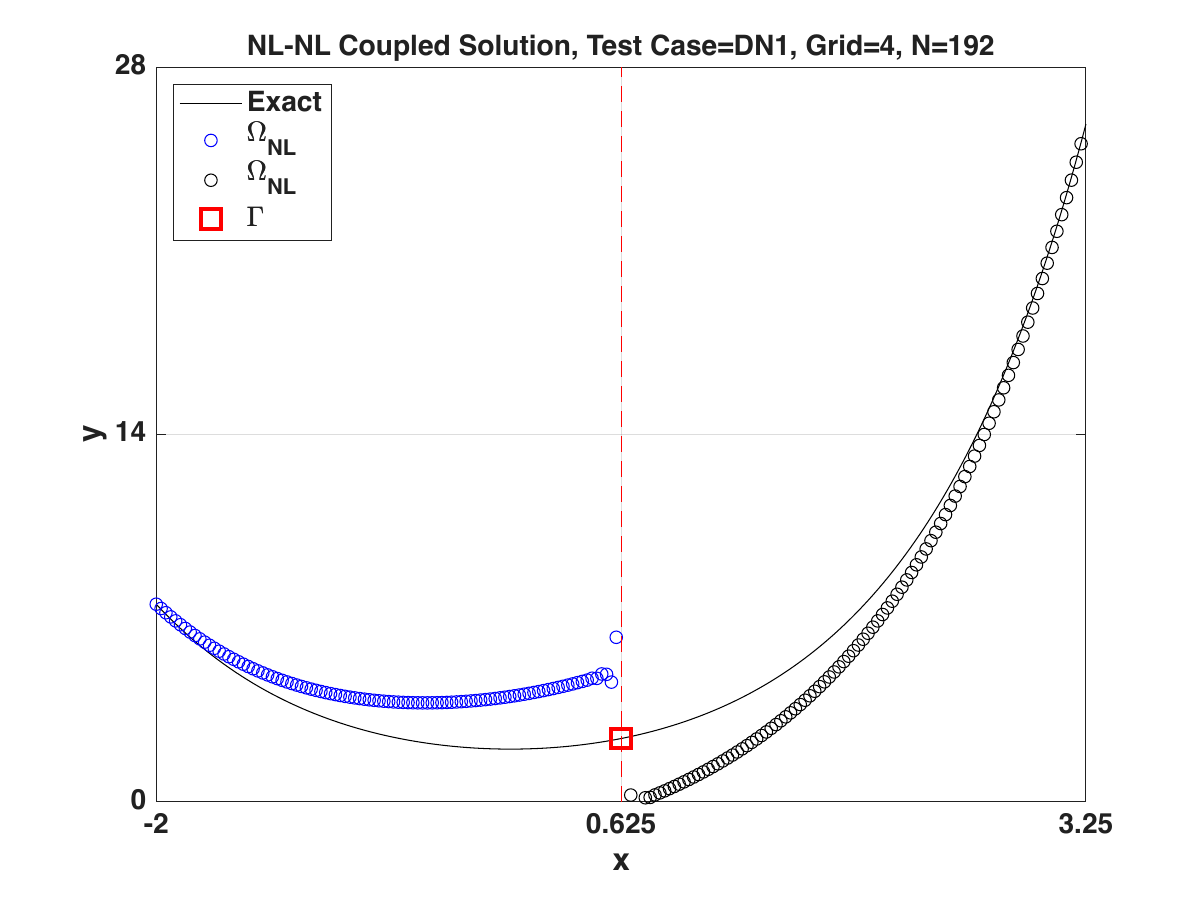}
\end{subfigure}
\begin{subfigure}[b]{0.47\textwidth}
\includegraphics[width=\textwidth]{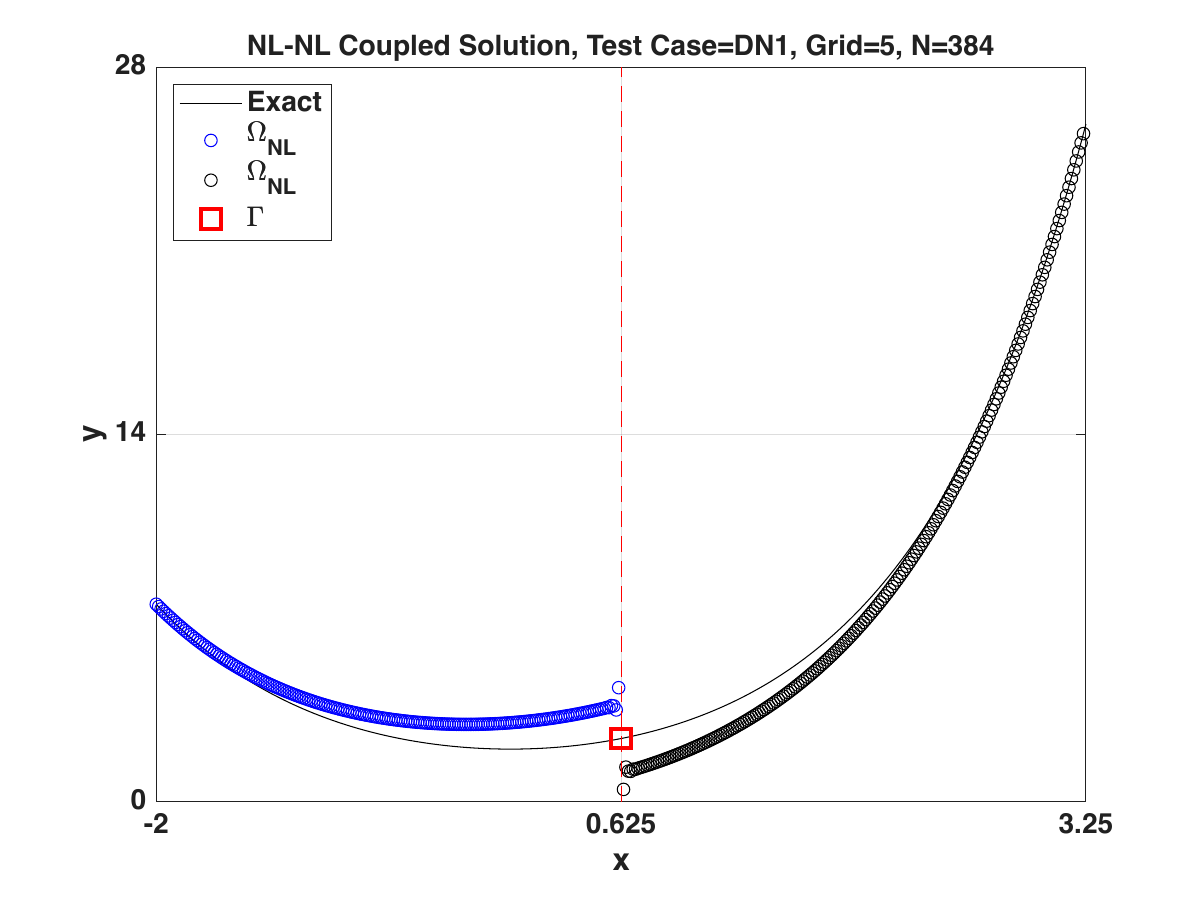}
\end{subfigure}
\begin{subfigure}[b]{0.47\textwidth}
\includegraphics[width=\textwidth]{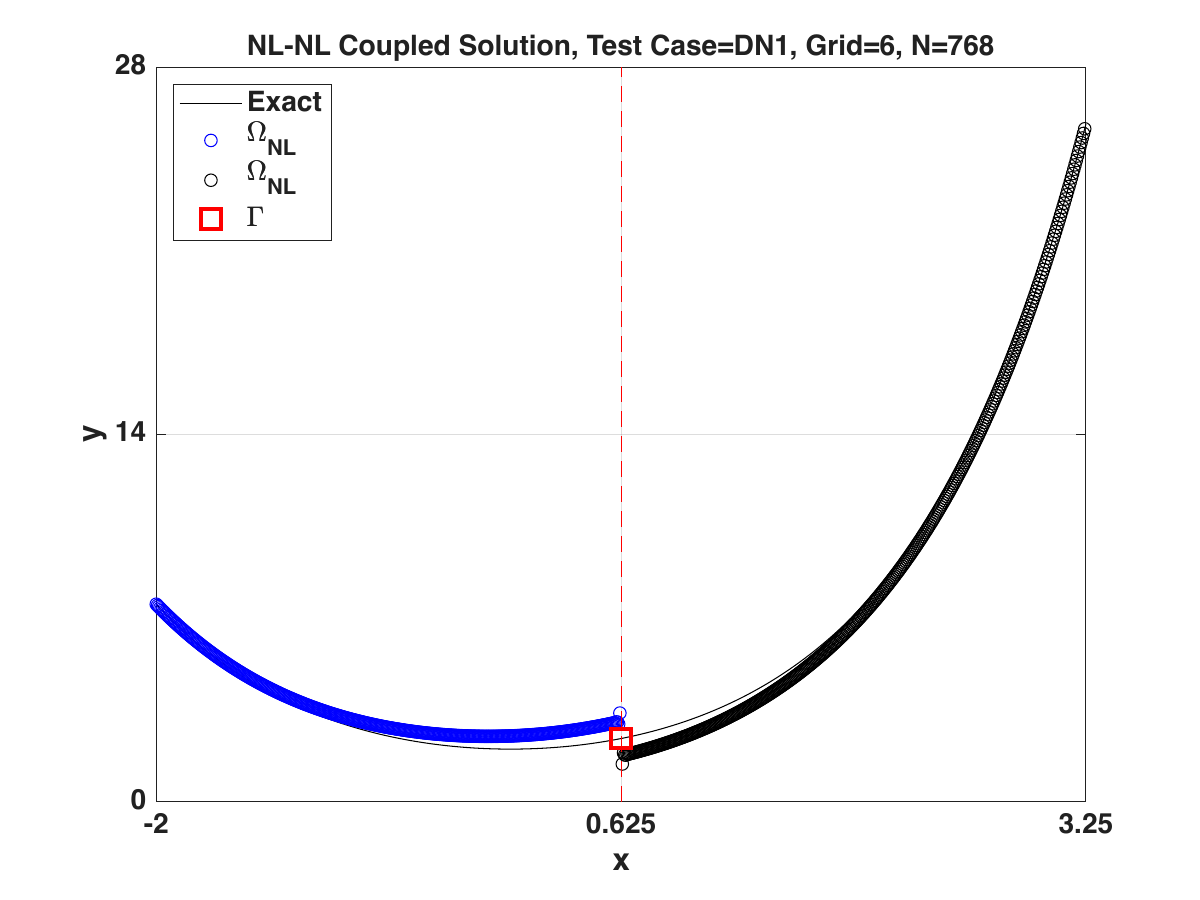}
\end{subfigure}
\begin{subfigure}[b]{0.47\textwidth}
\includegraphics[width=\textwidth]{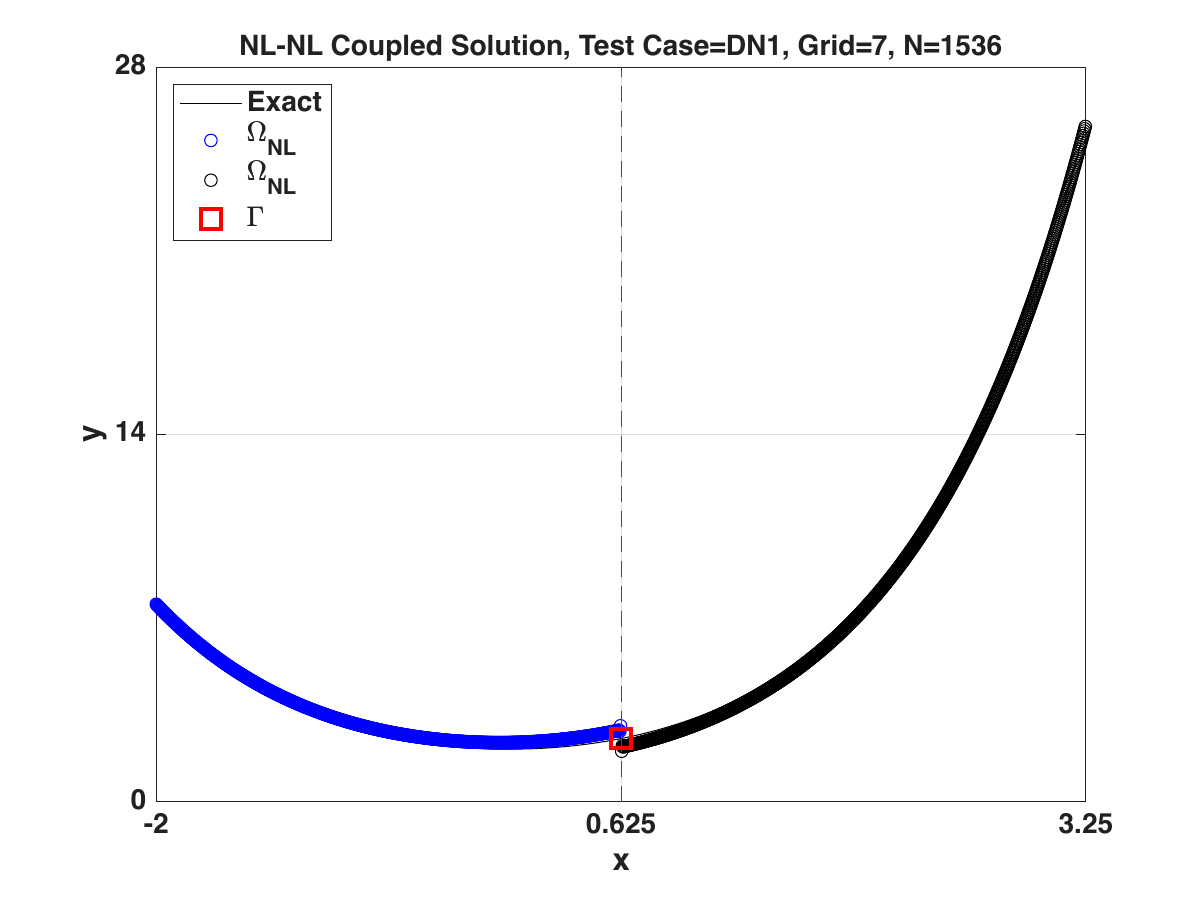}
\end{subfigure}
\caption{NL-NL coupling with $\BC=\DD$ (top 2 rows) $\BC=\DN$ (bottom 2 rows)}
\label{fig:NL_NL_coupling_DD1_DN1}
\end{figure}

\begin{figure}[tb]
\centering
\begin{subfigure}[b]{0.47\textwidth}
\includegraphics[width=\textwidth]{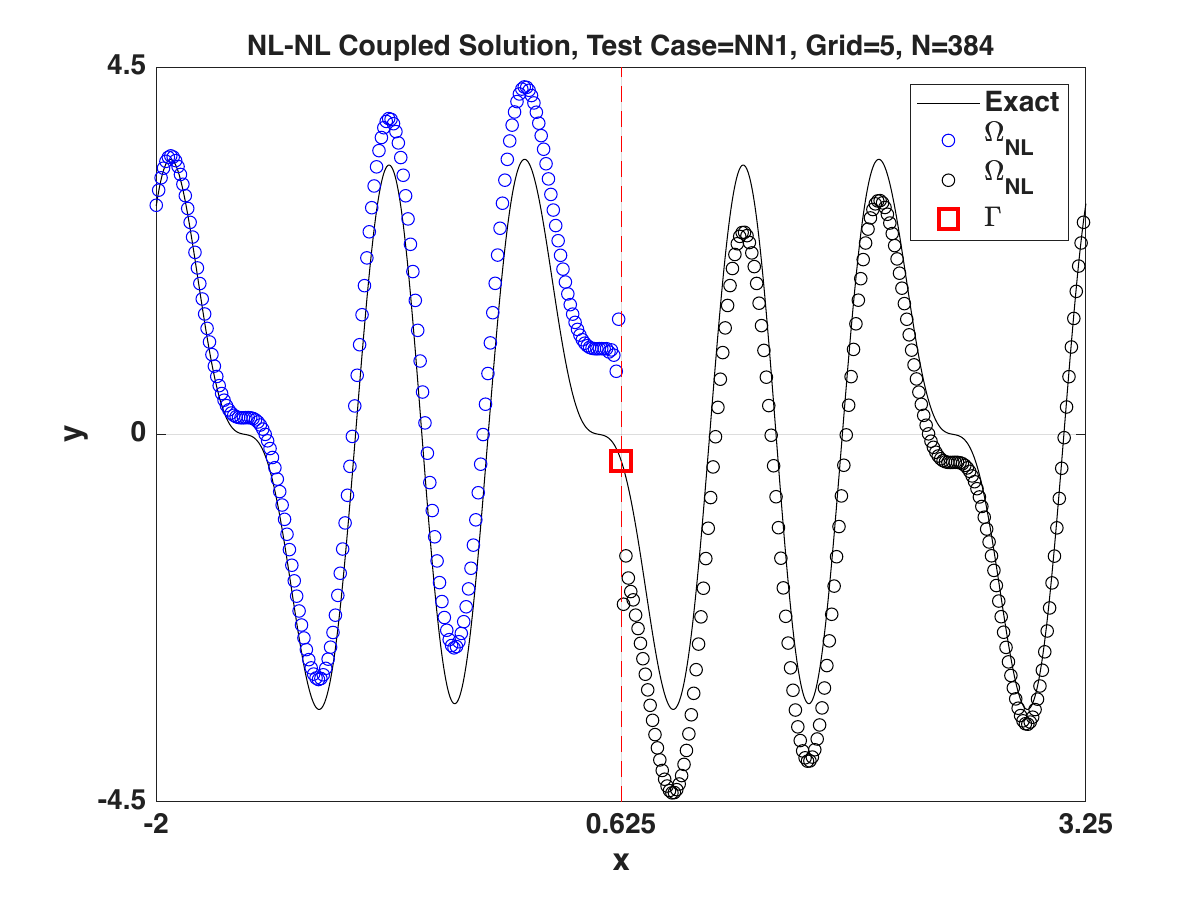}
\end{subfigure}
\begin{subfigure}[b]{0.47\textwidth}
\includegraphics[width=\textwidth]{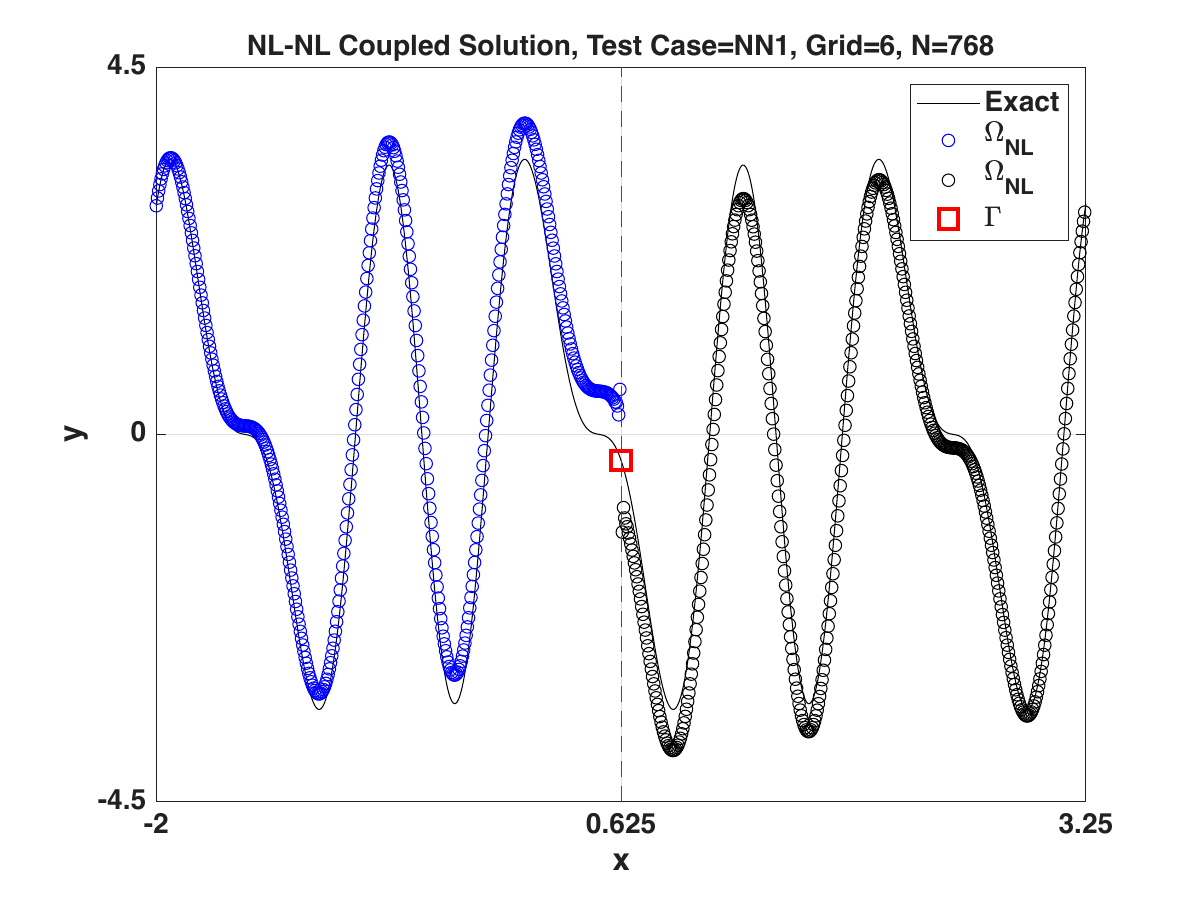}
\end{subfigure}
\begin{subfigure}[b]{0.47\textwidth}
\includegraphics[width=\textwidth]{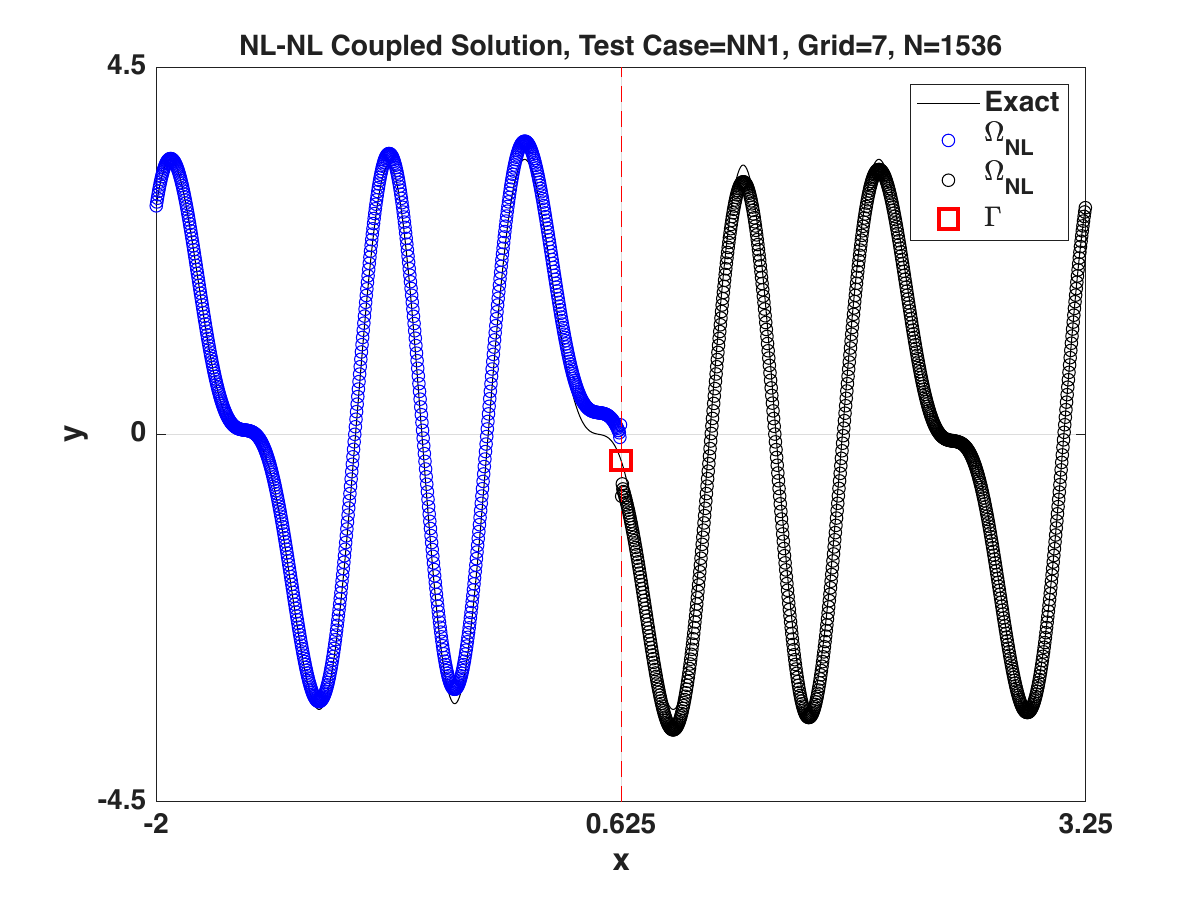}
\end{subfigure}
\begin{subfigure}[b]{0.47\textwidth}
\includegraphics[width=\textwidth]{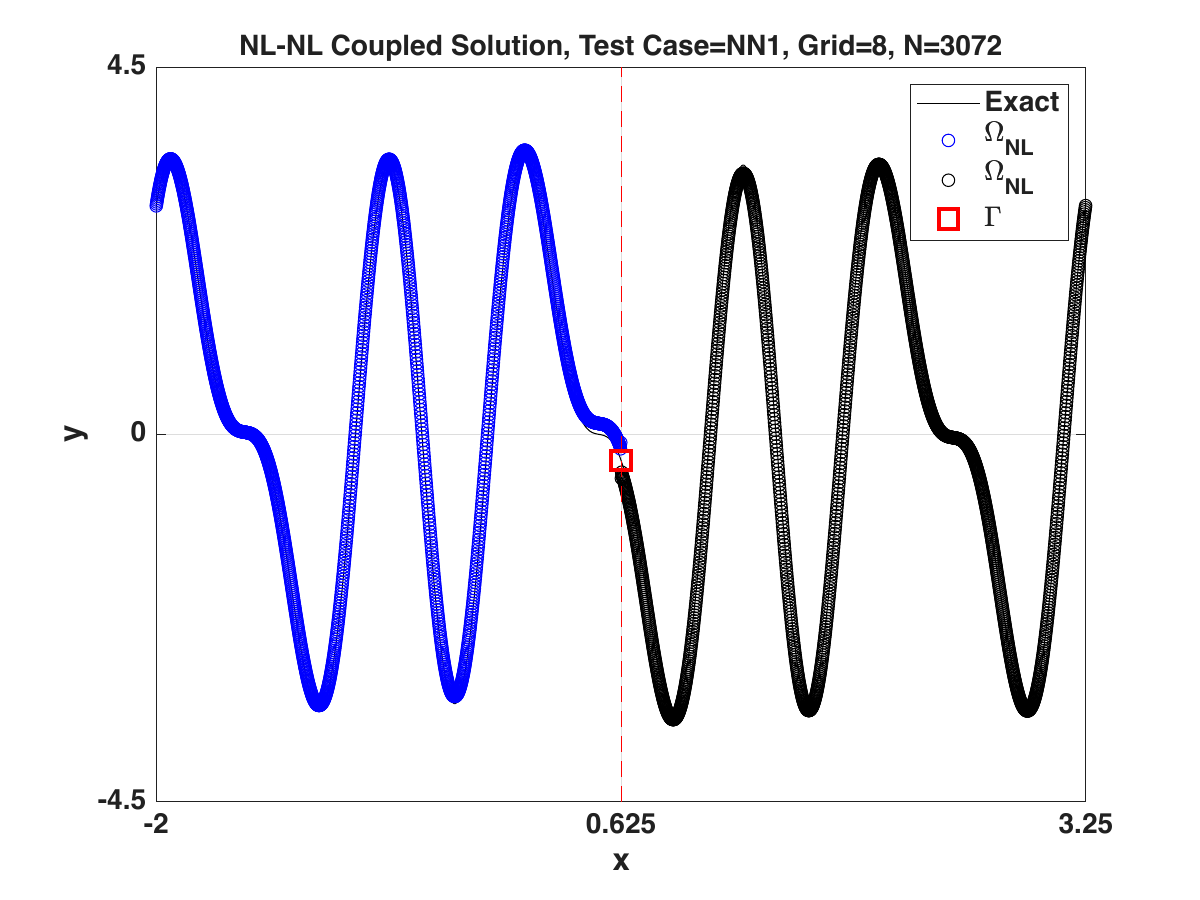}
\end{subfigure}
\centering
\begin{subfigure}[b]{0.47\textwidth}
\includegraphics[width=\textwidth]{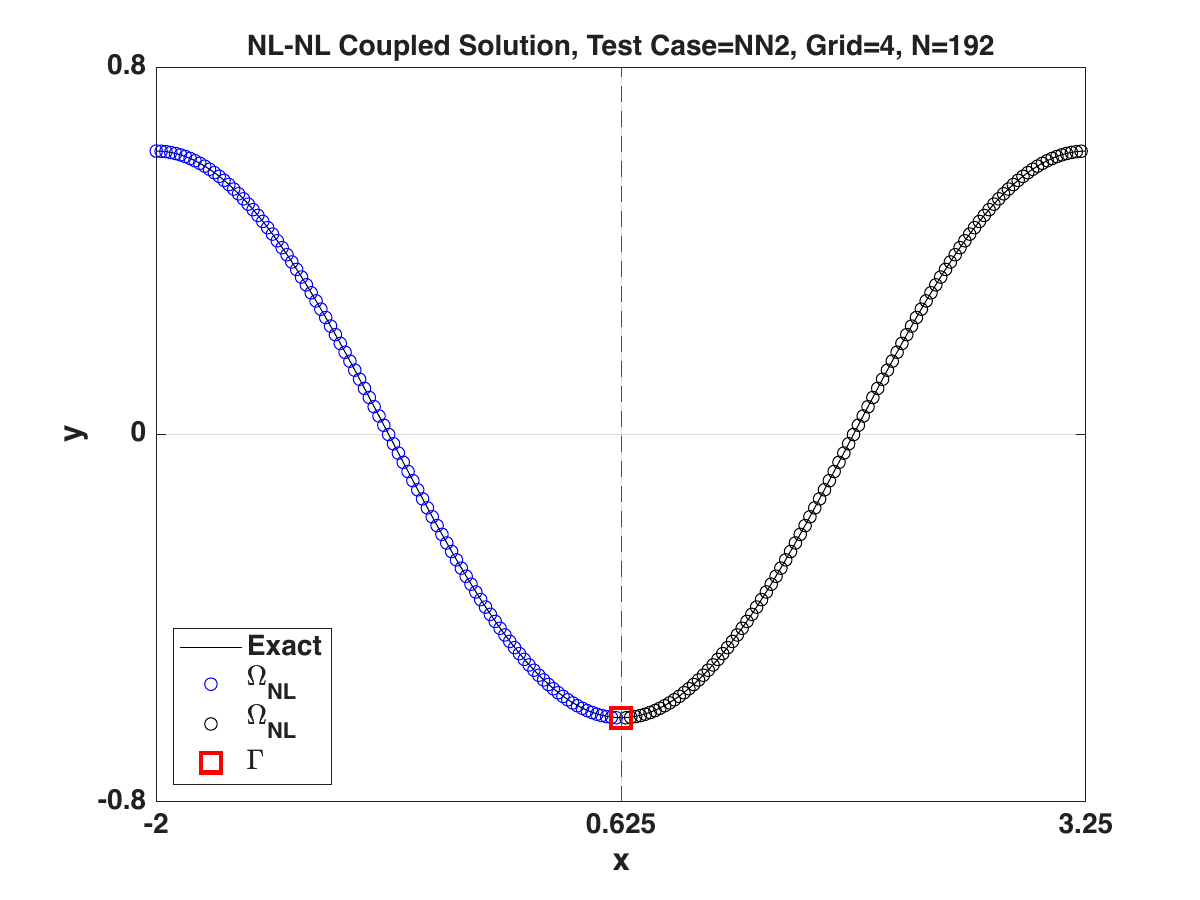}
\end{subfigure}
\begin{subfigure}[b]{0.47\textwidth}
\includegraphics[width=\textwidth]{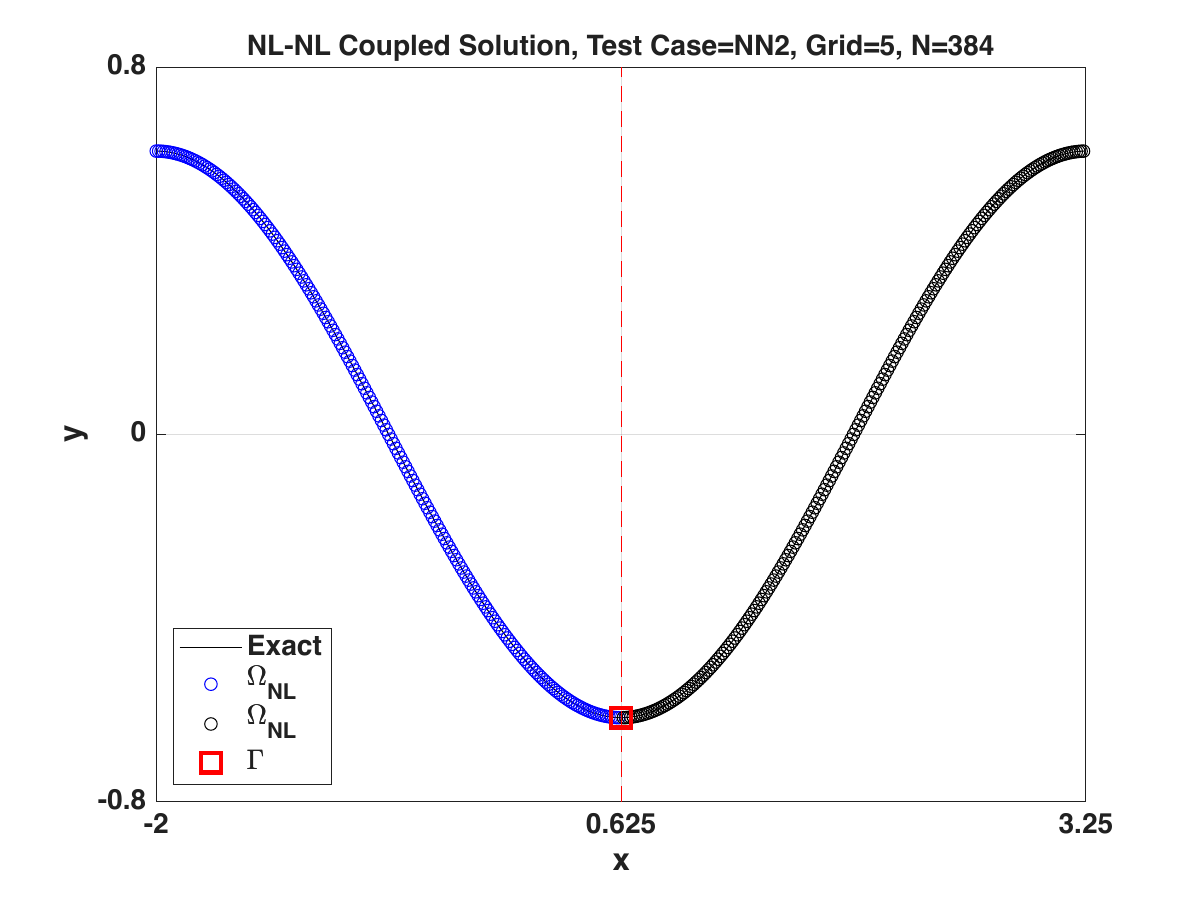}
\end{subfigure}
\begin{subfigure}[b]{0.47\textwidth}
\includegraphics[width=\textwidth]{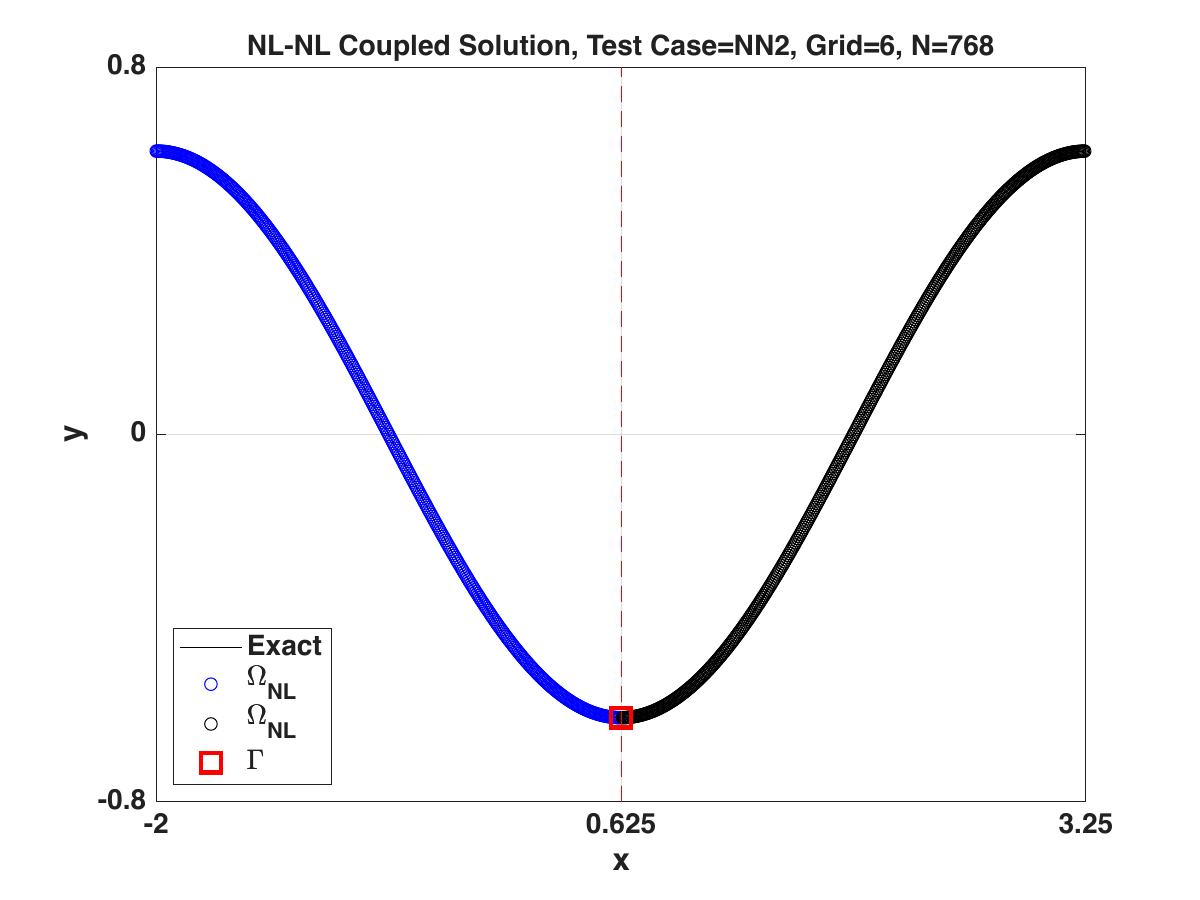}
\end{subfigure}
\begin{subfigure}[b]{0.47\textwidth}
\includegraphics[width=\textwidth]{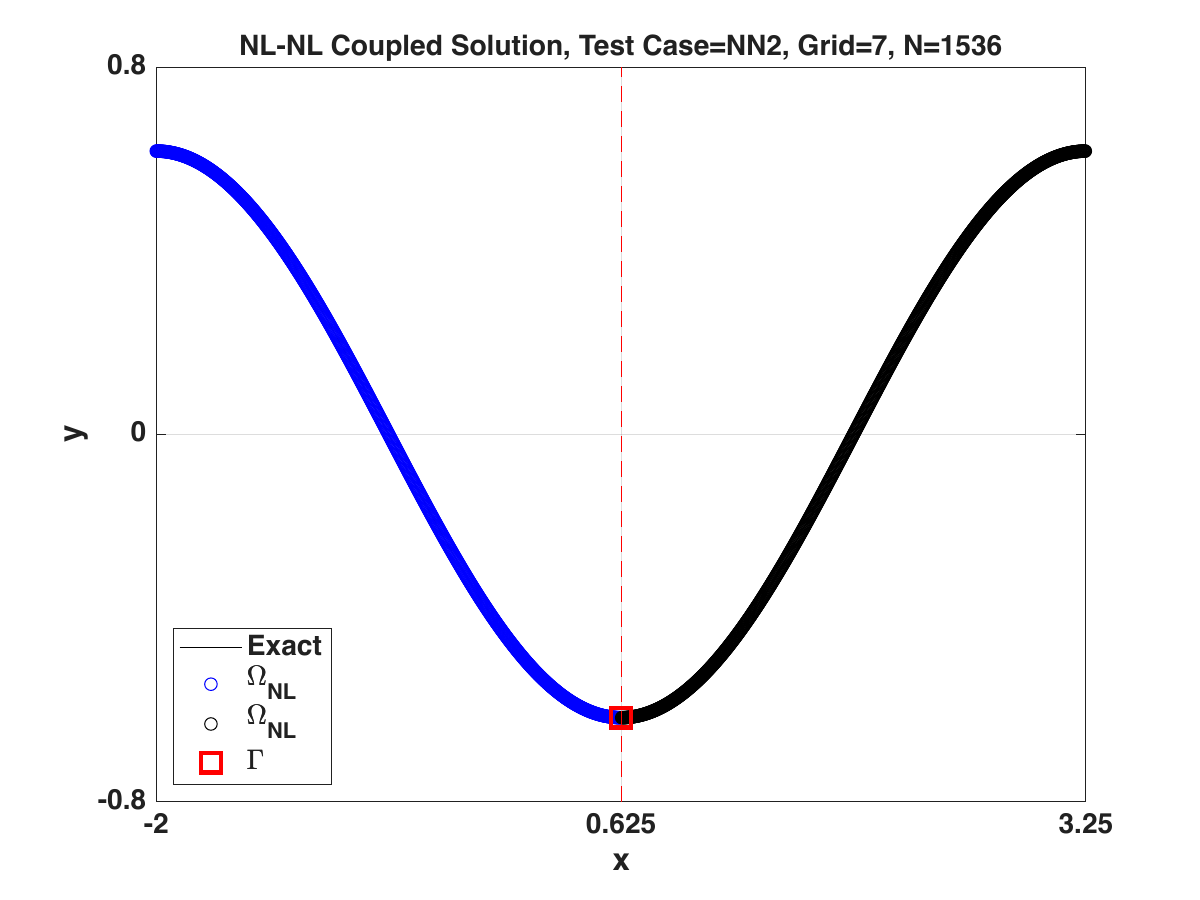}
\end{subfigure}
\caption{NL-NL coupling with $\BC=\NN$, $\NN1$ (top 2 rows) and $\NN2$ (bottom two rows)}
\label{fig:NL_NL_coupling_NN1_NN2}
\end{figure}

In all the experiments, we observe that the numerical solutions at the
interface nodes are very close to the exact solutions compared to
numerical solutions at the collar region.  See how tightly the red
squares capture the exact solution in all the solution plots.  This
indicates that the discrete equation at the interface is sufficiently
accurate and satisfies {\bf (HF1)}.

We report the single domain experiments to determine if the rectification process affects the order of convergence.  We observe that in all test cases, the NL problem enjoys the optimal rate of convergence of $2$.  
In two-subdomain L-L coupled configurations, we also observe the optimal rate of convergence.  For pure Neumann problems, this means that the rank-one update applied to solve the singular system does not deteriorate the convergence rate.  
In two-subdomain NL-NL coupled configurations, the convergence rate decreases to $1$ except the $\NN2$ test case, which we designed on purpose.  The interface falls at a location where the derivatives on either side is zero, thereby, satisfying the Neumann BC compatibility conditions.  Hence, the NL-NL coupled problem does not suffer from an incompatible Neumann condition.  
We cannot solve a mixed two-subdomain (L-NL or NL-N) configuration problem because the rectification process requires to have a NL problem on both ends of the domain.  

In both of the three-subdomain L-NL-L and NL-L-NL cases, where the middle
subdomain becomes floating.  This shows that our coupling method can
handle floating subdomains, which is a major theme in DDMs.  It is
also noteworthy that we can naturally solve problems in the NL-L-NL
configuration with Neumann BCs because a local Neumann condition using
a NL boundary is not straightforward.
In all test cases, the Neumann compatibility conditions are not
satisfied at the interface locations.  Hence, once local and NL
problems are coupled, the convergence rate decreases to $1$.

Since discontinuities are singularities for PDEs, a solution that has
a discontinuity cannot be treated directly with a Laplace operator on
a single domain.  One way to treat discontinuities is to move to
a different problem, known as an interface problem, that contains
multiple Laplace operators on subdomains and interface jump conditions
\cite{carraroWetterauer2016_discont_enrich_xfem}.  On the other hand,
integral operators admit discontinuous solutions. Hence, a
modification to the governing operator is not needed when
discontinuities are introduced to the original solution.

Even though the NL operator is designed to capture discontinuities, an
attentive reader must have noticed that our numerical experiments
contain only continuous solutions.  The reason is that just as the
classical FEM cannot handle discontinuous solutions, the classical
Galerkin projection employed in this study does not handle
discontinuities.  Since the main focus of the present paper is
coupling, we leave the discretization for discontinuous solutions to a
future paper.

\section{Conclusion} \label{sec:conclusion}

For 1D diffusion problems, we constructed a coupling method that is
$\bO(h)$ convergent for an arbitrary solution.  The design of the
method hinges on two hallmark features: The discretized interface
equation becomes a discretized bulk equation and the operators on
either side of the interface should produce the $\nabla u \cdot \bn$
operator.  Our coupling method can handle different configurations
such as L-NL-L and NL-L-NL, where the middle subdomain is floating.

Our coupling method is inspired by the nonoverlapping DDM.  Viewing
the domain decomposition of the local problem as a LtL coupling, we constructed the LtN
coupling.  We have been advocating that local BCs allow for the
transfer of well-established numerical methods developed for local
problems to NL problems.  Our coupling method sets a good example for
this idea.  The success of our method hinges on obtaining an interface
equation that approximates the bulk equation.  We established our
claim with extensive numerical experiments.  The ideas developed in
this study can be generalized to rectangular/box domains in higher
dimensions.

\section*{Declarations} 
\noindent{\bf Conflict of Interest} The authors declared that they have no conflict of interest.

\noindent{\bf Funding Statement} Burak Aksoylu was supported in part by the National Science Foundation DMS 2446826 grant.

\noindent{\bf Author Contributions} All authors contributed to the research and writing of the manuscript.

\section*{Acknowledgments}
The U.S. Department of Energy supported this work through the Los
Alamos National Laboratory. Los Alamos National Laboratory is operated
by Triad National Security, LLC, for the National Nuclear Security
Administration of the U.S. Department of Energy (Contract
No. 89233218CNA000001). Approved for public release with
LA-UR-26-24758.

\appendix
\section{Neumann Conditions and Their Taylor Expansions} \label{sec:Neumann_conds}

Throughout the paper, we utilized the discretized operator that gives
the Neumann BC with the choice of $\delta=h$ because the corresponding
matrix contains relatively small number of entries.  The discretized
operator with the choices $\delta=2h$ and $\delta=3h$ are more
involved and are provided below for completeness:
\begin{eqnarray*}
  \N_{\rright, \delta=2h}^\NL u(e) & = & 
  \begin{bmatrix*} 1 \\ 1 \\ 1 \end{bmatrix*}^\top
    \frac{1}{64h}
\begin{bmatrix*}[r]
  20 & -7 & -12 & -1 & 0 & 0  \\
  -7 & 28 & -8 & -12 & - 1 & 0 \\
-12 & -8 & 40 & -7 & -12 & -1
\end{bmatrix*}
\begin{bmatrix*}[l]
u(e) \\ u(e+h) \\ u(e+2h) \\ u(e+3h) \\ u(e+4h) \\ u(e+5h)
\end{bmatrix*}
\\
  \N_{\lleft, \delta=2h}^\NL u(e) & = & 
  \begin{bmatrix*} 1 \\ 1 \\ 1 \end{bmatrix*}^\top
    \frac{1}{64h}
\begin{bmatrix*}[r]
-1 & -12 & -7 & 40 & -8 & -12 \\
 0 & -1 & -12 & -8 & 28 & -7  \\
 0 & 0 & -1 & -12 & -7 & 20 \\
\end{bmatrix*}
\begin{bmatrix*}[l]
u(e-5h) \\ u(e-4h) \\ u(e-3h) \\ u(e-2h) \\ u(e-h) \\ u(e)
\end{bmatrix*}
\end{eqnarray*}

\begin{eqnarray*}
  \N_{\rright, \delta=3h}^\NL u(e) & = & 
  \begin{bmatrix*} 1 \\ 1 \\ 1 \\ 1 \end{bmatrix*}^\top
    \frac{1}{216h}
\begin{bmatrix*}[r]
  36 & 0 & -23 & -12 & -1 & 0 & 0  & 0 \\
0 & 49  &  -12 & -24 & -12 & - 1 & 0 & 0 \\
-23 & -12 & 71 & 0 & -23 & -12 & -1 & 0 \\
-12 & -24 & 0 & 72 & 0 & -23 & -12 & -1
\end{bmatrix*}
\begin{bmatrix*}[l]
u(e) \\ u(e+h) \\ u(e+2h) \\ u(e+3h) \\ u(e+4h) \\ u(e+5h) \\ u(e+6h)
\end{bmatrix*}
\\
  \N_{\lleft, \delta=3h}^\NL u(e) & = & 
  \begin{bmatrix*} 1 \\ 1 \\ 1 \\ 1 \end{bmatrix*}^\top
    \frac{1}{216h}
\begin{bmatrix*}[r]
-1 & -12 & -23 & 0 & 72 & 0 & -24 & -12 \\
0 & -1 & -12 & -23 & 0 & 71 & -12 & -23  \\
0 & 0 & -1 & -12 & -24 & -12 & 49 & 0 \\
0 & 0 & 0 & -1 & -12 & -23 & 0 & 36
\end{bmatrix*}
\begin{bmatrix*}[l]
u(e-6h) \\ u(e-5h) \\ u(e-4h) \\ u(e-3h) \\ u(e-2h) \\ u(e-h) \\ u(e)
\end{bmatrix*}.
\end{eqnarray*}

Taylor expansions yield
\begin{equation*}
\begin{aligned} 
\N_{\lleft, \delta=2h}^\NL u(e) && = && u'(e) + \bO(h) && = & \quad (\nabla u \cdot \bn)(e) + \bO(h) \\
\N_{\rright, \delta=2h}^\NL u(e) && = && -u'(e) + \bO(h) && = & \quad (\nabla u \cdot \bn)(e) + \bO(h) \\
\N_{\lleft, \delta=3h}^\NL u(e) && = && u'(e) + \bO(h) && = & \quad (\nabla u \cdot \bn)(e) + \bO(h) \\
\N_{\rright, \delta=3h}^\NL u(e) && = && -u'(e) + \bO(h) && =  & \quad (\nabla u \cdot \bn)(e) + \bO(h).
\end{aligned}
\end{equation*}
\fi


\end{document}